\documentclass[11pt]{article}
\usepackage{amsmath}
\usepackage{amsthm}
\usepackage{amsfonts}
\usepackage{makecell,booktabs}
\usepackage{cite}
\usepackage{curves,url}
\usepackage{enumerate}
\usepackage{graphpap}
\usepackage{longtable,booktabs}
\usepackage{makeidx}
\usepackage{mathrsfs}
\usepackage{tcolorbox}
\usepackage{float}
\usepackage[colorinlistoftodos]{todonotes}
\newcommand{\circlesize}{1.3}
\numberwithin{equation}{section}
\makeindex
\newcommand{\Epi}{\mathrm{Epi}}

\def\ov{\overline}

\newcommand{\bz}{\mathbf{z}}

\newcommand{\etal}{\textit{et~al}}
\newcommand{\infb}{in\-her\-ently non-finitely based}
\newcommand{\ib}{identity basis} 

\newcommand{\ibs}{identity bases}
\newcommand{\Ibs}{Identity bases}

\def\ov{\overline}

\def\tk1{{\mbox{\tiny $K_1$}}}

\renewcommand{\wr}{\mathbin{\mathrm{wr}}}

\newcommand{\V}{\operatorname{V}}
\newcommand{\var}{\operatorname{var}}
\newcommand{\vect}{{\vec{\tt v}}}
\newcommand{\dual}{\overleftarrow}

\newcommand{\Atwo}{A_2}
\newcommand{\Az}{A_0}

\newcommand{\Btwo}{B_2}
\newcommand{\Bz}{B_0}
\newcommand{\Ffour}{F_4}
\newcommand{\Gfour}{G_4}
\newcommand{\Nil}{N}
\newcommand{\Niltwo}{\Nil_2}
\newcommand{\Nilthree}{\Nil_3}
\newcommand{\Nilfour}{\Nil_4}
\newcommand{\Niltwoi}{\Nil_2^1}
\newcommand{\Nilthreei}{\Nil_3^1}
\newcommand{\Otwo}{O_2}
\newcommand{\Otwodual}{\dual{O_2}}
\newcommand{\SL}{S\ell_2}
\newcommand{\JI}{J} 
\newcommand{\JIi}{J^1}
\newcommand{\JIdual}{\dual{J}} 
\newcommand{\JIidual}{\dual{J^1}}
\newcommand{\nfbL}{L_3}
\newcommand{\LZ}{LZ_2}
\newcommand{\LZi}{LZ_2^1}
\newcommand{\Ptwo}{P_2}
\newcommand{\Ptwodual}{\dual{P_2}}

\newcommand{\RZ}{RZ_2}
\newcommand{\RZi}{RZ_2^1}
\newcommand{\Ztwo}{\mathbb{Z}_2}
\newcommand{\Zthree}{\mathbb{Z}_3}
\newcommand{\Zfour}{\mathbb{Z}_4}

\newcommand{\bu}{\mathbf{u}}
\newcommand{\bv}{\mathbf{v}}
\newcommand{\bw}{\mathbf{w}}
\newcommand{\ttG}{\texttt{G}}
\newcommand{\ttH}{\texttt{H}}
\newcommand{\ttI}{\texttt{I}}

\newcommand{\bB}{\mathbf{B}}
\newcommand{\bSL}{\mathbf{SL}}
\newcommand{\bT}{\mathbf{0}}
\newcommand{\bLN}{\mathbf{LN}}
\newcommand{\bLZ}{\mathbf{LZ}}
\newcommand{\bRB}{\mathbf{RB}}
\newcommand{\bRN}{\mathbf{RN}}
\newcommand{\bRZ}{\mathbf{RZ}}
\newcommand{\bE}{\mathbf{E}}
\newcommand{\bM}{\mathbf{M}}
\newcommand{\bN}{\mathbf{N}}
\newcommand{\bU}{\mathbf{U}}
\newcommand{\bV}{\mathbf{V}}
\newcommand{\bW}{\mathbf{W}}

\newcommand{\sL}{\mathscr{L}}
\newcommand{\sX}{\mathscr{X}}

\newcommand{\con}{\mathsf{con}}
\newcommand{\head}{\mathsf{h}}
\newcommand{\ini}{\mathsf{ini}}
\newcommand{\occ}{\mathsf{occ}}
\newcommand{\tail}{\mathsf{t}}
\newcommand{\up}{\textup}
\newcommand{\vquad}{\quad}

\newcommand{\GAP}{\textsf{GAP}}
\newcommand{\Smallsemi}{\textsf{Smallsemi}}
\allowdisplaybreaks

\newtheorem{theorem}{Theorem}[section]
\newtheorem{proposition}[theorem]{Proposition}
\newtheorem{lemma}[theorem]{Lemma}
\newtheorem{corollary}[theorem]{Corollary}

\theoremstyle{definition}

\newtheorem{prob}[theorem]{Problem}
\newtheorem{variety}{Variety}
\newtheorem*{variety*}{Variety~\texttt{N}}
\newtheorem*{varexampleA*}{Variety~43}
\newtheorem*{varexampleB*}{Variety~78}

\newcommand{\address}[2][]{
  \expandafter\gdef\expandafter\@address\expandafter{%
    \@address
  \setlength{\@textwidthminusparindent}{\textwidth}
  \setlength{\@saveparindent}{\parindent}
  \addtolength{\@textwidthminusparindent}{-\parindent}
  \footnotesize
  \begin{block}
    \begin{minipage}{\@textwidthminusparindent}%
     \setlength{\parindent}{\@saveparindent}%
      \noindent
      #2\ifthenelse{\equal{#1}{}}{}{\\[2pt]E-mail: #1}
    \end{minipage}
\end{block}}} 

\address[jjaraujo@fc.ul.pt]{J. Ara\'ujo, Departamento de Matem\'atica da Universidade Nova, CMA and CEMAT-CI\^{E}NCIAS, 2829--516 Caparica, Portugal}
\address[joao.p.araujo@tecnico.ulisboa.pt]{J. P. Ara\'ujo, Instituto Superior T\'ecnico, Av. Rovisco Pais, 1, 1049--001 Lisboa, Portugal}
\address[pjc20@st-andrews.ac.uk]{P. J. Cameron, University of St Andrews, KY16 9AJ, UK}
\address[edmond.lee@nova.edu]{E. W. H. Lee, Department of Mathematics, Nova Southeastern University, 3301 College Avenue, Fort Lauderdale, FL~33314, USA}
\address[raminhosj@gmail.com]{J. Raminhos, Universidade Aberta and CEMAT-CI\^{E}NCIAS, Universidade de Lisboa, Campo Grande, C6, 1749-016 Lisboa, Portugal}

\title{A Survey on Varieties Generated by Small Semigroups and a Companion Website}

\author{
Jo\~ao Ara\'{u}jo\thanks{Supported by the Funda\c{c}\~{a}o para a Ci\^{e}ncia e Tecnologia (Portuguese Foundation for Science and Technology) through the project CEMAT-CI\^{E}NCIAS UID/Multi/ 04621/2013, and through project Hilbert's 24th problem� PTDC/MHC-FIL/2583/2014.}
\and
Jo\~ao Pedro Ara\'{u}jo
\and
Peter J. Cameron\thanks{Supported by the Funda\c{c}\~{a}o para a Ci\^{e}ncia e Tecnologia (Portuguese Foundation for Science and Technology) through the project CEMAT-CI\^{E}NCIAS UID/Multi/ 04621/2013}
\and
Edmond W. H. Lee 
\and
Jorge Raminhos
}

\begin{document}
\maketitle

\setcounter{tocdepth}{2}

\tableofcontents


\begin{abstract}
{\bf Abstract} The aim of this paper is to provide an atlas of identity bases for varieties generated by small semigroups and groups.
To help the working mathematician easily find information, we provide a companion website that runs in the background automated reasoning tools, finite model builders, and GAP, so that the user has an automatic \textit{intelligent} guide on the literature.

This paper is mainly a survey of what is known about identity bases for semigroups or groups of small orders, and we also mend some gaps left unresolved by previous authors.
For instance, we provide the first complete and justified list of identity bases for the varieties generated by a semigroup of order up to~$4$, and the website contains the list of varieties generated by a semigroup of order up to~$5$.

 The website also provides identity bases for several types of semigroups or groups, such as bands, commutative groups, and metabelian groups.
 On the inherently non-finitely based finite semigroups side, the website can decide if a given finite semigroup possesses this property or not.
 We provide some other functionalities such as a tool that outputs the multiplication table of a semigroup given by a $C$-presentation, where~$C$ is any class of algebras defined by a set of first order formulas.


The companion website can be found here

\url{http://sgv.pythonanywhere.com}

Please send any comments/suggestions to \url{jj.araujo@fct.unl.pt}
\end{abstract}








\section{Introduction}

We assume familiarity with the general theory of varieties, semigroups, and groups.
For general references, we suggest the monographs of Almeida~\cite{Alm94}, Burris and Sankappanavar~\cite{BS81}, Howie~\cite{How95}, McKenzie {\etal}.~\cite{MMT87}, and H. Neumann~\cite{NeuH67}.

Studying the lattice of varieties of semigroups is an old area of research, but given its complexity, this topic remains very active up to the present and certainly will continue into the foreseeable future.
There are several very good surveys, such as Evans~\cite{Eva71}, Shevrin {\etal}.~\cite{SVV09}, and Vernikov~\cite{Ver15}, that allow the reader to become familiar with the main results and problems; our goal is different. {\color{black} We aim at a living survey powered by a companion computational tool that helps the working mathematician finding either new results or locate old ones in the literature.}

As an illustration, suppose that we are researchers in some area of mathematics who, for some reason, need to investigate semigroups satisfying the implication \[ xy \approx yx \,\ \Longrightarrow \,\ x \approx y, \] objects we might call \textit{anti-commutative semigroups}.
To understand their properties, we could use {\GAP} to find some small models, as for example, the semigroup~$U_1$ in Table~\ref{Tab: U1 U2}.
At a certain point, we observe that all elements of~$U_1$ are idempotents---such a semigroup satisfies the idempotency identity $x^2 \approx x$ and is commonly called a \textit{band}---and searching for varieties of bands we find a reference \cite{Fen71} that contains the lattice $\sL(\bB)$ of varieties of bands, as shown in Figure~\ref{F: bands}.

\begin{table}[ht]  \centering 
\begin{tabular}[t]{c|cccc}
$U_1$ & 1 & 2 & 3 & 4 \\ \hline
    1 & 1 & 1 & 3 & 3 \\
    2 & 2 & 2 & 4 & 4 \\
    3 & 1 & 1 & 3 & 3 \\
    4 & 2 & 2 & 4 & 4
\end{tabular}
\quad
\begin{tabular}[t]{c|ccccc}
$U_2$ & 1 & 2 & 3 & 4 & 5 \\ \hline
    1 & 1 & 1 & 1 & 1 & 5 \\
    2 & 1 & 2 & 3 & 4 & 5 \\
    3 & 1 & 3 & 4 & 2 & 5 \\
    4 & 1 & 4 & 2 & 3 & 5 \\
    5 & 1 & 5 & 5 & 5 & 5
\end{tabular}
\caption{The semigroups $U_1$ and $U_2$}
\label{Tab: U1 U2}
\end{table}

\begin{figure}[hbt]
\setlength{\unitlength}{3mm}
\begin{center}
\begin{picture}(34,37)
\put(9,36){\makebox(0,0){\mbox{$\bB$}}}
\put(5,33.5){\makebox(0,0){\vdots}}
\put(13,33.5){\makebox(0,0){\vdots}}
\put(5,30.5){\makebox(0,0){\mbox{\scriptsize $[\ov{\ttG_5}\ttG_4 \approx \ov{\ttH_5}\ttI_4]_\bB$}}}
\put(13,30.5){\makebox(0,0){\mbox{\scriptsize $[\ov{\ttG_4}\ttG_5 \approx \ov{\ttI_4}\ttH_5]_\bB$}}}
\put(5,24.5){\makebox(0,0){\mbox{\scriptsize $[\ov{\ttG_4}\ttG_4 \approx \ov{\ttI_4}\ttH_4]_\bB$}}}
\put(13,24.5){\makebox(0,0){\mbox{\scriptsize $[\ov{\ttG_4}\ttG_4 \approx \ov{\ttH_4}\ttI_4]_\bB$}}}
\put(5,18.5){\makebox(0,0){\mbox{\scriptsize $[\ov{\ttG_4}\ttG_3 \approx \ov{\ttH_4}\ttI_3]_\bB$}}}
\put(13,18.5){\makebox(0,0){\mbox{\scriptsize $[\ov{\ttG_3}\ttG_4 \approx \ov{\ttI_3}\ttH_4]_\bB$}}}
\put(5,12.5){\makebox(0,0){\mbox{\scriptsize $[\ov{\ttG_3}\ttG_3 \approx \ov{\ttI_3}\ttH_3]_\bB$}}}
\put(13,12.5){\makebox(0,0){\mbox{\scriptsize $[\ov{\ttG_3}\ttG_3 \approx \ov{\ttH_3}\ttI_3]_\bB$}}}
\put(1,27.5){\makebox(0,0){\mbox{\scriptsize $[\ttG_4 \approx \ttH_4]_\bB$}}}
\put(9,27.5){\makebox(0,0){\mbox{\scriptsize $[\ov{\ttG_4}\ttG_4 \approx \ov{\ttI_4}\ttI_4]_\bB$}}}
\put(17,27.5){\makebox(0,0){\mbox{\scriptsize $[\ov{\ttG_4} \approx \ov{\ttH_4}]_\bB$}}}
\put(1,21.5){\makebox(0,0){\mbox{\scriptsize $[\ttG_3 \approx \ttI_3]_\bB$}}}
\put(9,21.5){\makebox(0,0){\mbox{\scriptsize $[\ov{\ttG_4}\ttG_4 \approx \ov{\ttH_4}\ttH_4]_\bB$}}}
\put(17,21.5){\makebox(0,0){\mbox{\scriptsize $[\ov{\ttG_3} \approx \ov{\ttI_3}]_\bB$}}}
\put(1,15.5){\makebox(0,0){\mbox{\scriptsize $[\ttG_3 \approx \ttH_3]_\bB$}}}
\put(9,15.5){\makebox(0,0){\mbox{\scriptsize $[\ov{\ttG_3}\ttG_3 \approx \ov{\ttI_3}\ttI_3]_\bB$}}}
\put(17,15.5){\makebox(0,0){\mbox{\scriptsize $[\ov{\ttG_3} \approx \ov{\ttH_3}]_\bB$}}}
\put(1,9.5){\makebox(0,0){\mbox{\scriptsize $[\ttG_2 \approx \ttI_2]_\bB$}}}
\put(9,9.5){\makebox(0,0){$\bN$}}
\put(17,9.5){\makebox(0,0){\mbox{\scriptsize $[\ov{\ttG_2} \approx \ov{\ttI_2}]_\bB$}}}
\put(5,6.5){\makebox(0,0){$\bLN$}}
\put(9,6.5){\makebox(0,0){$\bRB$}}
\put(13,6.5){\makebox(0,0){$\bRN$}}
\put(5,3.5){\makebox(0,0){$\bLZ$}}
\put(9,3.5){\makebox(0,0){$\bSL$}}
\put(13,3.5){\makebox(0,0){$\bRZ$}}
\put(9,0.5){\makebox(0,0){$\bT$}}
\put(24,17){\makebox(0,0)[r]{$\bB$}} \put(24.3,17){\makebox(0,0)[l]{$= [x^2 \approx x]$}}
\put(24,15){\makebox(0,0)[r]{$\bN$}} \put(24.3,15){\makebox(0,0)[l]{$= [xyzx \approx xzyx]_\bB$}}
\put(24,13){\makebox(0,0)[r]{$\bLN$}} \put(24.3,13){\makebox(0,0)[l]{$= [xyz \approx xzy]_\bB$}}
\put(24,11){\makebox(0,0)[r]{$\bRN$}} \put(24.3,11){\makebox(0,0)[l]{$= [xyz \approx zyx]_\bB$}}
\put(24,9){\makebox(0,0)[r]{$\bSL$}} \put(24.3,9){\makebox(0,0)[l]{$= [xy \approx yx]_\bB$}}
\put(24,7){\makebox(0,0)[r]{$\bRB$}} \put(24.3,7){\makebox(0,0)[l]{$= [xyx \approx x]$}}
\put(24,5){\makebox(0,0)[r]{$\bLZ$}} \put(24.3,5){\makebox(0,0)[l]{$= [xy \approx x]$}}
\put(24,3){\makebox(0,0)[r]{$\bRZ$}} \put(24.3,3){\makebox(0,0)[l]{$= [xy \approx y]$}}
\put(24,1){\makebox(0,0)[r]{$\bT$}} \put(24.3,1){\makebox(0,0)[l]{$= [x \approx y]$}}
\put(2,28){\line(1,1){2}}
\put(6,30){\line(1,-1){2}}
\put(10,28){\line(1,1){2}}
\put(14,30){\line(1,-1){2}}
\put(2,22){\line(1,1){2}}
\put(6,24){\line(1,-1){2}}
\put(10,22){\line(1,1){2}}
\put(14,24){\line(1,-1){2}}
\put(2,16){\line(1,1){2}}
\put(6,18){\line(1,-1){2}}
\put(10,16){\line(1,1){2}}
\put(14,18){\line(1,-1){2}}
\put(2,10){\line(1,1){2}}
\put(6,12){\line(1,-1){2}}
\put(10,10){\line(1,1){2}}
\put(14,12){\line(1,-1){2}}
\put(2,27){\line(1,-1){2}}
\put(6,25){\line(1,1){2}}
\put(10,27){\line(1,-1){2}}
\put(14,25){\line(1,1){2}}
\put(2,21){\line(1,-1){2}}
\put(6,19){\line(1,1){2}}
\put(10,21){\line(1,-1){2}}
\put(14,19){\line(1,1){2}}
\put(2,15){\line(1,-1){2}}
\put(6,13){\line(1,1){2}}
\put(10,15){\line(1,-1){2}}
\put(14,13){\line(1,1){2}}
\put(2,9){\line(1,-1){2}}
\put(6,7){\line(1,1){2}}
\put(10,9){\line(1,-1){2}}
\put(14,7){\line(1,1){2}}
\put(6,4){\line(1,1){2}}
\put(10,4){\line(1,1){2}}
\put(6,6){\line(1,-1){2}}
\put(10,6){\line(1,-1){2}}
\put(5,4){\line(0,1){2}} \put(9,7){\line(0,1){2}}
\put(13,4){\line(0,1){2}}
\put(9,1){\line(0,1){2}}
\put(6,3){\line(1,-1){2}}
\put(10,1){\line(1,1){2}}
\put(21,0){\line(0,1){18}}
\end{picture}
\caption{
The lattice $\sL(\bB)$ of varieties of bands, where $[\bu \approx \bv]_\bB = \bB \cap [\bu \approx \bv]$
and details on the words $\ttG_n, \ttH_n, \ttI_n, \ov{\ttG_n}, \ov{\ttH_n}, \ov{\ttI_n}$ are given in Subsection~\ref{subsec: L(B)}.
}
\label{F: bands}
\end{center}
\end{figure}

Again we could use {\GAP} to see that our semigroup~$U_1$ violates the identities $xy \approx x$ and $xy \approx y$ but satisfies the identity $xyx \approx x$.
Therefore the variety $\var\{U_1\}$ generated by~$U_1$ is contained in the variety of bands defined by the identity $xyx \approx x$---the variety $\bRB$ of \textit{rectangular bands}---but is excluded from its two maximal subvarieties~$\bLZ$ and~$\bRZ$, whence $\var\{U_1\} = \bRB$.
Now an easy exercise shows that a semigroup is anti-commutative if and only if it satisfies the identity $xyx \approx x$, and from here we get access to an enormous amount of literature on our original object.
The key steps in the above process were the observation that~$U_1$ is a band and the complete knowledge of the lattice of varieties of bands.

Now suppose that we are working with a different theory and our test semigroup is~$U_2$ in Table~\ref{Tab: U1 U2}.
Since~$U_2$ is not a band, there is no general lattice, similar to Figure~\ref{F: bands}, that allows us to repeat what we have done with~$U_1$.
It turns out that the variety $\var\{U_2\}$ is defined by the identities $\{ x^4 \approx x, \, xyx \approx yx^2 \}$, but only a substantial search would allow us to locate a reference \cite[Proposition 3.16]{Tis07}.

In general, given a semigroup~$S$ of order up to~$6$, there is still a good chance that information on the variety $\var\{S\}$ and its subvarieties can be found in the literature,
since such varieties have received much attention over the years \cite{ELL10,Lee04,Lee06,Lee07a,Lee07b,Lee08a,Lee08b,Lee09,Lee10,Lee11,Lee12,Lee18,LV07,LV11,Mel72,Tis07,Vol05,ZL09},
especially in the investigation of the finite basis problem for small semigroups \cite{Edm77,Edm80,Lee13,LL11,LZ15,Per69,Sap87a,Sap87b,Tis80,Tra81,Tra87,Tra91,Tra94,ZL11}.
The first goal of this survey is to provide such information, but we go far beyond that.
The overall aim is to provide a survey on {\ibs} defining varieties generated by finite semigroups and set up a companion website, running {\GAP} and automated reasoning tools in the background, that will be continuously updated to better assist the working mathematician.
Resources provided by the present survey and the website so far are described as follows.
\begin{enumerate}
\item {\Ibs}, and corresponding proofs or references, for all varieties generated by a semigroup of order up to~$4$.
This survey is the first source providing this information.

\item {\Ibs} for many varieties generated by semigroups of higher orders, including all semigroups of order~$5$, the proofs of which will be disseminated elsewhere.

\begin{figure}[hbt]
  \includegraphics[width=\textwidth]{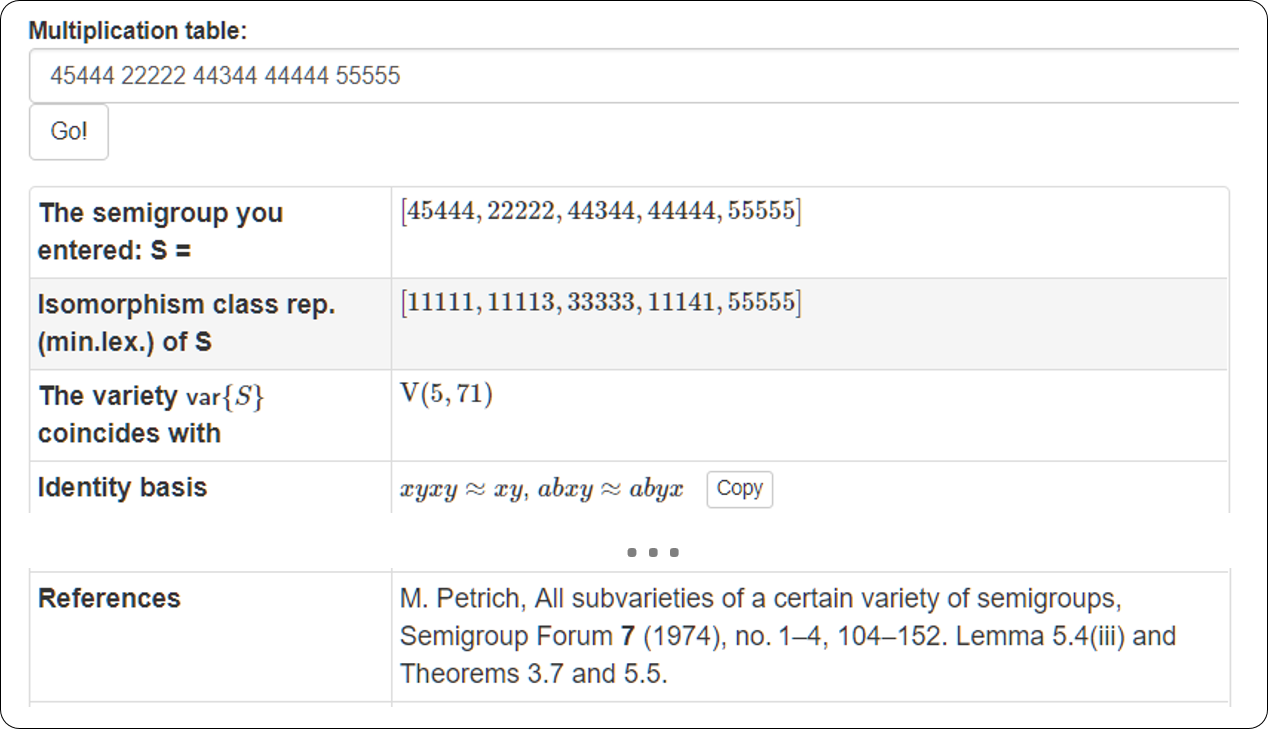}
  \caption{Companion website: example of a reference given for an order 5 semigroup}
    \label{F: 5ref}
\end{figure}

\item {\color{black}{\Ibs} for all varieties generated by a group that has abelian normal and factor subgroups $N$ and $G/N$ such that $\gcd(|N|,|G/N|)=1$.}%



\item For some classes of semigroups, including bands and some classes of groups, the website finds {\ibs} for varieties generated by arbitrarily large finite models; see Figure~\ref{F: Bands} on page~\pageref{F: Bands}.

\begin{figure}
  \includegraphics[width=\textwidth]{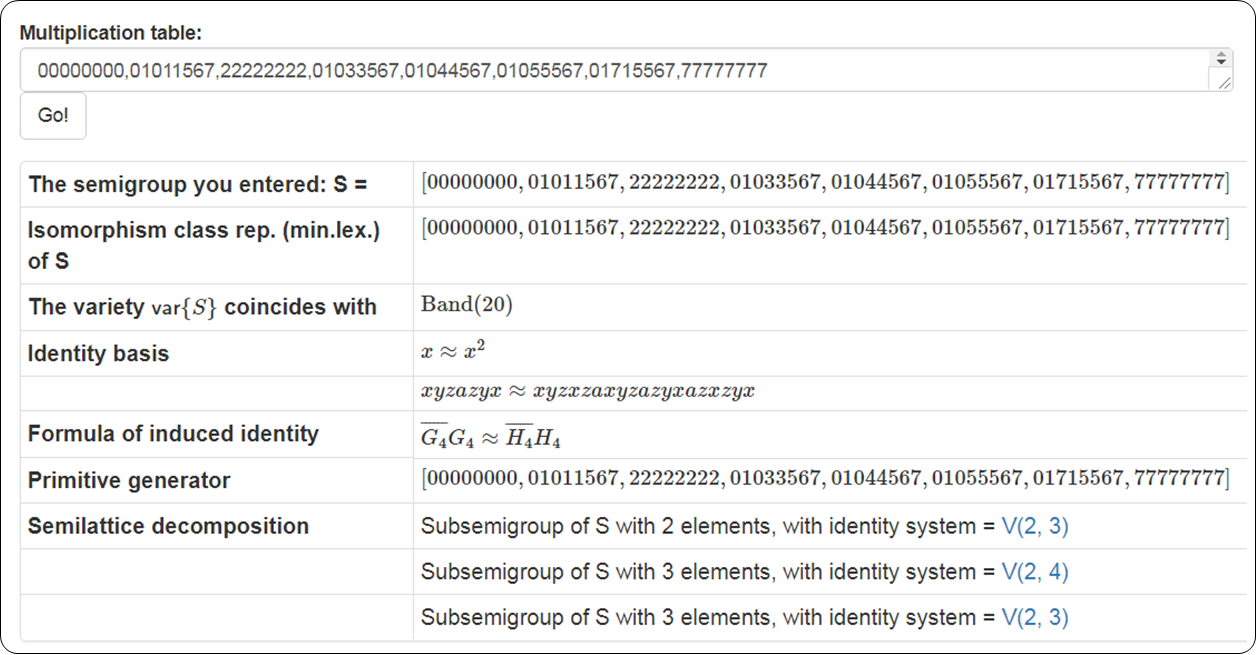}
 \caption{Companion website: the variety generated by an order $8$ band}
    \label{F: Bands}
\end{figure}

\item For a given finite semigroup~$S$, the companion website gives bibliographic information about the variety $\var\{S\}$, its prime decomposition, varieties that cover it, and a generator for $\var\{S\}$ of minimal order; see Figure~\ref{F: varinf} on page~\pageref{F: varinf}.

\begin{figure}[hbt]
  \includegraphics[width=\textwidth]{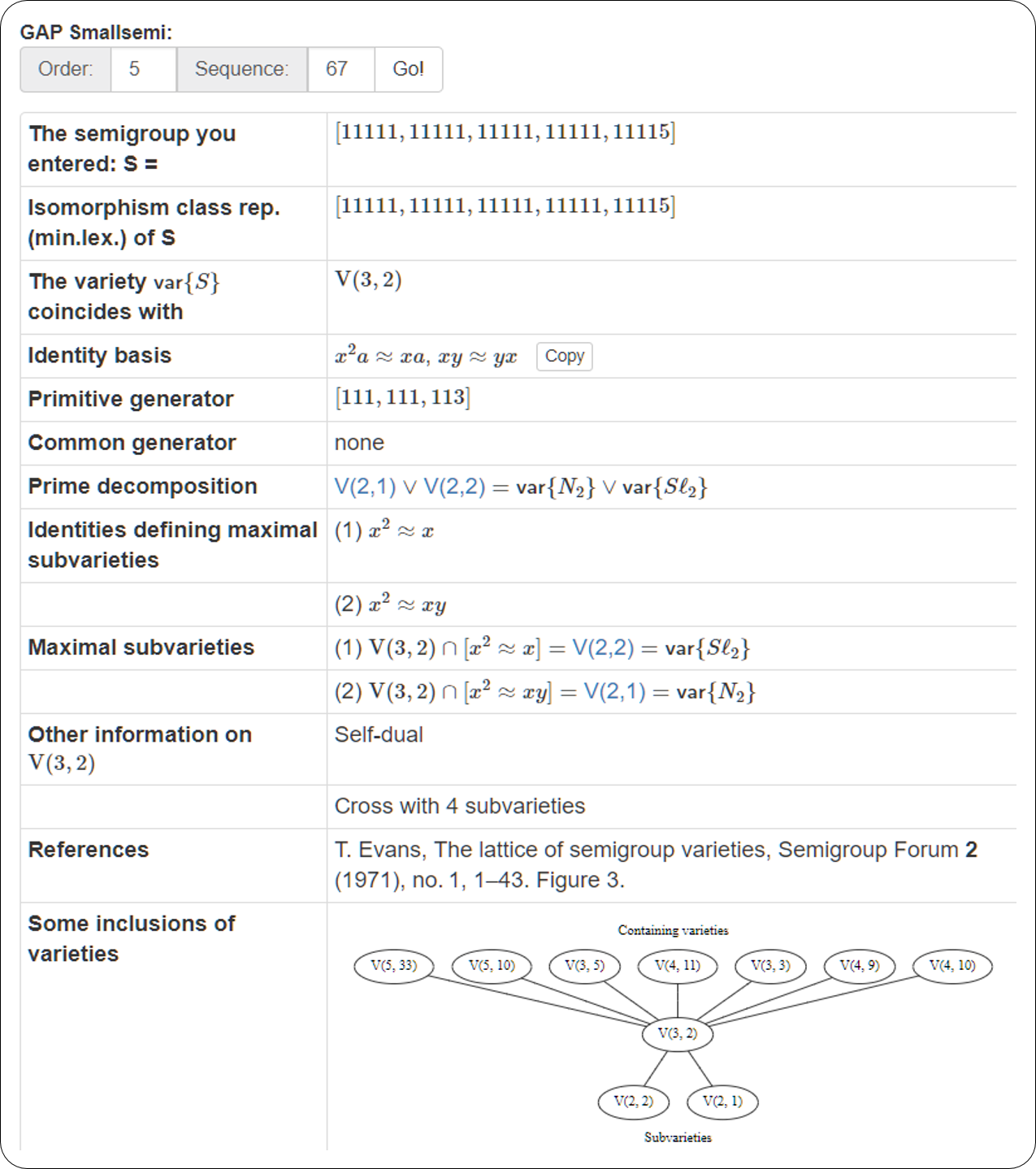}
  \caption{Companion website: example of information displayed if the {\ib} for the variety generated by the given semigroup is found}
  \label{F: varinf}
\end{figure}

\item The \textit{vector} of a semigroup~$S$ of order~$n$, denoted by $\vect(S)$, is the vector of dimension~$n^2$ that is formed by concatenating the~$n$ rows of the Cayley table of~$S$.
For example, the vector $\vect(\JI)$ of the semigroup~$\JI$ in Table~\ref{Tab: U} is $[1,1,1,1,1,1,1,2,3]$; it is unambiguous, and in fact clearer, to only use commas to separate different rows, that is, \[\vect(\JI)=[111,111,123].\]
The isomorphic copies of a given semigroup can then be lexicographically ordered as vectors; for example, the semigroup~$\JI'$ in Table~\ref{Tab: U} is isomorphic to~$\JI$, but since \[ \vect(\JI) = [111,111,123] <_\mathrm{lex} [333,123,333] = \vect(\JI'), \] we place~$\JI$ before~$\JI'$.
The companion website finds the smallest element in each isomorphism class and this is the standard form of the output; of course, this is an expensive feature that can only be applied to semigroups of relatively small order (up to~$11$).
See Figure~\ref{F: minlex} on page~\pageref{F: minlex}.

\begin{table}[ht] \centering
\begin{tabular}{c|ccc}
$\JI$ & 1 & 2 & 3 \\ \hline
    1 & 1 & 1 & 1  \\
    2 & 1 & 1 & 1  \\
    3 & 1 & 2 & 3
\end{tabular}
\quad
\begin{tabular}{c|ccc}
$\JI'$ & 1 & 2 & 3 \\ \hline
    1 & 3 & 3 & 3  \\
    2 & 1 & 2 & 3  \\
    3 & 3 & 3 & 3
\end{tabular}
\caption{The semigroups $\JI$ and $\JI'$}
\label{Tab: U}
\end{table}

\begin{figure}[ht]
  \includegraphics[width=\textwidth]{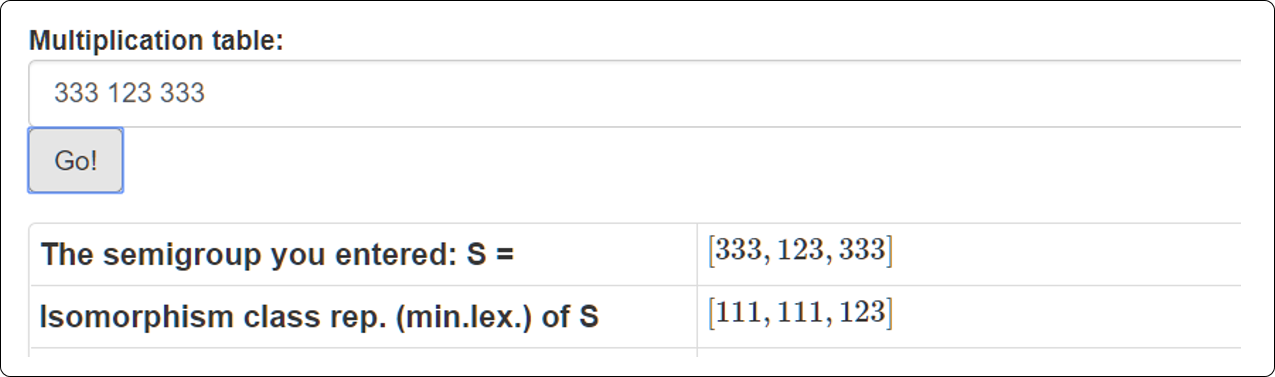}
  \caption{Companion website: computation of the smallest element in the
   isomorphism class of [333,123,333]}
  \label{F: minlex}
\end{figure}

\item For any finitely generated variety~$\bV$, there exist only finitely many non-isomorphic generators of minimal order, say $S_1,S_2,\ldots,S_k$ with \[\vect(S_1) <_\mathrm{lex} \vect(S_2) <_\mathrm{lex} \cdots <_\mathrm{lex} \vect(S_k).\]
Then~$S_1$ is called the \textit{primitive generator} of~$\bV$.
Each variety has a label $\V(n,k)$, where~$n$ is the order of its primitive generator~$S$ and~$k$ is the number of primitive semigroups of order~$n$ for other varieties that lexicographically precede~$S$.
For example, the variety $\var\{\JI\}$ is defined by the identities $\{ x^2a \approx xa, xy^2 \approx yx^2 \}$ and is labeled $\V(3,3)$, meaning that its primitive generator has order~$3$---which happens to be~$\JI$---and there are two semigroups of order~$3$ with vectors preceding $\vect(\JI)$ that are primitive generators for two other distinct varieties, namely $\V(3,1)$ and $\V(3,2)$.
See Figure~\ref{F: v31v32v33} on page~\pageref{F: v31v32v33}.

\begin{figure}[ht]
  \includegraphics[width=\textwidth]{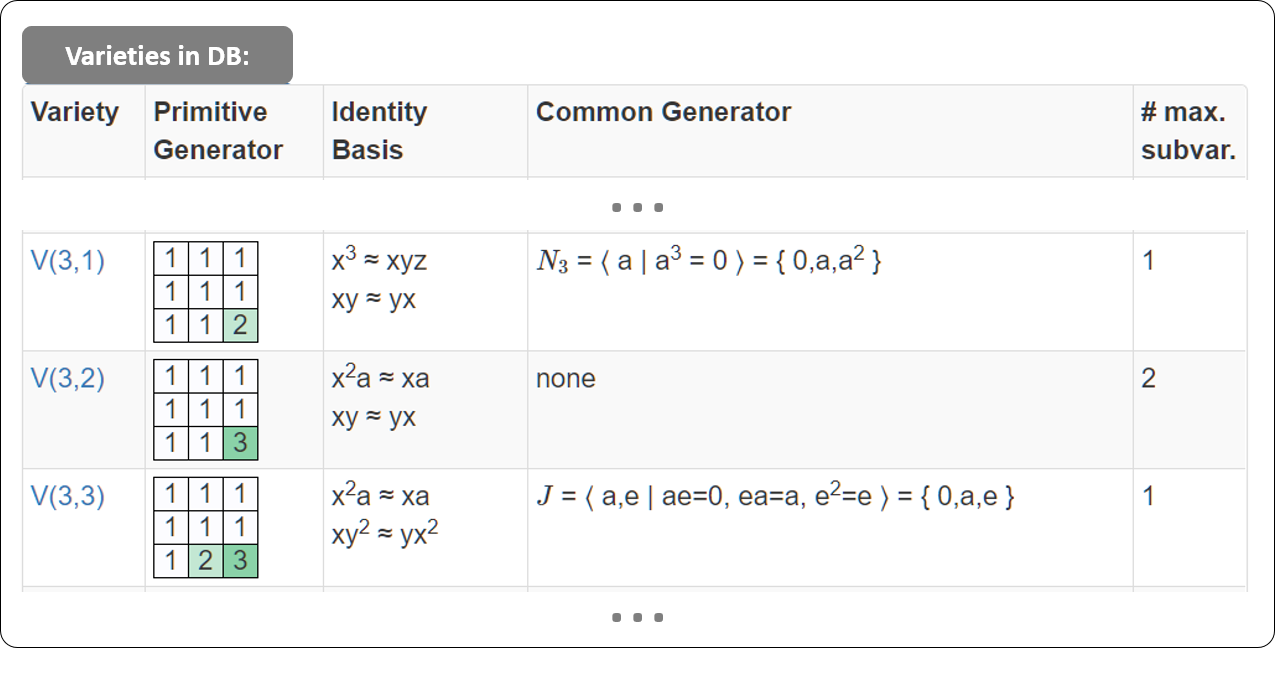}
  \caption{Companion website: varieties $\V(3,1)$, $\V(3,2)$, and $\V(3,3)$ in the database}
  \label{F: v31v32v33}
\end{figure}

\item In many cases, the website provides a presentation for the primitive generator of the given variety.
Conversely, the user can introduce a semigroup as a semigroup presentation in any variety or quasi-variety.
For instance, Kiselman~\cite{Kis02} considered the semigroup with the presentation
\[	
\left\langle c,\ell,m \ \bigg| \begin{array}[c]{l} c^2=c, \, \ell^2=\ell, \, m^2=m, \, c\ell c=\ell c, \, \ell c\ell=\ell c, \\ cmc=mc, \, mcm=mc, \, \ell m\ell=m\ell, \, m\ell m=m\ell \end{array} \! \right\rangle
\]
while investigating some operators in convexity theory.
In less than a second the website shows that this semigroup has 17~elements as a semigroup presentation (as shown in Kiselman~\cite{Kis02}), 7~elements as a band presentation, and 3~elements as a left cancellative semigroup presentation, etc.
See Figure~\ref{F: presentation} on page~\pageref{F: presentation}.


\begin{figure}[ht]
 \includegraphics[width=\textwidth]{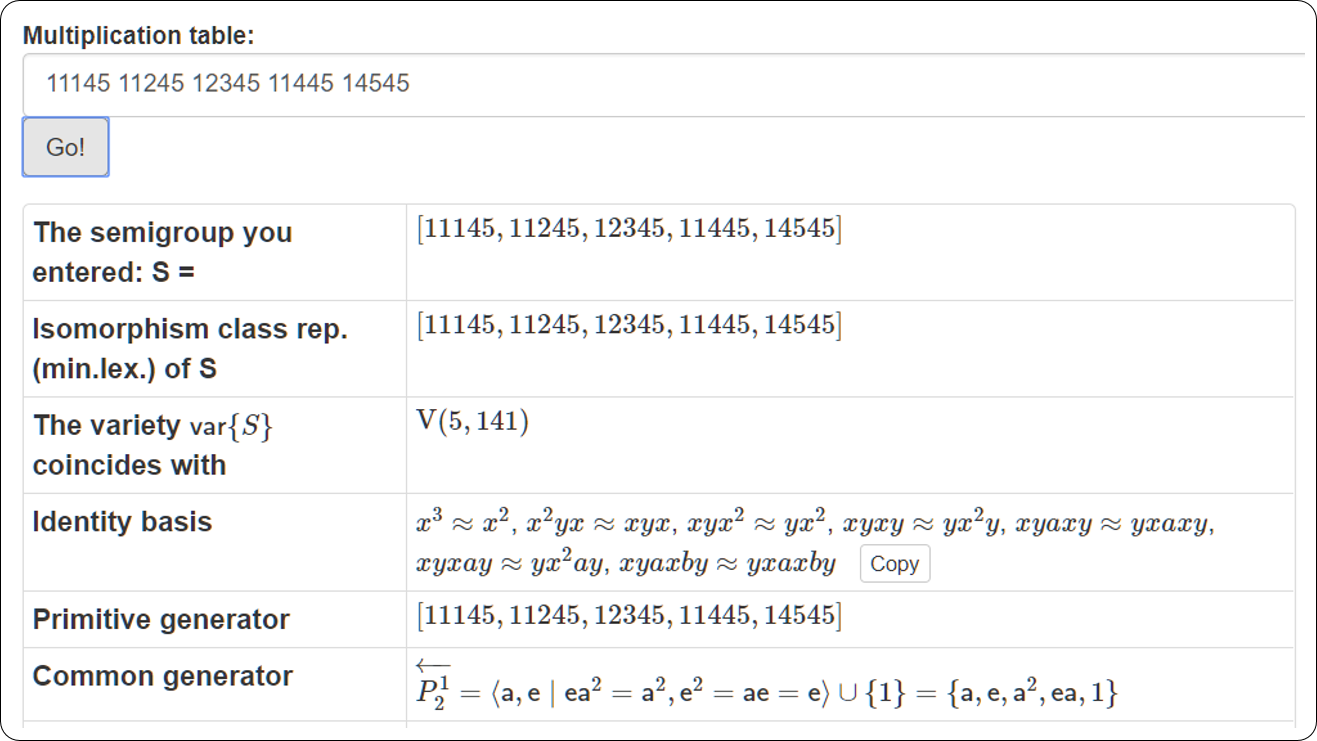}
  \caption{Companion website: example of a presentation provided}
  \label{F: presentation}
\end{figure}

\item The website does not provide an {\ib} for the variety generated by the Kiselman semigroup of order~$17$ computed above.
However, it will say that the Kiselman semigroup of order~$7$ generates the variety of semilattices whose primitive generator is the chain of length two.
If a given semigroup~$S$ of arbitrarily finite order generates a variety whose primitive generator has order~$5$ or less, then the website will automatically provide an {\ib} for the variety $\var\{S\}$.
See Figure~\ref{F: kiselman} on page~\pageref{F: kiselman}.

\begin{figure}[ht]
  \includegraphics[width=\textwidth]{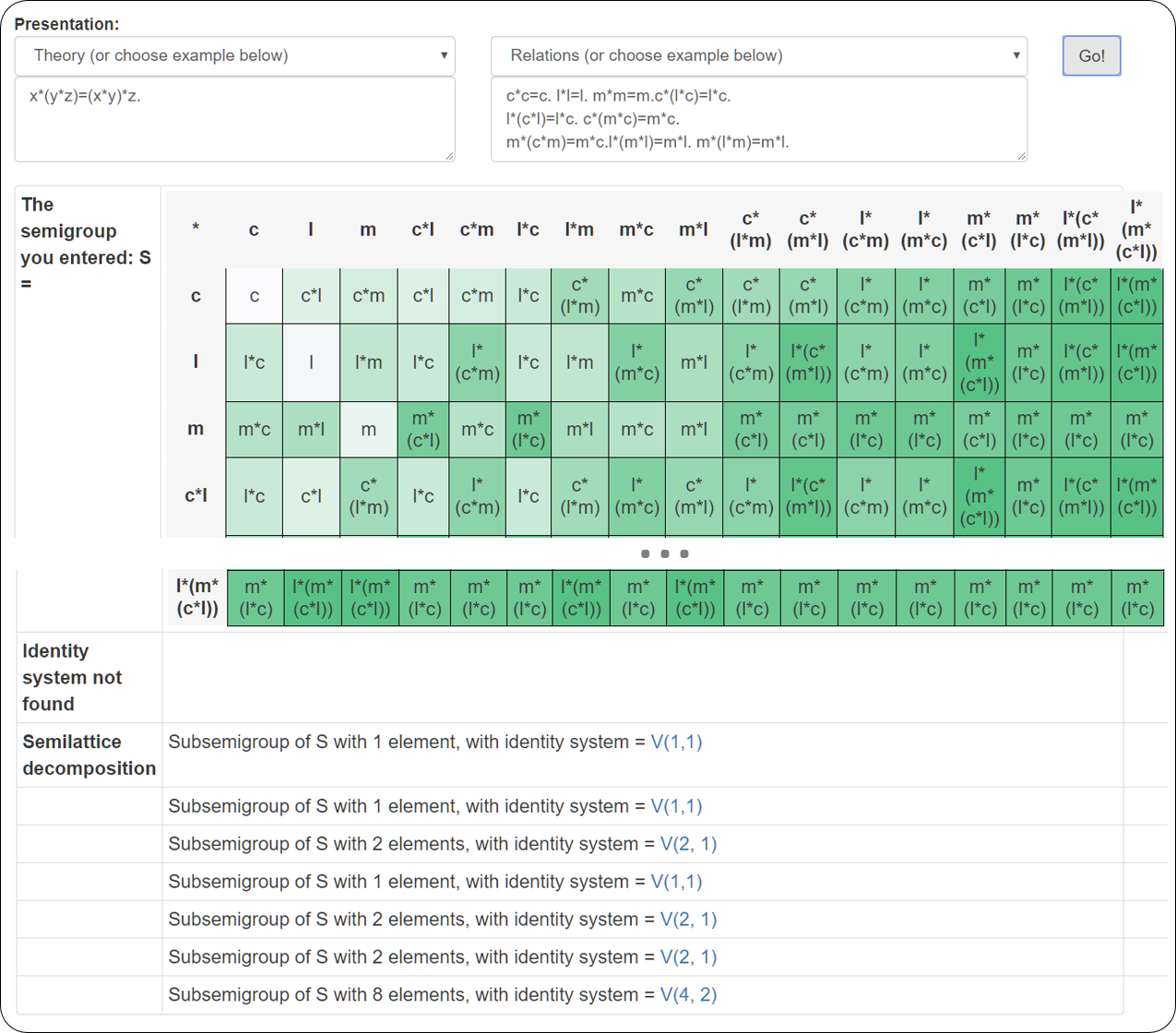}
  \caption{Companion website: Kiselman semigroup entered as a presentation}
  \label{F: kiselman}
\end{figure}

\item In addition to presentations, the user can input semigroups by giving the Cayley table, with several formats and on different sets that one can define, or by introducing identification numbers in the {\GAP} libraries of small groups or small semigroups.
See Figure~\ref{F: input} on page~\pageref{F: input}.


\begin{figure}[ht]
  \includegraphics[width=\textwidth]{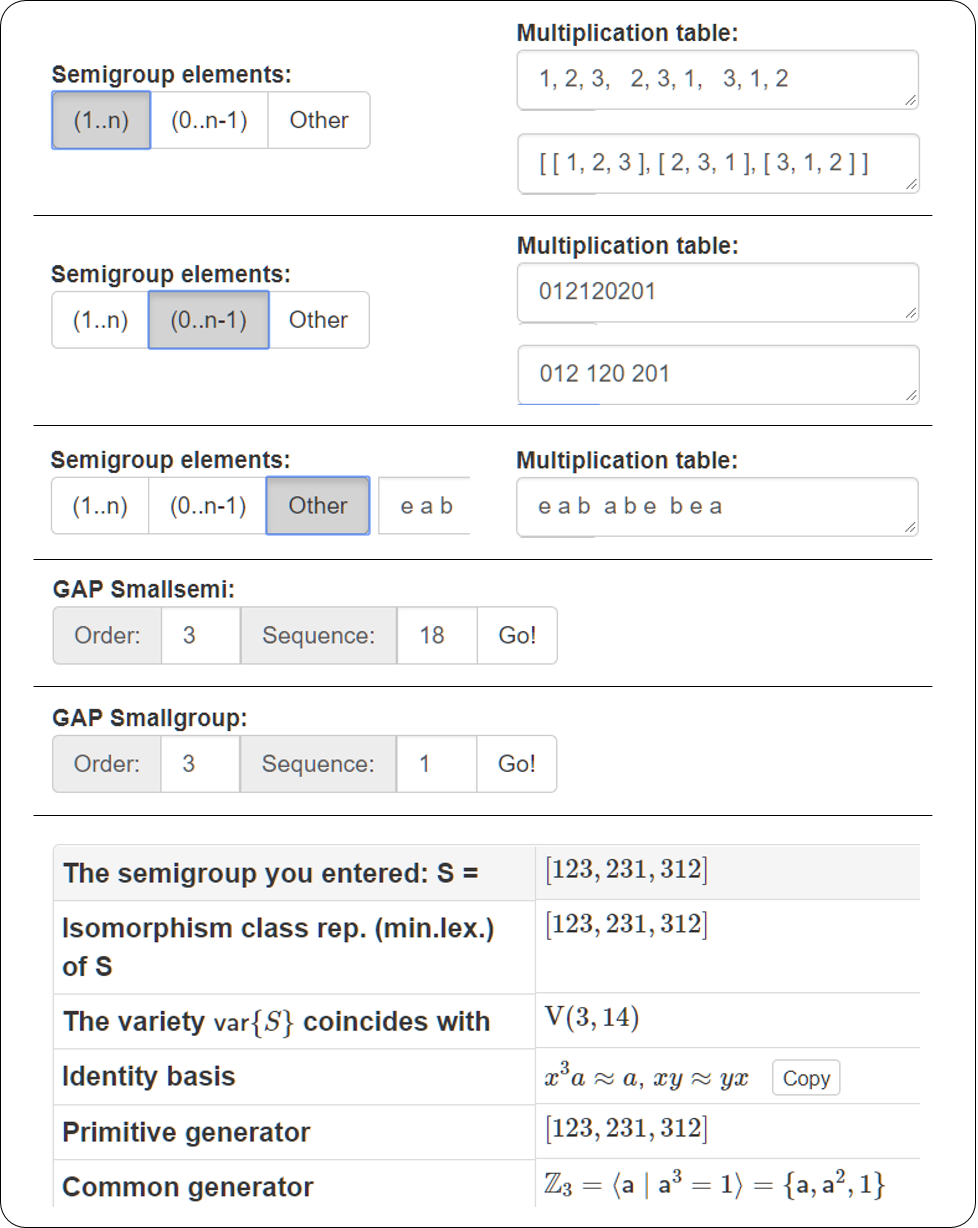}
  \caption{Companion website: examples of input alternatives for semigroup
  [123,231,312]}
  \label{F: input}
\end{figure}

\item The companion website also provides information on dual varieties or self-dual ones when applicable.
Recall that a variety of semigroups is \textit{self-dual} if it is closed under anti-isomorphism.

%

\item Let $\bE_n$ denote the variety of unary semigroups defined by the identities
%
\[ xx^* \approx x^*x, \quad x(x^*)^2 \approx x^*, \quad x^{n+1}x^* \approx x^n. \]
Then the proper inclusions $\bE_1 \subset \bE_2 \subset \bE_3 \subset \cdots$ hold, and for any finite semigroup~$S$, there exist a unary operation~$^*$ on~$S$ and some minimal $n \geq 1$ such that $(S,\,^*\,)$ is a unary semigroup in~$\bE_n$; the companion website finds this natural number~$n$.
For more information on the operation~$^*$ and the varieties~$\bE_n$, see Subsection~\ref{subsec: epigroups} and Shevrin~\cite{She05}.



\item A semilattice~$Y$ is a partially ordered set in which every pair $i,j\in Y$ of elements has a greatest lower bound $i\wedge j$, called the \textit{meet} of~$i$ and~$j$.
A semigroup~$S$ is a \textit{semilattice of semigroups} if there exist a semilattice $(Y,\le)$ and a family $\{ S_i \}_{i\in Y}$ of semigroups indexed by~$Y$ such that $S=\bigcup_{i\in Y}S_i$ and $S_iS_j \subseteq S_{i\wedge j}$.
Every semigroup can be decomposed as a semilattice of semigroups $\{ S_i \}_{i\in Y}$ with each $S_i$ being semilattice indecomposable~\cite{Tam56}.
Based on results from Tamura~\cite{Tam72}, the companion website finds the largest semilattice decomposition of a given semigroup~$S$ into semilattice indecomposable semigroups $\{ S_i \}_{i\in Y}$, and provides the variety generated by each~$S_i$.
This tool can be used on a relatively large semigroup~$S$, even when we cannot determine an {\ib} for the variety $\var\{S\}$.


\begin{figure}[ht]
  \includegraphics[width=\textwidth]{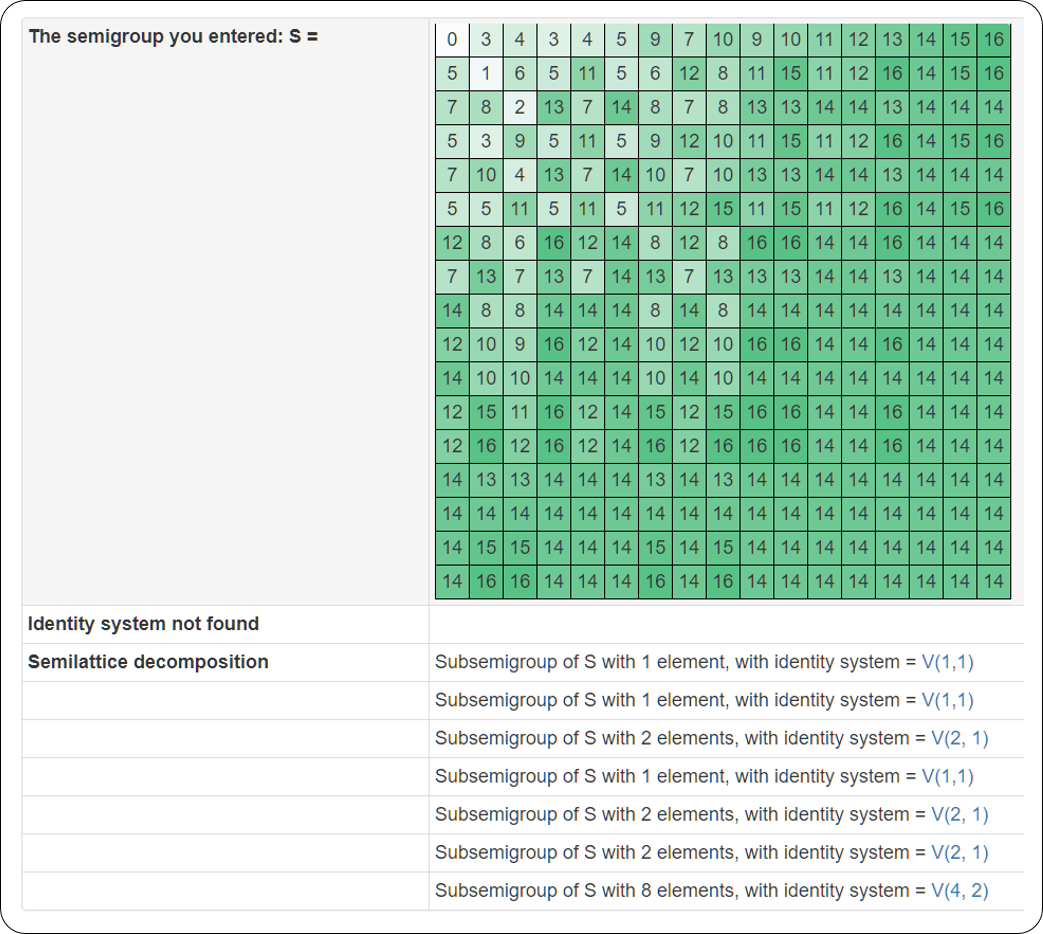}
  \caption{Companion website: example of semitalattice decomposition of a
  order 17 semigroup}
  \label{F: semilattices}
\end{figure}

\item Let~$\Sigma$ be some given first order theory.
The companion website can find all the varieties~$\bV$ in the database such that $\Sigma \vdash \bV$ or $\bV \models \Sigma$.


\begin{figure}[H]
    \includegraphics[width=10cm,height=5cm]{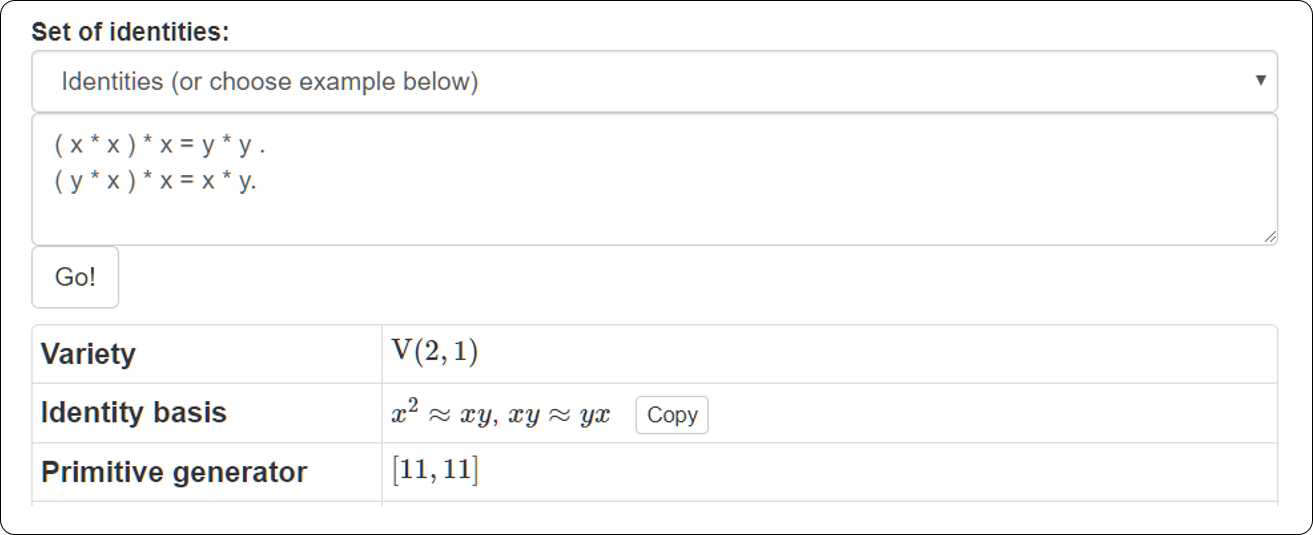}
  \caption{Companion website: a set of identities entered
  by the user is found to be equivalent to an {\ib} for a variety in the database}
  \label{F: setidt}
\end{figure}

\begin{figure}[H]
  \includegraphics[width=10cm,height=8cm]{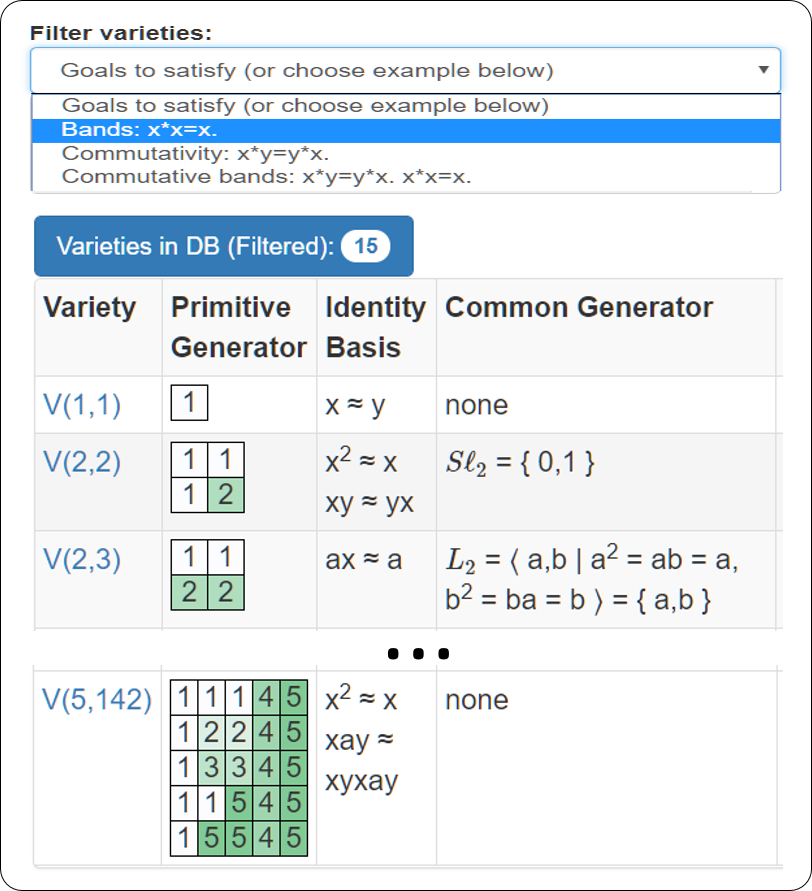}
  \caption{Companion website: finding all varieties in the database that satisfy some given conditions}
  \label{F: filtervar}
\end{figure}

\item  The companion website can provide conjectures for the variety generated by a large semigroup, using an algorithm that gave the correct result on all semigroups up to order~$5$.
However, given the computational cost of this algorithm, anyone interested should first contact one of the authors.

{\color{black} Let $S$ be the semigroup with universe $\{1,\ldots,5\}$ and whose  Cayley table has the following rows:  $11111,11111,11113,44444,12345$. The variety generated by $S$ is defined by the identities:
 \[x^3=x^2\hspace{0.4cm} x^2yx=xyx\hspace{0.4cm} xyxz=x^2yz\hspace{0.4cm}  xyz^2x=x^2yz^2.\]  Our algorithm produced the following candidate:  
 \[x^3=x^2\hspace{0.5cm} x^2yx=xyx\hspace{0.4cm} xyxz=x^2yz\hspace{0.5cm} xy^2x=x^2y^2.\]  

It is easy to see that the two sets are equivalent and hence our candidate base is in fact a base for the variety generated by $S$. Note that the two bases differ only on the last identity, with the elegance prize going to the one found by the computer.}

\item In particular, the companion website can give some information about user's conjectures.
Suppose that we have a semigroup~$S$ and guess that a certain set~$\Sigma$ of identities is an {\ib} for $\var\{S\}$.
Then the website will try to see if $\var\{S\}$ is in the database; if yes, it will try to prove if the stored {\ib} is equivalent to the given one and return the result; it is very unlikely that no result is returned in such a case.
If $\var\{S\}$ does not belong to the database, then the website will try to find identities holding in~$S$, but not provable from~$\Sigma$.
If some are found, then the result is returned.
Otherwise, the user's conjecture is returned as a reasonable one.
%
%
\item We will keep the website updated with new discovered results in order to have a state of the art tool assisting the work of mathematicians.
\end{enumerate}

In Section~\ref{preliminaries}, we give some background material on varieties of groups and of semigroups, the lattice of varieties of bands, varieties with infinitely many subvarieties, an infinite chain of varieties of epigroups, and semilattice decompositions of semigroups.
Section~\ref{sec: groups} is dedicated to a survey of some known results on varieties generated by small groups; it consists of mostly old material and we collect the main results here to highlight the gaps waiting to be filled.
It is our conviction that the topic was more or less \textit{abandoned}, 
not because everything was too easy, but exactly the opposite.
Given the classification of finite simple groups, perhaps it is time for group theorists to start looking into varieties of groups again.
In addition, for experts in semigroup theory, it might be useful to know to which varieties of groups belong the maximal subgroups (the ${\mathcal H}$-classes) of the semigroup.
Section~\ref{smallsemigroups} deals with varieties generated by semigroups of order~$5$, and also treats the case of inherently non-finitely based finite semigroups.
Section~\ref{companion} introduces the features of the companion website and explains how to use it.
Section~\ref{sec: varieties small semigroups} provides the database of varieties generated by semigroups of orders up to~$4$.
Then we have a section on problems, and three appendix sections providing justifications of results in Section~\ref{sec: varieties small semigroups}.

\section{Preliminaries}\label{preliminaries}
\subsection{Isomorphic semigroups and lexicographic minimum}

Two algebras~$A$ and~$B$ of the same type are said to be \textit{isomorphic}, indicated by $A \cong B$, if there exists an isomorphism between them.
The relation~$\cong$ is an equivalence relation on any class of algebras of the same type.
Occasionally, given a finite algebra~$A$, it is practical to have a canonical representative of the equivalence class $[A]_{\cong}$.
For a semigroup~$S$, an obvious choice for the representative of the class $[S]_{\cong}$ is the semigroup whose vector lexicographically precedes the vectors of all other semigroups in $[S]_{\cong}$.
For instance, consider the semigroup
\[
P = \langle a,b \,|\, ab=a, \, ba = 0, \, b^2=b \rangle = \{ 0,a,b\}.
\]
Then there are six semigroups on the set $\{1,2,3\}$ that are isomorphic to~$P$, as shown in Table~\ref{Tab: P}.
Since $\vect(S_1) \leq_\mathrm{lex} \vect(S_i)$ for all $i \neq 1$, the semigroup~$S_1$ is the representative of the class $[P]_{\cong}$.

\begin{table}[ht] \centering
\begin{tabular}{c|ccc}
$S_1$ & 1 & 2 & 3 \\ \hline 1 & 1 & 1 & 1 \\ 2 & 1 & 1 & 2 \\ 3 & 1 & 1 & 3
\end{tabular}
\quad
\begin{tabular}{c|ccc}
$S_2$ & 1 & 2 & 3 \\ \hline 1 & 1 & 1 & 1 \\ 2 & 1 & 2 & 1 \\ 3 & 1 & 3 & 1
\end{tabular}
\quad
\begin{tabular}{c|ccc}
$S_3$ & 1 & 2 & 3 \\ \hline 1 & 1 & 2 & 2 \\ 2 & 2 & 2 & 2 \\ 3 & 3 & 2 & 2
\end{tabular}
\\[0.08in]
\begin{tabular}{c|ccc}
$S_4$ & 1 & 2 & 3 \\ \hline 1 & 1 & 3 & 3 \\ 2 & 2 & 3 & 3 \\ 3 & 3 & 3 & 3
\end{tabular}
\quad
\begin{tabular}{c|ccc}
$S_5$ & 1 & 2 & 3 \\ \hline 1 & 2 & 2 & 1 \\ 2 & 2 & 2 & 2 \\ 3 & 2 & 2 & 3
\end{tabular}
\quad
\begin{tabular}{c|ccc}
$S_6$ & 1 & 2 & 3 \\ \hline 1 & 3 & 1 & 3 \\ 2 & 3 & 2 & 3 \\ 3 & 3 & 3 & 3
\end{tabular}
\caption{Semigroups isomorphic to~$P$}
\label{Tab: P}
\end{table}

The \textit{dual} of a semigroup~$S$, denoted by~$\dual{S}$, is the semigroup obtained from~$S$ by reversing its operation, that is, for any $a,b \in \dual{S} = S$, the product~$ab$ in~$\dual{S}$ is equal to the product~$ba$ in~$S$.
The Cayley table of~$\dual{S}$ is obtained simply by transposing the Cayley table of~$S$.
For instance, the semigroup~$\dual{S_1}$ is isomorphic to the semigroup~$\JI$ in Table~\ref{Tab: U}.
The \textit{dual} of a variety~$\bV$ is the variety $\dual{\bV} = \{ \dual{S} \,|\, S \in \bV \}$.
A variety~$\bV$ is \textit{self-dual} if $\bV = \dual{\bV}$.

Two semigroups~$S$ and~$T$ are \textit{equivalent} if either $S \cong T$ or $\dual{S} \cong T$.
In the {\GAP} package {\Smallsemi}, semigroups are stored up to equivalence but not up to isomorphism, a decision not without some disadvantages.
In this paper, unless otherwise stated, we work with semigroups up to isomorphism.

%

\subsection{Varieties of semigroups} 

The variety generated by an algebra~$A$, denoted by $\var\{A\}$, is the smallest class of algebras of the same type containing~$A$ that is closed under the formation of homomorphic images, subalgebras, and arbitrary direct products.
Since a variety $\var\{A\}$ coincides with the class of all algebras that satisfy the identities of~$A$, two algebras generate the same variety if and only if they satisfy the same identities.
It is clear that if~$A$ and~$B$ are isomorphic algebras, then $\var\{A\} = \var\{B\}$; however, the converse does not hold in general, even if the algebras~$A$ and~$B$ have the same order.
For example, the dihedral group~$D_4$ and the quaternion group~$Q$ are groups of order~$8$ that generate the same variety~\cite{Wei62}, but they are not isomorphic.

Up to isomorphism, the number of semigroups of order up to five is~2,133 \cite[A027851]{OEIS}, while the number of varieties generated by these semigroups is only~218.

A \textit{\ib} for a variety~$\bV$ is a set of identities holding in~$\bV$ from which all other identities of~$\bV$ can be deduced.
A variety is \textit{finitely based} if it possesses a finite {\ib}.
Since a semigroup satisfies the same identities as the variety it generates, it is unambiguous to define an \textit{\ib} for a semigroup~$S$ to be an {\ib} for $\var\{S\}$, and say that~$S$ is \textit{finitely based} whenever $\var\{S\}$ is finitely based.
Every variety generated by a semigroup of order at most~$5$ is finitely based, but up to isomorphism, precisely four semigroups of order~$6$ are non-finitely based~\cite{LLZ12}; see Subsection~\ref{subsec: NFB order 6}


\subsection{Varieties of groups} \label{subsec:L(G)}

For a general reference on varieties of groups, we recommend the monograph of H. Neumann~\cite{NeuH67}.
Unlike what happens in semigroups, every variety generated by a finite group has a finite {\ib}, and in group theory, every finite set of identities is equivalent to a single identity.
Therefore every variety generated by a finite group can be defined by a single identity.
We will see a similar phenomenon in the variety of bands below.
More details on varieties of groups can be found in Section~\ref{sec: groups}.

\subsection{The lattice of varieties of bands} \label{subsec: L(B)}

A description of the lattice $\sL(\bB)$ of varieties of bands can be found in Birjukov~\cite{Bir70}, Fennemore~\cite{Fen71}, Gerhard~\cite{Ger70}, Gerhard and Petrich~\cite{GP89}, and Howie~\cite{How95}; see Figure~\ref{F: bands}.
At the very top of the lattice is the variety $\bB = [x^2 \approx x]$ of all bands.
In the lower region is the sublattice $\sL(\bN)$ of $\sL(\bB)$ consisting of eight varieties:
\begin{align*}
\bN & = [xyzx \approx xzyx]_\bB, & & \text{normal bands}; \\
\bLN & = [xyz \approx xzy]_\bB, & & \text{left normal bands}; \\
\bRN & = [xyz \approx yxz]_\bB, & & \text{right normal bands}; \\
\bSL & = [xy \approx yx]_\bB, & & \text{semilattices}; \\
\bRB & = [xyx \approx x], & & \text{rectangular bands}; \\
\bLZ & = [xy \approx x], & & \text{left zero bands}; \\
\bRZ & = [xy \approx y], & & \text{right zero bands}; \\
\bT & = [x \approx y], & & \text{trivial bands}.
\end{align*}
The remaining varieties in the lattice $\sL(\bB)$ are defined by identities that are formed by the words $\{ \ttG_n,\ttH_n,\ttI_n \,|\, n \geq 2 \}$ inductively defined as follows:
\begin{align*}
& & \ttG_2&= x_2x_1, & \ttH_2 & = x_2, & \ttI_2 & = x_2x_1x_2,\\
& \text{and} & \ttG_n&= x_n\ov{\ttG_{n-1}}, & \ttH_n &= \ttG_nx_n\ov{\ttH_{n-1}}, & \ttI_n & = \ttG_nx_n\ov{\ttI_{n-1}}, \qquad \text{for all } n\geq3,
\end{align*}
where~$\ov{X}$ is the word~$X$ written in reverse.
For example,
\[
[\ttG_3 \approx \ttH_3]_\bB = [x_3x_1x_2 \approx x_3x_1x_2x_3x_2, \, x^2 \approx x].
\]
By simple inspection of the identities in Figure~\ref{F: bands}, it is clear that the varieties in column~3 are self-dual, the varieties in columns~1 and~5 are dual to each other, and the varieties in columns~2 and~4 are dual to each other.

The variety generated by a band~$B$ is the variety~$\bV$ of bands that satisfies both of the following properties: $B$ belongs to~$\bV$ and~$B$ is excluded from every maximal subvariety of~$\bV$.
When a semigroup~$S$ is entered into the companion website, there is a first test to check if~$S$ is a band.
In the affirmative case, the website crawls up the lattice in Figure~\ref{F: bands}; the first identity satisfied by~$S$ defines the variety $\var\{S\}$.


\subsection{Varieties with infinitely many subvarieties}

A variety that contains only finitely many subvarieties is said to be \textit{small}.
It easily follows from the well-known theorem of Oates and Powell~\cite{OP64} that every finite group generates a small variety of semigroups.
But this result does not hold in general.
A small counterexample is the monoid $\Niltwo^1$ obtained by adjoining an identity element to the nilpotent semigroup $\Niltwo = \langle a \,|\, a^2=0 \rangle$ of order~$2$; see Figure~\ref{F: N2i} on page~\pageref{F: N2i}.
Not only is the variety $\mathbf{\Niltwo^1} = \var\{\Niltwo^1\}$ not small~\cite{Eva71}, it is the only non-small variety among all varieties generated by a semigroup of order~$3$ or less; see Section~\ref{sec: varieties small semigroups}.

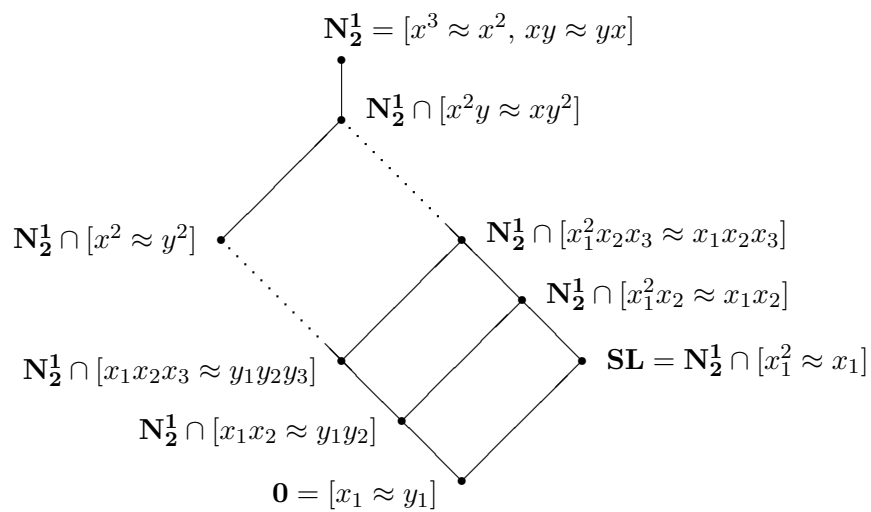
\begin{figure}[ht]
\begin{center}
\begin{picture}(210,200)(00,10) \setlength{\unitlength}{0.8mm}
\put(30,80){\circle*{\circlesize}}
\put(30,70){\circle*{\circlesize}}
\put(10,50){\circle*{\circlesize}} \put(50,50){\circle*{\circlesize}}
\put(60,40){\circle*{\circlesize}}
\put(30,30){\circle*{\circlesize}} \put(70,30){\circle*{\circlesize}}
\put(40,20){\circle*{\circlesize}}
\put(50,10){\circle*{\circlesize}}
\put(10,50){\line(1,1){20}} \put(30,30){\line(1,1){20}} \put(40,20){\line(1,1){20}} \put(50,10){\line(1,1){20}}
\put(50,10){\line(-1,1){22.75}} \put(70,30){\line(-1,1){22.75}}
\put(30,70){\line(0,1){10}}
\put(27,85){\makebox(0,0)[l]{$\mathbf{\Niltwo^1}= [x^3 \approx x^2, \, xy \approx yx]$}}
\put(34,69){\makebox(0,0)[bl]{$\mathbf{\Niltwo^1} \cap [x^2y \approx xy^2]$}}
\put(54,51){\makebox(0,0)[l]{$\mathbf{\Niltwo^1} \cap [x_1^2x_2x_3 \approx x_1x_2x_3]$}}
\put(64,41){\makebox(0,0)[l]{$\mathbf{\Niltwo^1} \cap [x_1^2x_2 \approx x_1x_2]$}}
\put(74,30){\makebox(0,0)[l]{$\bSL = \mathbf{\Niltwo^1} \cap [x_1^2 \approx x_1]$}}
\put(06,50){\makebox(0,0)[r]{$\mathbf{\Niltwo^1} \cap [x^2 \approx y^2]$}}
\put(26,31){\makebox(0,0)[tr]{$\mathbf{\Niltwo^1} \cap [x_1x_2x_3 \approx y_1y_2y_3]$}}
\put(36,21){\makebox(0,0)[tr]{$\mathbf{\Niltwo^1} \cap [x_1x_2 \approx y_1y_2]$}}
\put(46,10){\makebox(0,0)[tr]{$\bT = [x_1 \approx y_1]$}}
\thicklines
\qbezier[12](10,50)(18,42)(26,34) \qbezier[12](30,70)(38,62)(46,54)
\end{picture}
\end{center}
\caption{The lattice of subvarieties of $\mathbf{N_2^1} = \var\{N_2^1\}$ }
\label{F: N2i}
\end{figure}

As for the variety generated by a semigroup of order greater than~$3$, properties more extreme than being non-small can be satisfied.
For instance, there exist
\begin{itemize}
\item semigroups of order~$4$ that generate varieties that are \textit{finitely universal}~\cite{Lee07a} in the sense that their lattices of subvarieties each embeds all finite lattices;
\item semigroups of order~$6$ that generate varieties with continuum many subvarieties~\cite{ELL10,Jac00}.
\end{itemize}
All examples of varieties with continuum many subvarieties discovered so far are also finitely universal.
It is unknown if there exists a variety with continuum many subvarieties that is not finitely universal.
Refer to Shevrin {\etal}.~\cite{SVV09} for a survey of results regarding other properties satisfied by lattices of varieties.

Given a finite semigroup, it is of natural interest to determine if it generates a small variety.
Whether or not smallness of a variety is decidable remains open, but some special case has been found.
Recall that an identity of the form \[ x_1 x_2 \cdots x_n \approx x_{\pi(1)} x_{\pi(2)} \cdots x_{\pi(n)}, \] where~$\pi$ is some nontrivial permutation on $\{1,2,\ldots,n\}$, is called a \textit{permutation identity}, while a nontrivial identity of the form \[ x_1 x_2 \cdots x_n \approx \bw \] that is not a permutation identity is said to be \textit{diverse}.

\begin{proposition}[Malyshev~\cite{Mal81}]
Any variety that satisfies some permutation identity and some diverse identity is small\up.
\end{proposition}

\subsection{Epigroups} \label{subsec: epigroups}


Let~$S$ be a semigroup.
An element $a\in S$ is an \textit{epigroup element} if there exists an integer $n \geq 1$ such that $a^n$ belongs to a subgroup of~$S$, that is, the $\mathcal{H}$-class $H_{a^n}$ of $a^n$ is a group; if $n=1$, then~$a$ is said to be \textit{completely regular}.
If we denote by~$e$ the identity element of $H_{a^n}$, then $ae$ is in $H_{a^n}$ and we define the \textit{pseudo-inverse}~$a'$ of~$a$ by $a'=(ae)^{-1}$, where $(ae)^{-1}$ denotes the inverse of $ae$ in the group $H_{a^n}$ \cite[Subsection~2.1]{She05}.
An \textit{epigroup} is a semigroup consisting entirely of epigroup elements, and a \textit{completely regular semigroup} is a semigroup whose elements are all completely regular.
The important fact for us is that all finite semigroups are examples of epigroups.
Following Petrich and Reilly~\cite{PR99} for completely regular semigroups and Shevrin~\cite{She05} for epigroups, it is now customary to consider an epigroup or a completely regular semigroup $(S,\,\cdot\,)$ as a \textit{unary} semigroup $(S,\,\cdot\,,\,'\,)$, where $x\mapsto x'$ is the map sending each element to its pseudo-inverse.

For any semigroup~$S$, let $\Epi(S)$ denote the set of all epigroup elements of~$S$ and let $\Epi_n(S)$ denote the subset of $\Epi(S)$ consisting of elements of index bounded by~$n$.
Then the inclusions \[\Epi_1(S) \subseteq \Epi_2(S) \subseteq \cdots \subseteq \bigcup_{n\geq 1} \Epi_n(S) = \Epi(S) \] hold, where $\Epi_1(S)$ consists of completely regular elements of~$S$, and $\Epi(S) = S$ if and only if~$S$ is an epigroup.

For any $a\in \Epi_n(S)$, let~$e_a$ denote the identity element of the group~$H_{a^n}$.
Then $ae_a = e_aa$ is in~$H_{a^n}$ and the definition of \textit{pseudo-inverse} introduced above leads to a characterization of the epigroup elements of the semigroup: $a\in \Epi(S)$ if and only if there exist some $n \geq 1$ and some (necessarily unique) element $a'\in S$ such that
\begin{equation}\label{etfh}
a'aa' = a',\quad aa'=a'a,\quad a^{n+1} a' = a^n;
\end{equation}
see Shevrin~\cite[Section 2]{She05}.
If~$a$ is an epigroup element, then so is~$a'$ with $a'' = aa'a$.
The element~$a''$ is always completely regular and $a''' = a'$.
A standard notation in finite semigroup theory is to write $a^\omega = aa'$ for an epigroup element~$a$; see, for example, Almeida~\cite{Alm94}.
Then \[ a^\omega = a''a'=a'a'', \quad (a')^\omega = (a'')^\omega = a^\omega, \] and more generally, for any $m \geq 1$, \[ a^\omega = (aa')^m=(a')^m a^m = a^m(a')^m. \]

For each $n \geq 1$, the class $\bE_n$ consisting of all epigroups~$S$ such that $S = \Epi_n(S)$ is a variety; in particular, $\bE_1$ is the class of completely regular semigroups.
The chain $\bE_1 \subset \bE_2 \subset \bE_3 \subset \cdots$ of varieties has the following property~\cite{She05}: for any variety~$\bV$ of epigroups, there exists a smallest $n \geq 1$ such that $\bV \subseteq \bE_n$.
Given a finite semigroup~$S$, the companion website finds the smallest~$n$ such that $S \in \bE_n$.
This gives some occasionally useful information about the given semigroup, but of course it does not match knowing an {\ib} for $\var\{S\}$.

\subsection{Semilattice decompositions of semigroups}
There are many ways that a semigroup can be decomposed into smaller subsemigroups, for example, direct products, subdirect products, and Zappa--Sz\'{e}p extensions.
Some has the property that each component cannot be further decomposed using the same tool, in which case the decomposition is said to be \textit{atomic}.
An obvious example of atomic decompositions for finite algebras is the direct product decomposition as, resorting on an argument similar to the one used to prove that every natural number is a product of prime numbers, we can easily show that every finite algebra can be decomposed in a direct product  of directly indecomposable algebras.
Finding atomic decompositions of infinite semigroups is more difficult; according to Bogdanovi\'{c} {\etal}.~\cite{BCP11}, there are only five known atomic decompositions of general semigroups: semilattice decompositions~\cite{Tam56}, ordinal decomposition~\cite{Lya60}, $U$-decomposition \cite{She61}, orthogonal decomposition~\cite{BC95}, and the general subdirect decomposition whose atomicity was proved by Birkhoff.
%

Here we will concentrate on semilattice decompositions of semigroups.
We saw above that a semilattice is a commutative band.
It is easy to prove that every semilattice~$Y$ induces a partially ordered set in which every pair $i,j\in Y$ of elements has a meet $i \wedge j$; conversely, every such partially ordered set induces a semilattice.
Therefore, the term \textit{semilattice} is commonly used to refer to a commutative band or a partially ordered set admitting meet of every pair of elements.
In this subsection it is more convenient to use it in the latter sense.

A semigroup~$S$ is a \textit{semilattice of semigroups} if there exist a semilattice $(Y,\le)$ and a collection $\{S_i\}_{i\in Y}$ of semigroups indexed by~$Y$ such that $S=\bigcup_{i\in Y}S_i$ and $S_iS_j\subseteq S_{i\wedge j}$.
Every semigroup can be decomposed as a semilattice $\{S_i\}_{i\in Y}$ of semigroups~$S_i$ that are semilattice indecomposable~\cite{Tam56}.

In Tamura~\cite{Tam72}, two equivalent ways of finding the smallest semilattice congruence are provided.
For any semigroup~$S$, let~$S^1$ denote the smallest monoid containing~$S$, that is,
\[
S^1 =
\begin{cases}
\, S & \text{if $S$ is not a monoid}, \\
\, S \cup \{1\} & \text{otherwise}.
\end{cases}
\]
Then the smallest semilattice decomposition of~$S$ is the smallest partition containing the sets
\[
\big\{(x,y)\in S^1\times S^1 \,\big|\, \{xy,yx,xyx\}\big\}.
\]

The companion website finds the largest semilattice decomposition of a given semigroup~$S$ into semilattice indecomposable semigroups $\{S_i\}_{i\in Y}$, and provides the variety generated by each~$S_i$.
This tool can be used on a semigroup~$S$ of relatively large order, even when we cannot determine an {\ib} for the variety $\var\{S\}$.


\begin{figure}[H]
  \includegraphics[width=\textwidth]{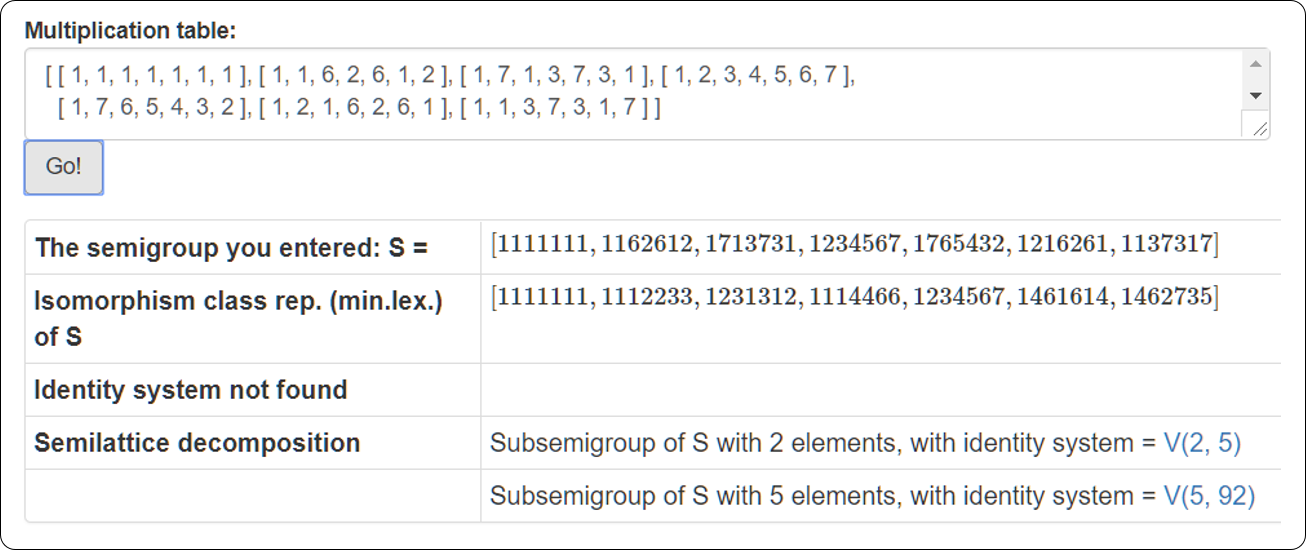}
  \caption{Companion website: semilattice decomposition of an order 7 semigroup}
  \label{F: 7n715832}
\end{figure}

\section{Varieties of groups} \label{sec: groups}

%

The theory of varieties of groups differs from that of semigroups in several ways, which will be briefly mentioned here.
In particular, after a decade of activity, the monograph~\cite{NeuH67} of H.~Neumann was published; this is still the best reference for the subject.
Also, the notation used in H.~Neumann~\cite{NeuH67} became standard among group theorists: we will point out some of the differences.
In particular, varieties of groups are typically denoted by Fraktur capital letters, such as~$\mathfrak{A}$ for the variety of abelian groups; following the usage established earlier, we will use bold-face letters such as~$\mathbf{A}$ instead.

\subsection{The basics}\label{groups}

As briefly noted in Section~\ref{subsec:L(G)}, every group identity can be put into the form $\bw \approx 1$, where~$\bw$ is a word in the variables and their inverses.
We can regard~$\bw$ as an element of the free group $F(X)$ over a countable set~$X$ of variables.
The identities satisfied by a variety~$\mathbf{V}$ form a \textit{fully invariant subgroup} of $F(X)$, one mapped into itself by all endomorphisms of the group.
Thus there is a bijection between varieties of groups and fully invariant subgroups of $F(X)$.

Each finite nontrivial group with finite exponent $e \geq 2$ satisfies the identity $x^e \approx 1$ and so also the identity $x^{e-1} \approx x^{-1}$.
Therefore any identity of a finite group is equivalent to one of the form $\bw \approx 1$, where~$\bw$ is a semigroup word.
In fact, a more specific result holds.
Recall that a \textit{commutator word} is an element of the derived subgroup of the free group. {\color{black} Alternatively, a commutator word can 
be described as one in which the sum of the exponents of every variable} 
is $0$.

\begin{theorem}[B. H. Neumann~\cite{NeuB37}]
Every identity of a finite group with exponent~$e$ is equivalent to $\{ x^e \approx 1, \, \bw \approx 1\}$ for some commutator word~$\bw$\up.
\end{theorem}

A \textit{factor} of a group~$G$ is a quotient of a subgroup of~$G$, that is, $H/K$ where $K\unlhd H\le G$; it is \textit{proper} unless $H=G$ and $K=1$.
%
A \textit{chief factor} is one where $K\unlhd G$ and $H/K$ is a minimal normal subgroup of $G/K$; {\color{black} a \textit{composition factor} is a factor $H/K$, when~$H$ and~$K$ subnormal in~$G$ (that is, terms in a descending series in which each term is normal in its predecessor) and $K$ is a maximal normal subgroup of $H$}.

If~$A$ and~$B$ are subgroups of~$G$, then $[A,B]$ is the subgroup generated by the commutators in $\{ [a,b] \,|\, a\in A, \, b\in B \}$.
The \textit{lower central series} is the descending series $G = G_1 > G_2 > \cdots$ with $G_{i+1}=[G_i,G]$; $G$ is \textit{nilpotent of class~$c$} if $G_{c+1}=1$ (and~$c$ is minimal subject to this).
The \textit{derived series} is the descending series $G=G^{(0)}>G^{(1)}>\cdots$ with $G^{(i+1)}=[G^{(i)},G^{(i)}]$; $G$ is \textit{solvable} \textit{of derived length~$\ell$} if $G^{(\ell)}=1$ (and~$\ell$ is minimal subject to this).

The \textit{product} $\bU\bV$ of varieties~$\bU$ and~$\bV$ consists of all groups~$G$ which are \textit{extensions} of a group $H \in \bU$ by a group $K \in \bV$, that is, $G$ has a normal subgroup isomorphic to~$H$ with quotient isomorphic to~$K$.
The product of two varieties is a variety, and the product operation is associative.
But product varieties are not usually generated by finite groups.

\begin{theorem}[\v{S}mel'ken~\cite{Sme64}]
A product of three or more nontrivial varieties is not generated by a finite group\up.
A product~$\bU\bV$ is generated by some finite group if and only if~$\bU$ and~$\bV$ have coprime exponents\up, $\bU$ is nilpotent\up, and $\bV$ is abelian\up.
\end{theorem}

The variety~$\bU\bV$ has an {\ib} of the form $\bu(\bv_1,\bv_2,\ldots,\bv_n) \approx 1$, where $\bu(x_1,x_2,\ldots,x_n) \approx 1$ is an identity of $\bU$ and each $\bv_i \approx 1$ is an identity of~$\bV$. {\color{black} (Note that, even if for some cases we can do better, usually all identities of $\bV$ are needed, not just an {\ib}.)} 

Many further results about varieties of groups are known, but the interest of the present survey lies in those that are finitely generated.

The most important result about varieties of finite groups is the \textit{Oates--Powell Theorem}, asserting that, for any finite group~$G$, the variety $\var\{G\}$ is finitely based.
Actually it is a little stronger.
A variety of groups is \textit{Cross} if it is finitely based, finitely generated, and small. {\color{black}(Recall that a variety of algebras of any type---in particular, groups---is \textit{finitely generated} if it is generated by one of its finite algebras.)}

\begin{theorem}[Oates and Powell~\cite{OP64}] \label{T: Oates Powell}
The variety generated by any finite group is Cross\up.
\end{theorem}


A group is \textit{critical} if it does not lie in the variety generated by all of its proper factors.
It is known that, if two non-isomorphic critical groups generate the same variety, then they have abelian monoliths.
Hence non-isomorphic finite simple groups generate different varieties.

\subsection{Abelian groups}

The structure of varieties generated by abelian groups is very simple.
The class $\mathbf{A}$ of all abelian groups is the variety defined by the identity $[x,y] \approx 1$; for each integer $m \geq 1$, the class $\mathbf{A}_m$ of abelian groups of exponent~$m$ is the variety defined by the commutator identity and the identity $x^m \approx 1$.
Hence the lattice of varieties of abelian groups is isomorphic to the set of positive integers ordered by divisibility, with a top element added.
We remark that \textsf{GAP} includes commands \verb!IsAbelian! and \verb!Exponent!, so these conditions are easily checked.

Inclusions in the other direction are more problematic.
For sufficiently large~$m$, there are uncountably many varieties of groups covering~$\mathbf{A}_m$ \cite{Koz12,IS04}.

\subsection{Metabelian groups} \label{subsec: metabelian}

A group is \textit{metabelian} if it lies in the product variety $\mathbf{AA}$, that is, it has an abelian normal subgroup with abelian quotient.
Among small groups, many are metabelian; for example, 1,005 of the 1,048 groups of order up to 100 are metabelian.
The smallest non-metabelian groups are the groups~$S_4$ and $\mathrm{SL}(2,3)$ of order~$24$.

A finite metabelian group lies in the variety $\mathbf{A}_m\mathbf{A}_n$ for some $m,n \geq 1$.
The smallest subgroup of a group~$G$ whose quotient is abelian of exponent dividing~$n$ is generated by the $n$th powers and commutators in~$G$, so the variety $\mathbf{A}_m\mathbf{A}_n$ is defined by the identities
\[x^{mn} \approx [x,y]^m \approx [x^n,y^n] \approx \big[x^n,[y,z]\big] \approx \big[[x,y],[z,w]\big] \approx 1.\]
However, finding an {\ib} for individual finite metabelian groups is more difficult.

Higman~\cite{Hig59} showed that for each prime~$p$ and $n \geq 1$, the proper subvarieties of $\mathbf{A}_p\mathbf{A}_n$ containing $\mathbf{A}_{pn}$ are characterized by an identity of the form
\[ [x^n, y^{d_1}, y^{d_2}, \ldots,y^{d_k}] \approx 1,\]
where $d_1>d_2>\cdots>d_k \geq 1$ are divisors of~$n$ such that~$d_i$ does not divide~$d_j$ whenever $i>j$.

As an example which we will examine later, consider the subvariety $\var\{A_4\}$ of $\mathbf{A}_2\mathbf{A}_3$.
The only possible Higman identity is $[x^3,y] \approx 1$, which does not hold in~$A_4$.
Therefore $\var\{A_4\} = \mathbf{A}_2\mathbf{A}_3$.

H. Neumann~\cite[p.179]{NeuH67} quotes a generalization of this, an unpublished result of C.~H.~Houghton according to which, assuming that $\gcd(m,n)=1$, any such variety lies between $\mathbf{A}_{rs}$ and $\mathbf{A}_r\mathbf{A}_s$ for some $r,s \geq 1$ such that $r$ divides~$m$ and $s$ divides~$n$.
Moreover, such a variety is defined by identities of the form \[[x^s,y^{d_1},\ldots,y^{d_k}]^t \approx 1,\] where~$t$ is a divisor of~$r$ and $d_1>d_2>\cdots>d_k \geq 1$ are divisors of~$n$ such that~$d_i$ does not divide~$d_j$ whenever $i>j$.

{\color{black}
Houghton did not publish the proof of his result.
The proof, and a generalization that determines when the equality $\var\{A\}\var\{B\}=\var\{A\wr B\}$ holds for abelian groups~$A$ and~$B$, can be found in Mikaelian~\cite{Mik07}.}

There are also some results for the case when the condition $\gcd(m,n)=1$ is relaxed.

For an example, consider \verb!SmallGroup(12,1)! in \textsf{GAP} with presentation \[ \langle a,b \,|\, a^3=1, \, b^4=1, \, b^{-1}ab=a^2 \rangle. \]
Clearly, this group lies in $\mathbf{A}_3\mathbf{A}_4$ ({\color{black} as $\gcd(m,n)=1$, this group can be handled with Higman's Theorem}), and the possible Higman identities are $[x^4,y] \approx 1$ and $[x^4,y^2] \approx 1$.
It is readily shown that the second is satisfied but the first is not.
Adding $[x^4,y^2] \approx 1$ to the {\ib} we see that the identity $[x^4,y^4] \approx 1$ is now redundant and can be discarded.
Further reductions are possible, but we do not strive for the simplest {\ib}.

{\color{black} A result of Kov\'acs \cite{Kov89} describes the variety generated by a finite
dihedral group. We have restated his theorem in a way which is more useful
for us.

\begin{theorem}[Kov\'acs] Let $D_{2n}$ denote the dihedral group of order
$2n$, where $n=2^dm$ and $m$ is odd.
\begin{enumerate}
\item If $d\le 1$, then $\var(D_{2n})=\mathbf{A}_m\mathbf{A}_2$.
\item If $m=1$ and $d>2$, then
$\var(D_{2n})=\mathbf{A}_{2^{d-1}}\mathbf{A}_2 \cap \mathbf{N}_d$,
where $\mathbf{N}_d$ is the variety of nilpotent groups of class at most $d$.
\item If $m>1$ and $d>2$, then $\var(D_{2n})=\var(D_{2m},D_{2^{d+1}})$.
\end{enumerate}
\end{theorem}

Now it follows from our general remarks on metabelian groups that an {\ib} for $\mathbf{A}_n\mathbf{A}_2$is given by $x^{2n}=[x^2,y^2]=1$.
(For a group lies in this variety if and only if the squares
commute and have orders dividing $n$.) An {\ib}
for $\mathbf{N}_d$ is given by the \emph{left-normed commutator}
$[x_1,x_2,\ldots,x_{d+1}]=1$ (this means
$[[\ldots[[x_1,x_2],x_3],\ldots],x_{d+1}]=1$). Given varieties $\mathbf{V}$ and
$\mathbf{W}$, an {\ib} for $\mathbf{V} \cap \mathbf{W}$ consists of the union of
the {\ibs} for $\mathbf{V}$ and $\mathbf{W}$.
Finally, the identities of $\var(G,H)$ consist of all products of an
identity for $G$ and an identity for $H$. So the identities for varieties of
dihedral groups can be described explicitly.}

\subsection{Other groups}\label{other}

Apart from the above, results about particular finite groups are fairly scarce.
Cossey and Macdonald~\cite{CM68} and Cossey {\etal}.~\cite{CMS70} found explicit {\ibs} for the varieties $\var\{G\}$, where $G \in \{ S_4, A_5, \mathrm{PSL}(2,7) \}$; they also found identities that hold in $\mathrm{PSL}(2,p^m)$ with prime~$p$, but without proof that these identities form an {\ib}.
In the case $p=2$, an {\ib} was found by Southcott~\cite{Sou74}.

Such cases are best dealt with by database lookup.

Description for the identities of the groups $\mathrm{SL}(2,q)$ in some cases---when $q=9$ or $q = p^m$ for some odd prime $p \not\equiv\pm1 \pmod {16}$ and odd $m \geq 1$---are also available.
In these cases, the identities are of the form $[\bw,x] \approx 1$ and $\bw^2 \approx 1$, where $\bw \approx 1$ ranges over an {\ib} for $\mathrm{PSL}(2,q)$ and~$x$ is a variable not occurring in~$\bw$.

In particular, this result holds for $\mathrm{SL}(2,3)$ and $\mathrm{PSL}(2,3) \cong A_4$, where identities of the latter group have been described in Subsection~\ref{subsec: metabelian}.

\subsection{Non-metabelian groups of order~$24$}

As noted earlier, $S_4$ and $\mathrm{SL}(2,3)$ are the only non-metabelian groups of order~$24$.
An {\ib} for the variety $\var\{S_4\}$ can be found in Cossey {\etal}.~\cite{CMS70}: \[x^{12} \approx \big((x^3y^3)^4[x^3,y^6]^3\big)^3 \approx [x^2,y^2]^2 \approx [x,y]^6 \approx [x^6,y^6] \approx \big[[x,y]^3,y^3,y^2\big] \approx 1.\]

The goal of this subsection is to describe the subvarieties of the varieties $\var\{S_4\}$ and $\var\{\mathrm{SL}(2,3)\}$, and to show that their proper subvarieties are all metabelian.


\begin{lemma} \label{L: nonabelian in varS3 implies S3}
Let $G$ be any non-abelian group in $\var\{S_3\}$\up.
Then~$G$ has a subgroup isomorphic to $S_3$\up.
\end{lemma}

\begin{proof}
We know that~$G'$ is a nontrivial elementary abelian 3-group while $G/G'$ is an abelian group that is a direct product of elementary abelian 2-groups and 3-groups.
Since~$G$ is non-abelian, there must be elements $a,b \in G$ that fail to commute.
We consider various cases, assuming that there is no subgroup isomorphic to $S_3$ and aiming for a contradiction.
Note that any two elements of order~$3$ commute, since $[x^2,y^2] \approx 1$ is an identity of $S_3$.
\begin{itemize}\itemsep0pt
\item $a$ and~$b$ have order~$2$.
Then $\langle a,b\rangle$ is a dihedral group of order~$6$ or~$12$ and so contains a subgroup isomorphic to~$S_3$.
So we may assume that involutions commute.
\item $a$ has order~$2$ and~$b$ has order~$3$.
Then $c=b^a$ is another element of order~$3$ and~$c$ commutes with~$b$.
Since $(bc^{-1})^a=cb^{-1}=(bc^{-1})^{-1}$, the subgroup $\langle a,b\rangle$ is isomorphic to~$S_3$.
Hence we can assume that elements of prime orders commute.
\item $a$ has order~$2$ or~$3$ and~$b$ has order~$6$.
Then~$a$ commutes with $b^2$ and $b^3$, and so with~$b$.
\item $a$ and~$b$ have order~$6$.
Then $a^2$ and $a^3$ both commute with~$b$, so that~$a$ and~$b$ commute.
\end{itemize} The proof is thus complete.\end{proof}



\begin{theorem}
Let $G$ be any critical group in $\var\{S_4\}$ that is not metabelian\up.
Then $\var\{G\}=\var\{S_4\}$\up.
\end{theorem}

\begin{proof}
Let~$N$ be the verbal subgroup of~$G$ corresponding to the identities of $\var\{S_3\}$, that is, the subgroup generated by values in~$G$ of the identities of~$S_3$.
Then~$N$ is an elementary abelian 2-group, and it is nontrivial because $1\ne G''\le N$.
Further, $G/N$ belongs to $\var\{S_3\}$.

If~$G/N$ is abelian, then $G' \leq N$, so that the contradiction $G''=1$ is deduced.
Therefore $G/N$ is non-abelian.
Further, $G/N$ has order divisible by~3, since otherwise~$G$ is a 2-group; but 2-groups in $\var\{S_4\}$ belong to $\var\{D_8\}$ and so are metabelian.
Therefore by Lemma~\ref{L: nonabelian in varS3 implies S3}, the group $G/N$ must contain a subgroup~$K$ isomorphic to $S_3$.

Moreover, such a subgroup in $G/N$ cannot centralize~$N$.
For if it did, then $C_G(N)$ (and hence~$G$) would have a normal 3-subgroup; but~$G$ is critical and therefore monolithic (it contains a unique minimal normal subgroup, which is a 2-group) \cite[51.32]{NeuH67}.

An orbit of~$K$ on~$N$ has order at most~6, and so generates a subgroup of order at most~$2^6$.
We show there must be such a subgroup of order~$2^2$.
First, consider the action of an element of order~$3$ in~$K$; let $\{x_1,x_2,x_3\}$ be an orbit.
The subgroup $\langle x_1,x_2,x_3\rangle$ has order~$2^2$ or~$2^3$; in the latter case, the subgroup $\langle x_1x_2,x_2x_3,x_3x_1\rangle$ has order~$2^2$.

If such a subgroup $\{1,y_1,y_2,y_3\}$ of order~$2^2$ is invariant under an element~$t$ of order~$2$ in~$K$, our claim is proved; so suppose not.
Let $z_i=y_i^t$ where $i=1,2,3$.
Then the group generated by the $y$s and $z$s has order $2^4$ and is invariant under $S_3$.
We can assume that conjugation by an element~$u$ of order~$3$ in~$K$ induces the permutation $(y_1,y_2,y_3)(z_1,z_3,z_2)$ (since~$t$ inverts~$u$).
Then the subgroup $\langle y_1z_1,y_2z_3,y_3z_2\rangle$ has order $2^2$ and is $S_3$-invariant.

Now the group generated by~$K$ together with this $K$-invariant subgroup of~$N$ is isomorphic to~$S_4$, and belongs to $\var\{G\}$.
So $\var\{S_4\}\subseteq\var\{G\}$, and we have equality as required.
\end{proof}


\begin{corollary}
Any proper subvariety of $\var\{S_4\}$ is metabelian.
\end{corollary}

The analogous result for $\mathrm{SL}(2,3)$ is similar but easier to establish.
We have noted in {\color{black} Subsection \ref{other}}
that the identities of $\mathrm{SL}(2,3)$ have the form $[\bw,x]\approx\bw^2\approx1$, where $\bw \approx 1$ ranges over the identities of~$A_4$ and~$x$ is a variable not in~$\bw$.

\begin{theorem}
Let~$G$ be any critical group in $\var\{\mathrm{SL}(2,3)\}$ that is not metabelian\up.
Then $\var\{G\}=\var\{\mathrm{SL}(2,3)\}$\up.
\end{theorem}

\begin{proof}
The preliminary result, that a non-abelian group in $\var\{A_4\}$ contains a subgroup isomorphic to~$A_4$, is proved similarly to the analogous result for~$S_3$.

Now let $G\in\var\{\mathrm{SL}(2,3)\}$ and suppose that~$G$ is critical and not metabelian.
Then~$G''$ is an elementary abelian 2-group and is contained in $Z(G)$, so all its subgroups are normal in~$G$.
Since~$G$ is monolithic, we find that $|G''|=2$.
Now $G/G''$ has a subgroup isomorphic to~$A_4$, and it is easy to see that this lifts to a subgroup of~$G$ isomorphic to $\mathrm{SL}(2,3)$.
\end{proof}


\begin{corollary}
Any proper subvariety of $\var\{\mathrm{SL}(2,3)\}$ is metabelian\up.
\end{corollary}

\subsection{Toward an explicit bound} \label{subsec: C(e,m,c)}

It follows from Theorem~\ref{T: Oates Powell}---the Oates--Powell Theorem---that the variety generated by a finite group is finitely based and small.
Can explicit bounds {\color{black} for the orders of critical groups in such a variety} be extracted from the proof of this result?

The proof of the Oates--Powell Theorem rests on three lemmas, of which the third concerns the class $\mathbf{C}(e,m,c)$ of finite groups~$G$ such that
\begin{itemize}
\item $G$ has exponent dividing~$e$;
\item the order of any chief factor of~$G$ is at most~$m$;
\item the nilpotency class of any nilpotent factor of~$G$ is at most~$c$.
\end{itemize}
Then $\mathbf{C}(e,m,c)$ is a class of finite groups in a variety,
whence if $G\in\mathbf{C}(e,m,c)$ then every critical group in $\var\{G\}$
belongs to $\mathbf{C}(e,m,c)$.


\begin{lemma}[H. Neumann~{\cite[52.23]{NeuH67}}]
The class $\mathbf{C}(e,m,c)$ contains only a finite number of \up(non-isomorphic\up) critical groups\up.
\end{lemma}


\begin{lemma}[H. Neumann~{\cite[52.5]{NeuH67}}]
If $G\in\mathbf{C}(e,m,c)$ is critical and has non-abelian monolith\up, then $|G| \le m!$\up.
\end{lemma}

The abelian monolith case is much harder.
Neumann {\color{black}\cite{NeuH67}}
says:
\begin{quote}
\item If a bound for the index of $\Phi(G)$ in~$G$ is found, then a bound for $|G|$ can be derived.
For, since $\Phi(G)$ consists of all non-generators of~$G$, the number of elements needed to generate~$G$ can be at most $|G/\Phi(G)|$.
But from bounds for the number of generators of~$G$ and the index of $\Phi(G)$ in~$G$, one obtains a bound for the number of generators of $\Phi(G)$ by means of Schreier's formula.
As $\Phi(G)$ is nilpotent, of class at most~$c$ and exponent dividing~$e$, this leads to a bound for the order of $\Phi(G)$, and so for the order of~$G$.
\end{quote}

Suppose that we can show that $|G/\Phi(G)|\le b$.
Then~$G$ has at most $\log_2b$ generators, so our bound for the number of generators of $\Phi(G)$ is $(b-1)\log_2b+1$, or in broad brush terms, $d\le b\log b$.
This gives a bound for the order of $\Phi(G)$ which is roughly $e^{d+d^2+\cdots+d^c}$, since the lower central factors are generated by commutators.

A small improvement is possible.
If $\Phi(G)$ is not a $p$-group, then it is the direct product of its Sylow $p$-subgroups, each of which contains a nontrivial normal subgroup of~$G$, contradicting the fact that~$G$ is monolithic.
So we can replace~$e$ in the above bound by the largest prime divisor of~$e$.

Continuing, the proof considers a series
\[\Phi(G)<F<C<G,\]
and shows that $|G/C|\le(m!)^c$ and $|F/\Phi(G)|\le m^c$, while $|C/F|\le(m!)^t$, where $t\le 1+ce(m!)$.
The bound for~$b$ is the product of these numbers.

Even for very moderate values of~$e$, $m$, and~$c$, the resulting bound is going to be rather large!

\section{The database of varieties generated by small semigroups}\label{smallsemigroups}

\subsection{The library of varieties generated by a semigroup of order up to~$5$}

We produced a database containing all the semigroups up to order~$5$ and an {\ib} for the variety generated by each of them.
All the proofs regarding semigroups up to order~$4$ appear (or are referred to) in this paper.
The proofs regarding semigroups of order~$5$ will be published elsewhere.

\begin{table}[h]
\centering
\begin{tabular}{r|r|r|r}
\multicolumn{1}{c|}{$n$} & \multicolumn{1}{c|}{\begin{tabular}[c]{@{}c@{}}Number of semi-\\groups of order~$n$,\\ up to equivalence \end{tabular}} & \multicolumn{1}{c|}{\begin{tabular}[c]{@{}c@{}}Number of semi-\\groups of order~$n$,\\up to isomorphism\end{tabular}} & \multicolumn{1}{c}{\begin{tabular}[c]{@{}c@{}}Number of varieties\\ with a primitive\\generator of order~$n$\end{tabular}} \\
\hline
1 & 1 & 1 & 1 \\
2 & 4 & 5 & 5 \\
3 & 18 & 24 & 14 \\
4 & 126 & 188 & 53 \\
5 & 1,160 & 1,915 & 145 \\
6 & 15,973 & 28,634 & At least 461 \\
7 & 836,021 & 1,627,672 & Unknown \\
8 & 1,843,120,128 & 3,684,030,417 & Unknown \\
9 & 52,989,400,714,478 & 105,978,177,936,292 & Unknown \\
\end{tabular}
\caption{Some numerical data}
\end{table}

\subsection{Non-finitely based varieties generated by a semigroup of order~$6$} \label{subsec: NFB order 6}

Every variety generated by a semigroup of order up to~$5$ is finitely based~\cite{Lee13,Tra83,Tra91}.
Among all varieties generated by a semigroup of order~$6$, precisely four are non-finitely based~\cite{LL11,LZ15}; these varieties are generated by the following semigroups:
\begin{itemize}
\item the monoid~$\Btwo^1$ obtained from the Brandt semigroup \[ \Btwo = \langle a,b \,|\, a^2=b^2=0,\,aba=a,\,bab=b \rangle = \{ 0,a,b,ab,ba\}; \]
\item the monoid~$\Atwo^1$ obtained from the 0-simple semigroup \[ \Atwo = \langle a,b \,|\, a^2=aba=a,\,bab=b,\,b^2=0 \rangle = \{ 0,a,b,ab,ba\}; \] 
\item the semigroup~$\Atwo^g$ obtained by adjoining a new element~$g$ to~$\Atwo$ with $g^2 = 0$ and $g\Atwo = \Atwo g = \{g\}$;
\item the $\mathscr{J}$-trivial semigroup \[ \nfbL = \langle a,b \,|\, a^2=a, \, b^2=b, \, aba=0 \rangle = \{ 0,a,b,ab,ba,bab\}. \]
\end{itemize}
The Cayley tables of these semigroups are given in Table~\ref{Tab: NFB order 6};
refer to Lee {\etal}.~\cite{LLZ12} for more historical information on their discovery.

\begin{table}[ht] \centering \addtolength{\tabcolsep}{-2pt}
\begin{tabular}[t]{c|cccccc}
$\Btwo^1$ & 1 & 2 & 3 & 4 & 5 & 6 \\ \hline
        1 & 1 & 1 & 1 & 1 & 1 & 1 \\
        2 & 1 & 1 & 1 & 2 & 2 & 3 \\
        3 & 1 & 2 & 3 & 1 & 3 & 1 \\
        4 & 1 & 1 & 1 & 4 & 4 & 6 \\
        5 & 1 & 2 & 3 & 4 & 5 & 6 \\
        6 & 1 & 4 & 6 & 1 & 6 & 1
\end{tabular}
\quad
\begin{tabular}[t]{c|cccccc}
$\Atwo^1$ & 1 & 2 & 3 & 4 & 5 & 6 \\ \hline
        1 & 1 & 1 & 1 & 1 & 1 & 1 \\
        2 & 1 & 1 & 1 & 2 & 2 & 3 \\
        3 & 1 & 2 & 3 & 2 & 3 & 3 \\
        4 & 1 & 1 & 1 & 4 & 4 & 6 \\
        5 & 1 & 2 & 3 & 4 & 5 & 6 \\
        6 & 1 & 4 & 6 & 4 & 6 & 6
\end{tabular}
\\[0.08in]
\begin{tabular}[t]{c|cccccc}
$\Atwo^g$ & 1 & 2 & 3 & 4 & 5 & 6 \\ \hline
        1 & 1 & 1 & 1 & 1 & 1 & 6 \\
        2 & 1 & 1 & 1 & 2 & 3 & 6 \\
        3 & 1 & 2 & 3 & 2 & 3 & 6 \\
        4 & 1 & 1 & 1 & 4 & 5 & 6 \\
        5 & 1 & 4 & 5 & 4 & 5 & 6 \\
        6 & 6 & 6 & 6 & 6 & 6 & 1
\end{tabular}
\quad
\begin{tabular}[t]{c|cccccc}
$\nfbL$ & 1 & 2 & 3 & 4 & 5 & 6 \\ \hline
      1 & 1 & 1 & 1 & 1 & 1 & 1 \\
      2 & 1 & 1 & 1 & 1 & 1 & 2 \\
      3 & 1 & 1 & 1 & 1 & 1 & 3 \\
      4 & 1 & 1 & 2 & 1 & 4 & 2 \\
      5 & 1 & 1 & 3 & 1 & 5 & 3 \\
      6 & 1 & 2 & 2 & 4 & 4 & 6
\end{tabular}
\caption{Non-finitely based semigroups of order~$6$}
\label{Tab: NFB order 6}
\end{table}

Besides the four non-finitely based semigroups of order six, many other non-finitely based finite semigroups have been discovered since the 1970s; see the survey by Volkov~\cite{Vol01}.
But explicit {\ibs} have not been found for varieties generated by most of these semigroups because the task is neither necessary (in establishing the non-finite basis property) nor trivial.
Nevertheless, explicit {\ibs} are available for a few  non-finitely based varieties.

\begin{proposition}[Jackson~{\cite[Proposition~4.1]{Jac05a}}]
The identities
\begin{gather*}
x^4 \approx x^3, \quad x^3y \approx yx^3, \quad x^2yx \approx x^3y, \quad xyx^2 \approx x^3y, \quad xyxzx \approx x^3yz, \\
\bigg(\prod_{i=1}^m x_i \bigg) \bigg(\prod_{i=m}^1 x_i \bigg) y^2 \approx y^2 \bigg(\prod_{i=1}^m x_i \bigg) \bigg(\prod_{i=m}^1 x_i \bigg), \quad m= 1,2,3,\ldots
\end{gather*}
constitute an {\ib} for a non-finitely based variety generated by a certain semigroup of order~$211$\up.
\end{proposition}

\begin{proposition}[Lee and Volkov~{\cite[Section~1]{LV11}}]
For each $n \geq 2$\up, the identities
\begin{gather*}
x^{n+2} \approx x^2, \quad (xy)^{n+1}x \approx xyx, \quad xyxzx \approx xzxyx, \\
\bigg(\prod_{i=1}^m x_i^n \bigg)^3 \approx \bigg(\prod_{i=1}^m x_i^n \bigg)^2, \quad m = 2,3,4,\ldots
\end{gather*}
constitute an {\ib} for the non-finitely based variety $\var\{\Atwo,\mathbb{Z}_n\}$\up.
In particular\up, $\var\{\Atwo,\mathbb{Z}_2\} = \var\{\Atwo^g\}$\up.
\end{proposition}

\begin{proposition}[Lee~{\cite[Corollary~3.5]{Lee15}}]
For each $n \geq 1$\up, the identities
\begin{gather*}
x^{n+2} \approx x^2, \quad x^{n+1}yx^{n+1} \approx xyx, \quad xhykxty \approx yhxkytx, \\
x \bigg(\prod_{i=1}^m (y_ih_iy_i) \bigg) x \approx x \bigg(\prod_{i=m}^1 (y_ih_iy_i) \bigg) x, \quad m= 2,3,4,\ldots
\end{gather*}
constitute an {\ib} for the non-finitely based variety $\var\{\nfbL,\mathbb{Z}_n\}$\up.
\end{proposition}

\subsection{Inherently non-finitely based finite semigroups} \label{subsec: INFB}

The \textit{finite basis problem}---first posed by Tarski~\cite{Tar68} in the 1960s as a decision problem---questions which finite algebras are finitely based.
This problem is undecidable for general algebras~\cite{Mck96} but remains open for finite semigroups.
In contrast, it is decidable if a finite semigroup~$S$ is \textit{\infb} in the sense that every locally finite variety containing~$S$ is non-finitely based.
This result follows from the work of Sapir~\cite{Sap87a,Sap87b}, a description of which requires the following concepts:
\begin{itemize}
\item the \textit{period} of a semigroup~$S$ is the least number~$d$ such that~$S$ satisfies the identity $x^{m+d} \approx x^m$ for some $m \geq 1$;
\item the \textit{upper hypercenter} of a group~$G$, denoted by $\Gamma(G)$, is the last term in the upper central series of~$G$;
\item a word~$\bw$ is an \textit{isoterm} for a semigroup~$S$ if~$S$ violates every nontrivial identity of the form $\bw \approx \bw'$;
\item the \textit{Zimin words} $\bz_1, \bz_2, \bz_3, \ldots$ are words over the variables $\{x_1,x_2,x_3,\ldots \}$ defined inductively by $\bz_1 = x_1$ and $\bz_{k+1} = \bz_kx_{k+1} \bz_k$ for each $k \geq 1$.
\end{itemize}

\begin{theorem}[Sapir~{\cite[Theorem~3.6.34]{Sap14}}] \label{T: INFB Sapir}
\begin{enumerate}[\ \rm(i)]
\item A finite semigroup~$S$ is {\infb} if and only if there exists some idempotent $e \in S$ such that the submonoid $eSe$ of~$S$ is {\infb}\up.
\item A finite monoid~$M$ with period~$d$ is {\infb} if and only if there exist $a \in M$ and an idempotent $e \in MaM$ such that the elements $eae$ and $ea^{d+1}e$ do not belong to the same coset of the maximal subgroup~$M_e$ of~$M$ containing~$e$ with respect to the upper hypercenter $\Gamma(M_e)$\up.
\item A finite semigroup~$S$ is {\infb} if and only if the Zimin words $\bz_1,\bz_2,\ldots, \bz_m$\up, where $m = |S|^3$\up, are isoterms for~$S$\up.
\end{enumerate}
\end{theorem}

The non-finitely based semigroups~$\Atwo^g$ and~$\nfbL$ are not {\infb} because they satisfy the identities $\bz_2 \approx x_1(x_2x_1)^3$ and $\bz_2 \approx x_1x_2x_1^2$, respectively.
On the other hand, the semigroups~$\Atwo^1$ and~$\Btwo^1$ are {\infb} since all Zimin words are isoterms \cite[Lemma~3.7]{Sap87b}.
It follows that a finite semigroup~$S$ is {\infb} if either $\Atwo^1 \in \var\{S\}$ or $\Btwo^1 \in \var\{S\}$.
Observe that the condition in Theorem~\ref{T: INFB Sapir}(ii) can hold in a trivial way, namely when $eae$ or $ea^{d+1}e$ does not belong to $M_e$, so that both elements do not belong to the same coset of $M_e$.
This is the case for $\Btwo^1$; see, for example, Volkov and Gol'berg \cite[observation after Proposition 1]{VG03}.

For certain finite monoids~$M$, the condition $\Btwo^1 \in \var\{M\}$ is not only sufficient, but also necessary for~$M$ to be {\infb}.

\begin{lemma} \label{L: INFB 4 conditions}
Let~$M$ be any finite monoid that satisfies the identity $x^{2n} \approx x^n$ for some $n \geq 2$\up.
Suppose that~$M$ satisfies at least one of the following four conditions\up: $|M| \leq 55$\up, $M$ is regular\up, the idempotents of~$M$ form a submonoid\up, and all subgroups of~$M$ are nilpotent\up.
Then the following conditions are equivalent\up:
\begin{enumerate}[\ \ \rm(a)]
\item $M$ is {\infb}\up;
\item $\Btwo^1 \in \var\{M\}$\up;
\item $M$ violates the identity \begin{equation} \big((xy)^n(yx)^n(xy)^n\big)^n \approx (xy)^n. \label{id: excl B2} \end{equation}
\end{enumerate}
\end{lemma}

\begin{proof}
${\rm(a)\Leftrightarrow\rm(b)}$:
This holds by Jackson~\cite[Theorems~1.4 and~2.2]{Jac02} and Sapir~\cite[Theorem~2]{Sap87a}.

\noindent${\rm(c)\Rightarrow\rm(b)}$:
If~$M$ violates the identity~\eqref{id: excl B2}, then $\Btwo \in \var\{M\}$ by Sapir and Suhanov \cite[Theorem~1]{SS81}, so that $\Btwo^1 \in \var\{M\}$ by Jackson \cite[Lemma~1.1]{Jac05b}.

\noindent${\rm(b)\Rightarrow\rm(c)}$:
It is routinely verified that~$\Btwo^1$ violates the identity~\eqref{id: excl B2}.
Therefore if~$M$ satisfies the identity~\eqref{id: excl B2}, then $\Btwo^1 \notin \var\{M\}$.
\end{proof}

There is yet another method to check if a finite monoid is {\infb}.
For each $n \geq 2$, define the words $[x,y]^n_1,[x,y]^n_2,[x,y]^n_3,\ldots$ over $\{x,y\}$ inductively by $[x,y]^n_1 = x^{n-1}y^{n-1}xy$ and $[x,y]^n_{k+1} = \big[[x,y]^n_k,y\big]^n_1$ for each $k \geq 1$.
Then for any variety~$\bV$ generated by a finite semigroup that satisfies the identity $x^{2n} \approx x^n$, the subsequence $\{ [x,y]^n_{k!} \}$ converges in the $\bV$-free semigroup over $\{x,y\}$; let $[x,y]^n_\infty$ denote the limit of this subsequence \cite[Subsection~4.4]{Vol00}.

\begin{lemma}[Volkov~{\cite[Proposition~4.4]{Vol00}}] \label{L: INFB Volkov}
Let~$M$ be any finite monoid that satisfies the identity $x^{2n} \approx x^n$ for some $n \geq 2$\up.
Then~$M$ is {\infb} if and only if it violates either~\eqref{id: excl B2} or \[ [\mathbf{e}z\mathbf{e},(\mathbf{e}y\mathbf{e})^{n-1}\mathbf{e}y^{n+1}\mathbf{e}]^n_\infty \approx \mathbf{e} \quad \text{with } \mathbf{e}=(xyzt)^n. \]
\end{lemma}

The companion website checks if an input finite semigroup~$S$ is {\infb} in the following manner.
Suppose that $e_1, e_2, \ldots, e_r$ are all the idempotents of~$S$.
Then by Theorem~\ref{T: INFB Sapir}(i), it suffices to check if some submonoid ${ M_i = e_iSe_i }$ of~$S$ is {\infb}; this can be achieved by applying Theorem~\ref{T: INFB Sapir}(ii).
As this is the most general result, the website can handle semigroups of order higher than~$55$; if the semigroup is {\infb}, then the website provides the relevant information such as the hypercenter. 
The website also allows the user to check if a semigroup is {\infb} with Lemma~\ref{L: INFB 4 conditions}. 
Results on isoterms are computationally demanding and hence are not used.

Refer to the surveys by Volkov~\cite{Vol00,Vol01} for more information on {\infb} semigroups and the finite basis problem for finite semigroups in general.

Based on results in this subsection, a description of {\infb} semigroups of order up to~$9$ is possible. 
For this purpose, the semigroup~$\Atwo^1$ and~$\Btwo^1$, together with those given in Tables~\ref{Tab: infb order 7}--\ref{Tab: infb order 9}, are required.

\begin{table}[ht]  \centering \addtolength{\tabcolsep}{-2pt}
\begin{tabular}[t]{c|ccccccc}
$U_7$ & 1 & 2 & 3 & 4 & 5 & 6 & 7 \\ \hline
    1 & 1 & 1 & 1 & 1 & 1 & 1 & 1 \\
    2 & 1 & 1 & 1 & 1 & 2 & 2 & 3 \\
    3 & 1 & 2 & 3 & 1 & 1 & 3 & 1 \\
    4 & 4 & 4 & 4 & 4 & 4 & 4 & 4 \\
    5 & 4 & 4 & 4 & 4 & 5 & 5 & 7 \\
    6 & 1 & 2 & 3 & 4 & 5 & 6 & 7 \\
    7 & 4 & 5 & 7 & 4 & 4 & 7 & 4
\end{tabular}
\
\begin{tabular}[t]{c|ccccccc}
$V_7$ & 1 & 2 & 3 & 4 & 5 & 6 & 7 \\ \hline
    1 & 1 & 1 & 1 & 1 & 1 & 1 & 1 \\
    2 & 1 & 1 & 1 & 1 & 2 & 2 & 3 \\
    3 & 1 & 2 & 3 & 1 & 2 & 3 & 3 \\
    4 & 4 & 4 & 4 & 4 & 4 & 4 & 4 \\
    5 & 4 & 4 & 4 & 4 & 5 & 5 & 7 \\
    6 & 1 & 2 & 3 & 4 & 5 & 6 & 7 \\
    7 & 4 & 5 & 7 & 4 & 5 & 7 & 7
\end{tabular}
\
\begin{tabular}[t]{c|ccccccc}
$W_7$ & 1 & 2 & 3 & 4 & 5 & 6 & 7 \\ \hline
    1 & 1 & 1 & 1 & 1 & 5 & 5 & 5 \\
    2 & 1 & 2 & 1 & 2 & 5 & 5 & 7 \\
    3 & 1 & 1 & 3 & 3 & 5 & 6 & 5 \\
    4 & 1 & 2 & 3 & 4 & 5 & 6 & 7 \\
    5 & 5 & 5 & 5 & 5 & 1 & 1 & 1 \\
    6 & 5 & 6 & 5 & 6 & 1 & 1 & 3 \\
    7 & 5 & 5 & 7 & 7 & 1 & 2 & 1
\end{tabular}
\caption{The semigroups $U_7$, $V_7$, and $W_7$}
\label{Tab: infb order 7}
\end{table}

\begin{table}[ht]  \centering \addtolength{\tabcolsep}{-2pt}
\begin{tabular}[t]{c|cccccccc}
$U_8$ & 1 & 2 & 3 & 4 & 5 & 6 & 7 & 8 \\ \hline
    1 & 1 & 1 & 1 & 1 & 1 & 1 & 1 & 1 \\
    2 & 1 & 1 & 1 & 1 & 2 & 2 & 3 & 4 \\
    3 & 1 & 2 & 3 & 4 & 3 & 4 & 4 & 4 \\
    4 & 4 & 4 & 4 & 4 & 4 & 4 & 4 & 4 \\
    5 & 1 & 2 & 3 & 4 & 5 & 6 & 7 & 8 \\
    6 & 4 & 4 & 4 & 4 & 6 & 6 & 7 & 8 \\
    7 & 4 & 6 & 7 & 8 & 7 & 8 & 8 & 8 \\
    8 & 8 & 8 & 8 & 8 & 8 & 8 & 8 & 8
\end{tabular}
\quad
\begin{tabular}[t]{c|cccccccc}
$V_8$ & 1 & 2 & 3 & 4 & 5 & 6 & 7 & 8 \\ \hline
    1 & 1 & 1 & 1 & 1 & 5 & 5 & 7 & 7 \\
    2 & 1 & 2 & 1 & 2 & 5 & 5 & 7 & 8 \\
    3 & 1 & 1 & 3 & 3 & 5 & 6 & 7 & 7 \\
    4 & 1 & 2 & 3 & 4 & 5 & 6 & 7 & 8 \\
    5 & 5 & 5 & 5 & 5 & 7 & 7 & 1 & 1 \\
    6 & 5 & 6 & 5 & 6 & 7 & 7 & 1 & 3 \\
    7 & 7 & 7 & 7 & 7 & 1 & 1 & 5 & 5 \\
    8 & 7 & 7 & 8 & 8 & 1 & 2 & 5 & 5
\end{tabular} 
\caption{The semigroups $U_8$ and $V_8$}
\label{Tab: infb order 8}
\end{table}

\begin{table}[ht] \centering \addtolength{\tabcolsep}{-2pt}
\begin{tabular}[t]{c|ccccccccc}
$U_9$ & 1 & 2 & 3 & 4 & 5 & 6 & 7 & 8 & 9 \\ \hline
    1 & 1 & 1 & 1 & 1 & 1 & 1 & 1 & 1 & 1 \\
    2 & 1 & 1 & 1 & 1 & 1 & 2 & 2 & 3 & 4 \\
    3 & 1 & 2 & 3 & 4 & 4 & 3 & 4 & 4 & 4 \\
    4 & 4 & 4 & 4 & 4 & 4 & 4 & 4 & 4 & 4 \\
    5 & 5 & 5 & 5 & 5 & 5 & 5 & 5 & 5 & 5 \\
    6 & 1 & 2 & 3 & 4 & 5 & 6 & 7 & 8 & 9 \\
    7 & 5 & 5 & 5 & 5 & 5 & 7 & 7 & 8 & 9 \\
    8 & 5 & 7 & 8 & 9 & 9 & 8 & 9 & 9 & 9 \\
    9 & 9 & 9 & 9 & 9 & 9 & 9 & 9 & 9 & 9
\end{tabular}
\quad
\begin{tabular}[t]{c|ccccccccc}
$V_9$ & 1 & 2 & 3 & 4 & 5 & 6 & 7 & 8 & 9 \\ \hline
    1 & 1 & 1 & 1 & 1 & 1 & 6 & 6 & 6 & 6 \\
    2 & 1 & 1 & 1 & 2 & 2 & 6 & 6 & 6 & 7 \\
    3 & 3 & 3 & 3 & 3 & 3 & 8 & 8 & 8 & 8 \\
    4 & 1 & 2 & 3 & 4 & 5 & 6 & 7 & 8 & 9 \\
    5 & 3 & 3 & 3 & 5 & 5 & 8 & 8 & 8 & 9 \\
    6 & 1 & 1 & 1 & 6 & 1 & 6 & 6 & 6 & 6 \\
    7 & 1 & 2 & 1 & 7 & 1 & 6 & 7 & 6 & 6 \\
    8 & 3 & 3 & 3 & 8 & 3 & 8 & 8 & 8 & 8 \\
    9 & 3 & 5 & 3 & 9 & 3 & 8 & 9 & 8 & 8
\end{tabular}
\\[0.1in]
\begin{tabular}[t]{c|ccccccccc}
$W_9$ & 1 & 2 & 3 & 4 & 5 & 6 & 7 & 8 & 9 \\ \hline
    1 & 1 & 1 & 1 & 1 & 1 & 6 & 6 & 6 & 6 \\
    2 & 1 & 1 & 1 & 2 & 2 & 6 & 6 & 6 & 7 \\
    3 & 3 & 3 & 3 & 3 & 3 & 8 & 8 & 8 & 8 \\
    4 & 1 & 2 & 3 & 4 & 5 & 6 & 7 & 8 & 9 \\
    5 & 3 & 3 & 3 & 5 & 5 & 8 & 8 & 8 & 9 \\
    6 & 1 & 1 & 1 & 6 & 1 & 6 & 6 & 6 & 6 \\
    7 & 1 & 2 & 1 & 7 & 2 & 6 & 7 & 6 & 7 \\
    8 & 3 & 3 & 3 & 8 & 3 & 8 & 8 & 8 & 8 \\
    9 & 3 & 5 & 3 & 9 & 5 & 8 & 9 & 8 & 9
\end{tabular}
\caption{The semigroups $U_9$, $V_9$, and $W_9$}
\label{Tab: infb order 9}
\end{table}

Since the semigroups in Tables~\ref{Tab: infb order 7}--\ref{Tab: infb order 9} are monoids, it is routinely checked by Lemma~\ref{L: INFB 4 conditions} that they are all {\infb}.
With the exception of~$V_7$ and~$U_8$, each of these semigroups is isomorphic to its dual.

\begin{proposition} \label{P: infb up to 9}
Let $S$ be any {\infb} semigroup of order~$9$ or less\up.
\begin{enumerate}[\ \rm(i)]
\item If $|S| \leq 6$\up, then~$S$ is isomorphic to one of the semigroups~$\Atwo^1$ and~$\Btwo^1$\up. 

\item If $|S| = 7$\up, then either~$S$ contains~$\Atwo^1$ or $\Btwo^1$ as a subsemigroup or~$S$ is isomorphic to one of the semigroups~$U_7$\up, $V_7$\up, $\dual{V}_7$\up, and~$W_7$\up.

\item If $|S| = 8$\up, then either~$S$ contains a proper subsemigroup that is {\infb} or~$S$ is isomorphic to one of the semigroups~$U_8$\up, $\dual{U}_8$\up, and~$V_8$\up.

\item If $|S| = 9$ and~$S$ satisfies the identity $x^4 \approx x^2$\up, then either~$S$ contains a proper subsemigroup that is {\infb} or~$S$ is isomorphic to one of the semigroups~$U_9$\up, $\dual{U}_9$\up, $V_9$\up, and~$W_9$\up.
\end{enumerate}
\end{proposition}

It is long and well known that the semigroups~$\Atwo^1$ and~$\Btwo^1$ of order~$6$ are the smallest {\infb} semigroups.
GAP's package SmallSemi contains all the semigroups of order up to~$8$ and hence we could routinely run the algorithm {\color{black} outlined after Lemma \ref{L: INFB Volkov}}.

To find {\infb} semigroups of order~$9$, we used the following algorithm (which in fact uses different results and computations to double check Proposition~\ref{P: infb up to 9} parts~(ii) and (iii)):
\begin{enumerate}
\item Use Mace4 \cite{prover} to generate all monoids of orders $6$--$9$ that satisfy the identity $x^4 \approx x^2$ but violate the identity \eqref{id: excl B2}, thus resorting to Lemma~\ref{L: INFB 4 conditions}; this led to 457,745 semigroups.

\item Use Isofilter to discard isomorphic copies; this led to 7,625 semigroups which are all {\infb}, but many of which contain proper subsemigroups that are {\infb}.

\item Use GAP's SmallSemi to discard the semigroups of order $n\in\{7,8,9\}$ that contain a proper subsemigroup that is {\infb}; this left us with the semigroups in Tables~\ref{Tab: infb order 7}--\ref{Tab: infb order 9}.
\end{enumerate}



\clearpage


\section{The companion webpage}\label{companion}

In this section we will give some brief details on the architecture of the website.

\subsection{Multiplication table}

A very flexible data entry tool was developed to allow the input of a multiplication table of a semigroup~$S$.
By default the elements of the semigroup are assumed to be $1,2,\ldots, \mathrm{N}$.
This is convenient to use the multiplication tables coming from {\GAP}.
Some other computational tools use the elements $0,1,\ldots, \mathrm{N}-1$, and this can also be used, along with sets on different (given) elements.

The entries of the Cayley table can be separated by commas or spaces, and optionally can include $[\cdot]$ to bound each line and/or the full multiplication table.
If the elements are all single-digit, all or part of the separators can be omitted.
For instance, all input strings below can be used as input for the same multiplication table:\\
~\\

\begin{longtable}[]{@{}ll@{}}
\toprule
\endhead
1 1 1 1 1 1 1 1 2 & space separated\tabularnewline
1,1,1,1,1,1,1,1,2 & comma separated\tabularnewline
1, 1, 1, 1, 1, 1, 1, 1, 2 & mixed commas and spaces\tabularnewline
{[}1, 1, 1, 1, 1, 1, 1, 1, 2{]} & "{[}" and "{]}"
enclosed\tabularnewline
{[} {[} 1, 1, 1 {]}, {[} 1, 1, 1 {]}, {[} 1, 1, 2 {]} {]} & {\GAP}
syntax\tabularnewline
111 111 112 & separators omitted (only for single digit
elements)\tabularnewline
111111112 & separators omitted (only for single digit
elements)\tabularnewline
\bottomrule
\end{longtable}

Using the {\GAP} syntax option, it is possible to copy a multiplication table from {\GAP} and paste it here.
For example, we can just copy and paste the output of {\GAP} coming from the following command:\\
~\\
gap\textgreater{} MultiplicationTable(SmallGroup(5,1));\\
{[} {[} 1, 2, 3, 4, 5 {]}, {[} 2, 3, 4, 5, 1 {]}, {[} 3, 4, 5, 1, 2 {]},
{[} 4, 5, 1, 2, 3 {]}, {[} 5, 1, 2, 3, 4 {]} {]}\\
~\\
The number of the multiplication table entries must be a perfect square, otherwise an error will be returned.
Only semigroups will be accepted, so the associativity property is checked by default.


Semigroups up to order~$100$ are accepted, but the representative in the isomorphism class of~$S$, whose vector $\vect(S)$ is lexicographically the least, will only be computed in case the order of~$S$ is~$10$ or less.

%
%
%

\subsection{Finding the least semigroup of its isomorphism class}

Finding the semigroup~$S$ in its isomorphism class whose vector $\vect(S)$ is the least lexicographically is not necessary to access the main tools available on the website;
however, it is much more convenient and an essential part of the way we name varieties.

An obvious algorithm would be to give to some model builder, such as Mace4, the Cayley table of the semigroup and ask for all the isomorphic models in the same underlying set.
This gives a list of vectors that we only need to order.

We decided to use our own algorithm that proved to deliver the result for semigroups of order up to~$10$ in less than a second, and that we now outline.

\subsubsection {The presentation to semigroup algorithm} \quad

\textbf{Input:} order, mtable: order and multiplication table of a semigroup.

\textbf{Output:} minlex: multiplication table of the least (lexicographically) semigroup isomorphic to the given semigroup.

\begin{tabbing}
\hspace*{0cm} \= \hspace{0.5cm} \= \hspace{1ex}
 \= \hspace{1ex} \= \hspace{1ex} \= \hspace{1ex} \= \hspace{1ex}
 \= \hspace{1ex} \= \hspace{1ex} \= \kill

routine Minlex (order mtable):\\

\>01:\>minlex = mtable\\
\>02:\>for i in 1 to order\\
\>03:\>\>newElem[i] = i\\
\>04:\>for x in order-permutations of order\\
\>05:\>\>for i = 1 to order\\
\>06:\>\>\>newElem[x[i]] = i\\
\>07:\>\>equal = True\\
\>08:\>\>stop = False\\
\>09:\>\>smaller = False\\
\>10:\>\>if newElem[mtable[x[1]][x[1]]] = 1\\
\>11:\>\>\>for l = 1 to order\\
\>12:\>\>\>\>for c = 1 to order\\
\>13:\>\>\>\>\>e = newElem[mtable[x[l]][x[c]]]\\
\>14:\>\>\>\>\>e0 = minlex[l][c]\\
\>15:\>\>\>\>\>if equal = True\\
\>16:\>\>\>\>\>\>if e $>$ e0\\
\>17:\>\>\>\>\>\>\>stop = True\\
\>18:\>\>\>\>\>\>\>exit for loop\\
\>19:\>\>\>\>\>\>else if e $<$ e0\\
\>20:\>\>\>\>\>\>\>equal = False\\
\>21:\>\>\>\>\>\>\>menor = True\\
\>22:\>\>\>\>\>a1[l][c] = e\\
\>23:\>\>\>\>if stop = True\\
\>24:\>\>\>\>\>exit for loop\\
\>25:\>\>if smaller == True\\
\>26:\>\>\>minlex = a1\\
\>27:\>return minlex\\
\end{tabbing}


\subsection{Generating a semigroup from a given presentation}

The presentation tool finds the multiplication table from a presentation. One of the distinctive features of this tool is that it allows to define infinitely many different presentations (semigroups, bands, etc.) defined as varieties or quasi-varieties. The  presentation (both theory and relations) must  be written in Prover9 syntax. A presentation has two ingredients: the theory and some  relations between the generators.
~\\
To specify the identities that define the theory and the relations, a
subset of Prover9 syntax is used:

\begin{itemize}
\item
  Variables (with names started by ``u'', ``v'', ``w'', ``x'', ``y'' and
  ``z"). No variables will be allowed at the {\em relations} window;
\item
  Constants (with names started with a $0-9$, $a-s$, or $A-Z$);
\item
  Binary operation character $*$;
\item
  Equal sign $=$;
\item
  Parentheses $($ and $)$;
\item
  Each identity must end with a final mark.
\end{itemize}

\textbf{Examples:}\\
~\\
Consider the following example presentations, and how to enter the 
corresponding theory:\\
~\\

\begin{longtable}[]{@{}lll@{}}
\toprule
Presentation & Theory & Relations\tabularnewline
\midrule
\endhead

$\langle a,e|ea^2=a^2,e^2=ae=e\rangle$  & $x*(y*z)=(x*y)*z.$ & $(e*a)*a=a*a.$       \\
$=\{a,e,a^2,ea\}$          &                    & $e*e=a*e.$           \\
                           &                    & $a*e=e.$             \\ \hline

$\langle a|a^5=1\rangle$                & $x*(y*z)=(x*y)*z.$ & $(((a*a)*a)*a)*a=1.$ \\
$=\{a,a^2,a^3,a^4,1\}$     & $x*1=x.$ $1*x=x.$  &                      \\ \hline

$ \langle a,e|ae=0,ea=a,e^2=e\rangle$    & $x*(y*z)=(x*y)*z.$ & $a*e=0.$             \\
$\cup$ $\{1\}=\{0,a,e,1\}$ & $x*0=0.$ $0*x=0.$  & $e*a=a.$             \\
                           & $x*1=x.$ $1*x=x.$  & $e*e=e.$             \\

\bottomrule
\end{longtable}

The tool will try to close the
multiplication table, but if more than 20 elements are reached, an error
will be returned.

Entering a semigroup as a presentation (or using given identities to find or filter varieties) demands the use of an automated theorem prover (in this site Prover/Mace4), something usually very expensive (in time). Therefore a  strategy to limit calls and also to speed-up the use of Prover9/Mace4 was implemented (see Table \ref{strategy}). 

\begin{table}[h]
\centering
\begin{tabular}{l|l|l}
\# & Step & Description \\
\hline\hline
&&                                  User enters a presentation in Prover9/\\
1 & Presentation &                  Mace4 format (both the theory and\\
&&                                  relations).\\
\hline
&&                                  User formulas are normalized to\\
2 & Normalization &                 a internal notation and ordering\\
&&                                  rules, to increase cache's  hit rate.\\
\hline
&&                                  If a similar presentation (in\\
3 & Presentation cache (SQL) &      normalized notation) is recorded\\
&&                                  in SQL, its result will be used.\\
\hline
&&                                  If the user had requested other similar\\
4 & Proofs cache (user session) &   proofs during the session, the results are\\
&&                                  used to reduce the number of proofs.\\
\hline
&&                                  If all users had requested other similar\\
5 & Proofs cache (SQL) &            proofs recorded in SQL, theirs result\\
&&                                  will be used to speed the process.\\
\hline
&&                                  Launched at the same time, but the \\
6 & Launch Prover9/Mace4 &    first to      find a proof or counterexample\\
&&                                  (respectively) stops the other.\\
\end{tabular}
\caption{Presentations algorithm}\label{strategy}
\end{table}

\subsection{Finding an {\ib} for a finitely generated variety} \label{subsec: finding basis}

Let~$\bV$ be any finitely generated variety.
Then the number of maximal subvarieties of~$\bV$ is some positive integer $k \geq 1$; see Lee {\etal}. \cite[Proposition~4.1]{LRS19}. 
Let $\bM_1,\bM_2,\ldots,\bM_k$ be these maximal subvarieties.
By maximality, each~$\bM_i$ can be defined within~$\bV$ by some identity~$\mu_i$.
If $k \geq 2$, then $\bV = \bM_i \vee \bM_j$ for all distinct~$i$ and~$j$; otherwise, $\bV$ has a unique maximal subvariety and is said to be \textit{prime}.
It follows that each finitely generated variety is either prime or a join of some of its prime subvarieties.

Now it is clear that for any finite semigroup~$S$, the equality $\var\{S\} = \bV$ holds if and only if $S \in \bV$ and $S \notin \bM_i$ for all~$i$.
However, if the variety~$\bV$ is finitely based and a finite {\ib}~$\Sigma$ is available, then the equality $\var\{S\} = \bV$ holds whenever $S \models \Sigma$ and $S \not\models \mu_i$ for all~$i$.
Therefore the identity system $(\Sigma;\mu_1,\mu_2,\ldots,\mu_k)$, called a \textit{Bas-Max system} for~$\bV$, provides an easily verifiable sufficient condition to check if a finite semigroup generates~$\bV$.
Presently, the website database contains Bas-Max systems for all of the following varieties:
\begin{enumerate}
\item varieties with a primitive generator of order up to~$4$;
\item proper subvarieties of Cross varieties in~(a);
\item varieties with a primitive generator of order~$5$.
\end{enumerate}

Now when a semigroup~$S$ entered into the website is shown to generate a variety~$\bV$ via its Bas-Max system $(\Sigma;\mu_1,\mu_2,\ldots,\mu_k)$, then besides the {\ib}~$\Sigma$ for $\var\{S\}$, other important information, such as the primitive generator for~$\bV$, any decomposition of~$\bV$ into a join of its prime subvarieties, and the number of subvarieties of~$\bV$, will also be displayed by the website.

Bas-Max systems for varieties in~(a) and~(b), together with the aforementioned properties, will be listed in Section~\ref{sec: varieties small semigroups}, while their proofs will be given in the appendix sections.
Justification of the Bas-Max systems for varieties in~(c) will be disseminated elsewhere.

The website will be regularly updated with newly established Bas-Max systems for varieties.

\subsection{Testing for equivalent {\ibs}}

Suppose we have a finite set~$\Sigma$ of identities and would like to know information about the variety $[\Sigma]$ of semigroups, such as the primitive generator for $[\Sigma]$ and the varieties covered by $[\Sigma]$.
If this variety happens to be in our database, then many of these information is available.
The question is how do we identify $[\Sigma]$ with a variety in the database.
A tool was developed that will, by specifying one or more identities in Prover9 format, retrieve the variety whose {\ib} is equivalent to~$\Sigma$.

%



\begin{figure}[H]
  \includegraphics[width=\textwidth]{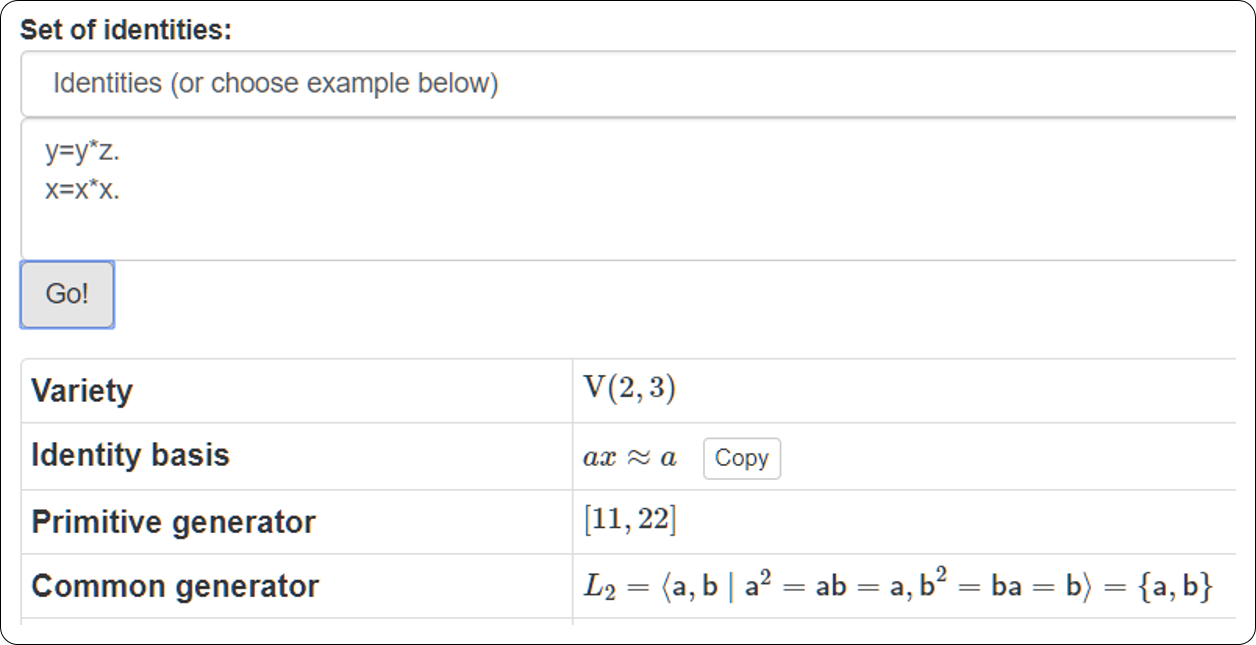}
  \caption{Companion website: example of testing for equivalent {\ib}}
  \label{F: equivbasis}
\end{figure}


{\color{black} \subsection{Filtering varieties using conditions}

Suppose we have some property and want to check which varieties in the database satisfy the property.
This can be done on the website.
To specify the identities, a subset of Prover9 syntax is used.
Only variables (with names started by $u-z$, the operation character $*$, the equal sign $=$, parentheses $($ and $)$, and final mark). 

It is not necessary to specify associativity.


The automatic theorem prover Prover9 and its accompanying program Mace4 that look for counterexamples will run simultaneously to check if the {\ib} for each variety in the database implies the identities provided. \\
~\\
There exist four options to invoke:\\
~\\

\begin{longtable}[]{@{}lll@{}}
\toprule
\begin{minipage}[b]{0.30\columnwidth}\raggedright
Option\strut
\end{minipage} & \begin{minipage}[b]{0.60\columnwidth}\raggedright
Prover9/Mace4 status\strut
\end{minipage}\tabularnewline
\midrule
\endhead
\begin{minipage}[t]{0.30\columnwidth}\raggedright
Proofs\strut
\end{minipage} & \begin{minipage}[t]{0.60\columnwidth}\raggedright
\begin{itemize}
\item
  All varieties for which a proof was found by Prover9 within 1 second.
\end{itemize}\strut
\end{minipage}\tabularnewline
\begin{minipage}[t]{0.30\columnwidth}\raggedright
No countermodels\strut
\end{minipage} & \begin{minipage}[t]{0.60\columnwidth}\raggedright
\begin{itemize}
\item
  The varieties for which a proof was found by Prover9 within 1 second
  plus:
\item
  The varieties where a proof was not found by Prover9 within 1 second
  but Mace4 also didn't found a countermodel within 1 second.
\end{itemize}\strut
\end{minipage}\tabularnewline
\begin{minipage}[t]{0.30\columnwidth}\raggedright
Countermodels\strut
\end{minipage} & \begin{minipage}[t]{0.60\columnwidth}\raggedright
\begin{itemize}
\item
  The varieties for which a countermodel was found by Mace4 within 1
  second;
\end{itemize}\strut
\end{minipage}\tabularnewline
\begin{minipage}[t]{0.30\columnwidth}\raggedright
No proofs\strut
\end{minipage} &  \begin{minipage}[t]{0.60\columnwidth}\raggedright
\begin{itemize}
\item
  The varieties for which a countermodel was found by Mace4 within 1
  second, plus:
\item
  The varieties for which a countermodel was not found by Mace4 within 1
  second, but also a proof was not found by Prover9 within 1 second.
\end{itemize}\strut
\end{minipage}\tabularnewline
\bottomrule
\end{longtable}
}

It is possible to apply successive filters to the sets of varieties obtained.

\subsection{Obtaining lattices of varieties}

A tool was developed to obtain a lattice of a set of varieties created with the filtering tool.\\

It is also possible to filter the list of varieties by leaving only the maximal varieties.\\

\begin{figure}[H]
  \includegraphics[width=\textwidth]{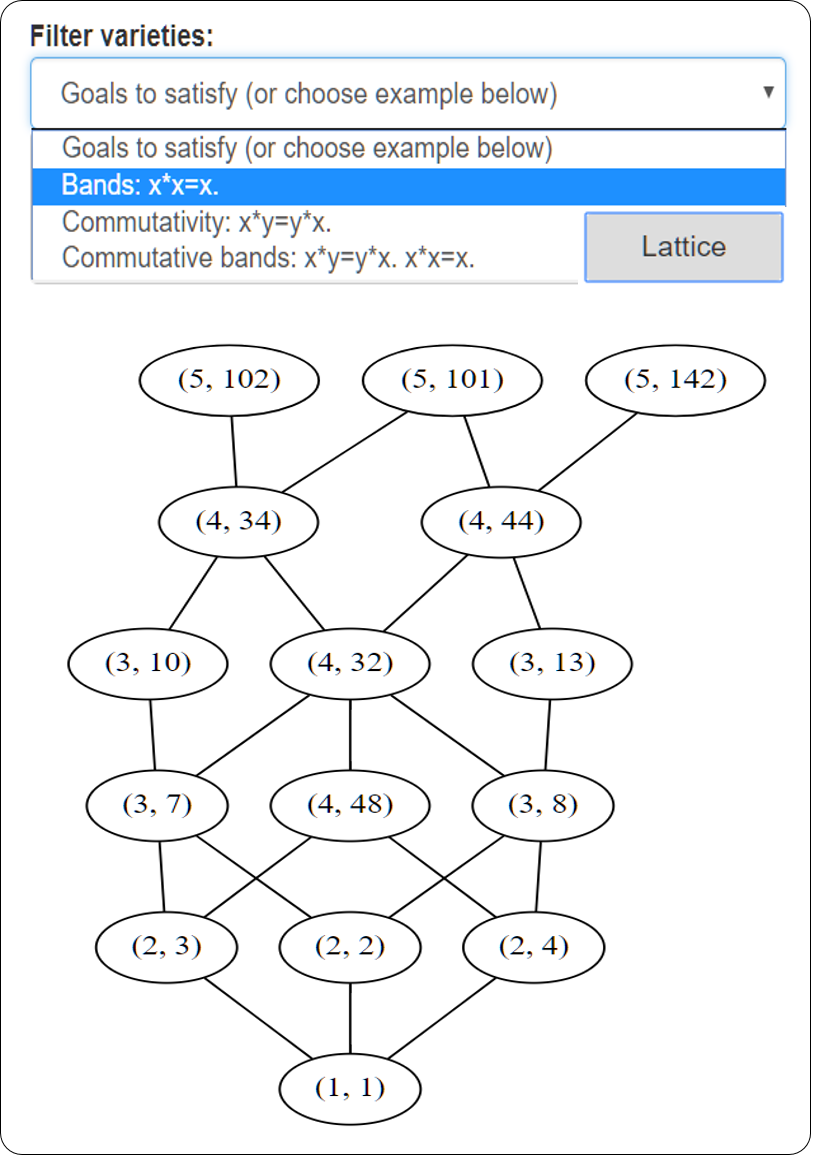}
  \caption{Companion website: obtaining the lattice of varieties generated
  by bands up to order 5}
  \label{F: latbands}
\end{figure}

\subsection{Extending the database: finding {\ibs} for new varieties}

Suppose we have {\ibs} for all varieties generated by a semigroup of order $n-1$ and we want to find an {\ib} for the variety generated some semigroup~$S$ of order~$n$.
If~$S$ does not belong to any variety generated by a semigroup of order less than~$n$, then $\var\{S\}$ is a new variety and we want to find an {\ib} for it.
The website has a tool to try to find candidates of identities that can form an {\ib} for $\var\{S\}$.
The first thing it does is to check, based on results from Subsection~\ref{subsec: INFB}, if~$S$ is inherently non-finitely based.
If the semigroup~$S$ is not inherently non-finitely based, then the website searches, in some \textit{ad hoc} intelligent ways, for candidates of identities of~$S$ to form an {\ib} for $\var\{S\}$.
Of course, if~$S$ happens to be non-finitely based, then the process will not terminate.
But if we are lucky, then the website will produce a natural conjecture for an {\ib}~$\Sigma$ for $\var\{S\}$.
The variety defined by~$\Sigma$ coincides with $\var\{S\}$ if the conjecture is correct, and properly contains $\var\{S\}$ otherwise.
We checked this procedure against all varieties generated by semigroups of order up to~$5$ and in every case, the procedure gave an {\ib} equivalent to the known one.

\section{Varieties generated by small semigroups} \label{sec: varieties small semigroups}


As mentioned in Subsection~\ref{subsec: finding basis}, the present section lists Bas-Max systems for all varieties generated by a semigroup of order up to~$4$ and for some that are their proper subvarieties.
Important information such as primitive generators, decompositions into joins of prime subvarieties, and number of subvarieties are also given.
To this end, the semigroups in Tables~\ref{Tab: generators 2}--\ref{Tab: generators 4} play a crucial role; these semigroups are primitive generators for the varieties they generate, which are in fact precisely all prime varieties generated by a semigroup of order up to~$4$.

\begin{table}[ht] \centering
\begin{tabular}[t]{c|cc}
$\Niltwo$ & 1 & 2 \\ \hline
        1 & 1 & 1 \\
        2 & 1 & 1
\end{tabular}
\ \
\begin{tabular}[t]{c|cc}
$\SL$ & 1 & 2 \\ \hline
    1 & 1 & 1 \\
    2 & 1 & 2
\end{tabular}
\ \
\begin{tabular}[t]{c|cc}
$\LZ$ & 1 & 2 \\ \hline
    1 & 1 & 1 \\
    2 & 2 & 2
\end{tabular}
\ \
\begin{tabular}[t]{c|cc}
$\RZ$ & 1 & 2 \\ \hline
    1 & 1 & 2 \\
    2 & 1 & 2
\end{tabular}
\ \
\begin{tabular}[t]{c|cc}
$\Ztwo$ & 1 & 2 \\ \hline
      1 & 1 & 2 \\
      2 & 2 & 1
\end{tabular}
\caption{Primitive generators of prime varieties generated by a semigroups of order~$2$}
\label{Tab: generators 2}
\end{table}

\begin{table}[ht] \centering
\begin{tabular}[t]{c|ccc}
$\Nilthree$ & 1 & 2 & 3 \\ \hline
          1 & 1 & 1 & 1 \\
          2 & 1 & 1 & 1 \\
          3 & 1 & 1 & 2
\end{tabular}
\quad
\begin{tabular}[t]{c|ccc}
$\JI$ & 1 & 2 & 3 \\ \hline
    1 & 1 & 1 & 1 \\
    2 & 1 & 1 & 1 \\
    3 & 1 & 2 & 3
\end{tabular}
\quad
\begin{tabular}[t]{c|ccc}
$\JIdual$ & 1 & 2 & 3 \\ \hline
        1 & 1 & 1 & 1 \\
        2 & 1 & 1 & 2 \\
        3 & 1 & 1 & 3
\end{tabular}
\quad
\begin{tabular}[t]{c|ccc}
$\Niltwoi$ & 1 & 2 & 3 \\ \hline
         1 & 1 & 1 & 1 \\
         2 & 1 & 1 & 2 \\
         3 & 1 & 2 & 3
\end{tabular}
\\[0.08in]
\begin{tabular}[t]{c|ccc}
$\LZi$ & 1 & 2 & 3 \\ \hline
     1 & 1 & 1 & 1 \\
     2 & 1 & 2 & 3 \\
     3 & 3 & 3 & 3
\end{tabular}
\quad
\begin{tabular}[t]{c|ccc}
$\RZi$ & 1 & 2 & 3 \\ \hline
     1 & 1 & 1 & 3 \\
     2 & 1 & 2 & 3 \\
     3 & 1 & 3 & 3
\end{tabular}
\quad
\begin{tabular}[t]{c|ccc}
$\Zthree$ & 1 & 2 & 3 \\ \hline
        1 & 1 & 2 & 3 \\
        2 & 2 & 3 & 1 \\
        3 & 3 & 1 & 2
\end{tabular}
\caption{Primitive generators of all prime varieties generated by a semigroups of order~$3$}
\label{Tab: generators 3}
\end{table}

\begin{table}[ht] \centering \addtolength{\tabcolsep}{-2pt}
\begin{tabular}[t]{c|cccc}
$\Ffour$ & 1 & 2 & 3 & 4 \\ \hline
       1 & 1 & 1 & 1 & 1 \\
       2 & 1 & 1 & 1 & 1 \\
       3 & 1 & 1 & 1 & 1 \\
       4 & 1 & 1 & 2 & 1
\end{tabular}
\quad
\begin{tabular}[t]{c|cccc}
$\Gfour$ & 1 & 2 & 3 & 4 \\ \hline
       1 & 1 & 1 & 1 & 1 \\
       2 & 1 & 1 & 1 & 1 \\
       3 & 1 & 1 & 1 & 2 \\
       4 & 1 & 1 & 2 & 1
\end{tabular}
\quad
\begin{tabular}[t]{c|cccc}
$\Nilfour$ & 1 & 2 & 3 & 4 \\ \hline
         1 & 1 & 1 & 1 & 1 \\
         2 & 1 & 1 & 1 & 1 \\
         3 & 1 & 1 & 1 & 2 \\
         4 & 1 & 1 & 2 & 3
\end{tabular}
\\[0.08in]
\begin{tabular}[t]{c|cccc}
$\Nilthreei$ & 1 & 2 & 3 & 4 \\ \hline
           1 & 1 & 1 & 1 & 1 \\
           2 & 1 & 1 & 1 & 2 \\
           3 & 1 & 1 & 2 & 3 \\
           4 & 1 & 2 & 3 & 4
\end{tabular}
\quad
\begin{tabular}[t]{c|cccc}
$\Bz$ & 1 & 2 & 3 & 4 \\ \hline
    1 & 1 & 1 & 1 & 1 \\
    2 & 1 & 1 & 1 & 2 \\
    3 & 1 & 2 & 3 & 1 \\
    4 & 1 & 1 & 1 & 4
\end{tabular}
\quad
\begin{tabular}[t]{c|cccc}
$\Az$ & 1 & 2 & 3 & 4 \\ \hline
    1 & 1 & 1 & 1 & 1 \\
    2 & 1 & 1 & 1 & 2 \\
    3 & 1 & 2 & 3 & 2 \\
    4 & 1 & 1 & 1 & 4
\end{tabular}
\\[0.08in]
\begin{tabular}[t]{c|cccc}
$\JIi$ & 1 & 2 & 3 & 4 \\ \hline
     1 & 1 & 1 & 1 & 1 \\
     2 & 1 & 1 & 1 & 2 \\
     3 & 1 & 2 & 3 & 3 \\
     4 & 1 & 2 & 3 & 4
\end{tabular}
\quad
\begin{tabular}[t]{c|cccc}
$\Ptwo$ & 1 & 2 & 3 & 4 \\ \hline
      1 & 1 & 1 & 1 & 1 \\
      2 & 1 & 1 & 1 & 3 \\
      3 & 3 & 3 & 3 & 3 \\
      4 & 4 & 4 & 4 & 4
\end{tabular}
\quad
\begin{tabular}[t]{c|cccc}
$\JIidual$ & 1 & 2 & 3 & 4 \\ \hline
         1 & 1 & 1 & 1 & 1 \\
         2 & 1 & 1 & 2 & 2 \\
         3 & 1 & 1 & 3 & 3 \\
         4 & 1 & 2 & 3 & 4
\end{tabular}
\\[0.08in]
\begin{tabular}[t]{c|cccc}
$\Otwo$ & 1 & 2 & 3 & 4 \\ \hline
      1 & 1 & 1 & 1 & 1 \\
      2 & 1 & 2 & 3 & 4 \\
      3 & 3 & 3 & 3 & 3 \\
      4 & 3 & 4 & 1 & 2
\end{tabular}
\quad
\begin{tabular}[t]{c|cccc}
$\Otwodual$ & 1 & 2 & 3 & 4 \\ \hline
          1 & 1 & 1 & 3 & 3 \\
          2 & 1 & 2 & 3 & 4 \\
          3 & 1 & 3 & 3 & 1 \\
          4 & 1 & 4 & 3 & 2
\end{tabular}
\quad
\begin{tabular}[t]{c|cccc}
$\Ptwodual$ & 1 & 2 & 3 & 4 \\ \hline
          1 & 1 & 1 & 3 & 4 \\
          2 & 1 & 1 & 3 & 4 \\
          3 & 1 & 1 & 3 & 4 \\
          4 & 1 & 3 & 3 & 4
\end{tabular}
\quad
\begin{tabular}[t]{c|cccc}
$\Zfour$ & 1 & 2 & 3 & 4 \\ \hline
       1 & 1 & 2 & 3 & 4 \\
       2 & 2 & 1 & 4 & 3 \\
       3 & 3 & 4 & 2 & 1 \\
       4 & 4 & 3 & 1 & 2
\end{tabular}
\caption{Primitive generators of all prime varieties generated by a semigroups of order~$4$}
\label{Tab: generators 4}
\end{table}

Some well-known semigroups in Tables~\ref{Tab: generators 2}--\ref{Tab: generators 4} are the semilattice~$\SL$ of order~$2$, the left zero band~$\LZ$ of order~$2$, the right zero band~$\RZ$ of order~$2$, the monogenic nilpotent semigroup \[ \Nil_n = \langle a \,|\, a^n=0\rangle = \{ a,a^2,\ldots,a^{n-1},0\} \] of order~$n$, and the cyclic group \[ \mathbb{Z}_n = \langle a \,|\, a^n=1\rangle = \{ a,a^2,\ldots,a^{n-1},1\} \] of order~$n$.
Recall that for any semigroup~$S$, the smallest monoid containing~$S$ is denoted by~$S^1$, and the dual of~$S$ is denoted by~$\overleftarrow{S}$.

In the remainder of the section, information on~88 varieties are grouped by the order of their primitive generators and given below in four subsections; these varieties are named Variety~\texttt{N}, or simply $\bV_\texttt{N}$, where $\texttt{N} \in \{ 1,2,\ldots,88\}$.
Proofs and references for all results are deferred to the appendix sections.

To illustrate how information on each variety can be read, consider Variety~\ref{V: LZ N2i} in Subsection~\ref{subsec: variety 4}, repeated here for reader convenience.

\begin{tcolorbox}[colback=white]
\begin{varexampleA*}[Subsection~\ref{subsec: LZ Nni}] \vquad 
\begin{enumerate}[\qquad(1)]
\item[(Gen)] $[1111,1112,3333,1214]$
\item[(Bas)] $x^3 \approx x^2$, $axy \approx ayx$
\item[(Max)] $x^2y^2 \approx y^2x^2$; $a^2x^2 \approx a^2x$
\item[(Dec)] $\var\{\LZ\} \vee \var\{\Niltwoi\}$
\item[(Sub)] Countably infinite
\end{enumerate}
\end{varexampleA*}
\end{tcolorbox}

\noindent The vector of the primitive generator of the variety~$\bV_{\ref{V: LZ N2i}}$ is given in (Gen).
The two identities in~(Bas) form an {\ib} for~$\bV_{\ref{V: LZ N2i}}$, while each identity in~(Max) defines within~$\bV_{\ref{V: LZ N2i}}$ a maximal subvariety; in other words, the identities in (Bas) and (Max) form a Bas-Max system for~$\bV_{\ref{V: LZ N2i}}$.
The join in~(Dec) is a decomposition of~$\bV_{\ref{V: LZ N2i}}$ into the join of the prime subvarieties $\var\{\LZ\}$ and $\var\{\Niltwoi\}$.
As indicated in~(Sub), the variety~$\bV_{\ref{V: LZ N2i}}$ has countably infinitely many subvarieties.
All these results regarding~$\bV_{\ref{V: LZ N2i}}$ are established in Subsection~\ref{subsec: LZ Nni}.

For another example, consider Variety~\ref{V: Cd} in Subsection~\ref{subsec: variety >4}.

\begin{tcolorbox}[colback=white]
\begin{varexampleB*}[Zhang and Luo~{\cite[Variety~$\mathbf{C}$ in Figure~4]{ZL09}}; Figure~\ref{F: JId LZi}] \vquad 
\begin{enumerate}[\qquad(1)]
\item[(Gen)] $[11111,11113,11133,11144,11155]$
\item[(Bas)] $ax^2 \approx ax$, $xyx \approx x^2y$, $a^2xy \approx a^2yx$
\item[(Max)] $axy \approx ayx$
\item[(Dec)] None
\item[(Sub)] 11
\end{enumerate}
\end{varexampleB*}
\end{tcolorbox}

\noindent The vector of the primitive generator of the variety~$\bV_{\ref{V: Cd}}$ is given in (Gen).
The three identities in~(Bas) form an {\ib} for~$\bV_{\ref{V: Cd}}$, while the identity in~(Max) define the unique maximal subvariety within~$\bV_{\ref{V: Cd}}$.
Since~$\bV_{\ref{V: Cd}}$ has only one maximal subvariety, it is prime and cannot be decomposed into a join of two or more prime subvarieties, as indicated by ``None" in~(Dec).
The number~11 in~(Sub) is the number of subvarieties of~$\bV_{\ref{V: Cd}}$.
Justification of the all these results regarding~$\bV_{\ref{V: Cd}}$ can be found in Zhang and Luo~\cite[Variety~$\mathbf{C}$ in Figure~4]{ZL09}.
For any variety with finitely many subvarieties, its lattice of subvarieties is given in Section~\ref{app: finite}.
Specifically, the lattice of subvarieties of~$\bV_{\ref{V: Cd}}$ can be found in Figure~\ref{F: JId LZi}.

\subsection{Varieties with primitive generator of order~$2$} \label{subsec: variety 2}

\begin{variety}[Evans~{\cite[Figure~3]{Eva71}}] \label{V: N2} \vquad
\begin{enumerate}[\qquad(1)]
\item[(Gen)] $[11,11] = \Niltwo$
\item[(Bas)] $x^2 \approx xy$, $xy \approx yx$
\item[(Max)] $x \approx y$
\item[(Dec)] None
\item[(Sub)] $2$
\end{enumerate}
\end{variety}

\begin{variety}[Evans~{\cite[Figure~3]{Eva71}}] \label{V: SL} \vquad
\begin{enumerate}[\qquad(1)]
\item[(Gen)] $[11,12] = \SL$
\item[(Bas)] $x^2 \approx x$, $xy \approx yx$
\item[(Max)] $x \approx y$
\item[(Dec)] None
\item[(Sub)] $2$
\end{enumerate}
\end{variety}

\begin{variety}[Evans~{\cite[Figure~3]{Eva71}}] \label{V: LZ} \vquad
\begin{enumerate}[\qquad(1)]
\item[(Gen)] $[11,22] = \LZ$
\item[(Bas)] $ax \approx a$
\item[(Max)] $x \approx y$
\item[(Dec)] None
\item[(Sub)] $2$
\end{enumerate}
\end{variety}

\begin{variety}[Evans~{\cite[Figure~3]{Eva71}}] \label{V: RZ} \vquad
\begin{enumerate}[\qquad(1)]
\item[(Gen)] $[12,12] = \RZ$
\item[(Bas)] $xa \approx a$
\item[(Max)] $x \approx y$
\item[(Dec)] None
\item[(Sub)] $2$
\end{enumerate}
\end{variety}

\begin{variety}[Lee {\etal}.~{\cite[Proposition~5.4]{LRS19}}] \label{V: Z2} \vquad
\begin{enumerate}[\qquad(1)]
\item[(Gen)] $[12,21] = \Ztwo$
\item[(Bas)] $x^2a \approx a$, $xy \approx yx$
\item[(Max)] $x \approx y$
\item[(Dec)] None
\item[(Sub)] $2$
\end{enumerate}
\end{variety}

\subsection{Varieties with primitive generator of order~$3$} \label{subsec: variety 3}

\begin{variety}[Tishchenko~{\cite[Variety~$\mathbf{CN}_3$ on page~439]{Tis17}}; Figures~\ref{F: N3 P2}, \ref{F: SL N3}, \ref{F: N3 Zn}, or~\ref{F: N4}] \label{V: N3} \vquad
\begin{enumerate}[\qquad(1)]
\item[(Gen)] $[111,111,112] = \Nilthree$
\item[(Bas)] $x^3 \approx xyz$, $xy \approx yx$
\item[(Max)] $x^3 \approx x^2$
\item[(Dec)] None
\item[(Sub)] $4$
\end{enumerate}
\end{variety}

\begin{variety}[Evans~{\cite[Figure~3]{Eva71}}; Figures~\ref{F: JI JId}, \ref{F: JI LZi}, \ref{F: JId LZi}, \ref{F: SL N3}, or~\ref{F: JI Zp}] \label{V: N2 SL} \vquad
\begin{enumerate}[\qquad(1)]
\item[(Gen)] $[111,111,113]$
\item[(Bas)] $x^2a \approx xa$, $xy \approx yx$
\item[(Max)] $x^2 \approx x$; $x^2 \approx xy$
\item[(Dec)] $\var\{\Niltwo\} \vee \var\{\SL\}$
\item[(Sub)] $4$
\end{enumerate}
\end{variety}

\begin{variety}[Zhang and Luo~{\cite[Variety~$\mathbf{D}$ in Figure~2]{ZL09}}; Figures~\ref{F: JI JId}, \ref{F: JI LZi}, \ref{F: JId LZi}, or~\ref{F: JI Zp}] \label{V: JI} \vquad
\begin{enumerate}[\qquad(1)]
\item[(Gen)] $[111,111,123] = \JI$
\item[(Bas)] $x^2a \approx xa$, $xy^2 \approx yx^2$
\item[(Max)] $xy \approx yx$
\item[(Dec)] None
\item[(Sub)] $5$
\end{enumerate}
\end{variety}

\begin{variety}[Evans~{\cite[Figure~3]{Eva71}}; Figures~\ref{F: JI LZi}, \ref{F: JId LZi}, or~\ref{F: N3 P2}] \label{V: N2 LZ} \vquad
\begin{enumerate}[\qquad(1)]
\item[(Gen)] $[111,111,333]$	
\item[(Bas)] $x^2 \approx xy$
\item[(Max)] $x^2 \approx x$; $xy \approx yx$
\item[(Dec)] $\var\{\Niltwo\} \vee \var\{ \LZ \}$	
\item[(Sub)] 4
\end{enumerate}
\end{variety}

\begin{variety}[Zhang and Luo~{\cite[Variety~$\mathbf{E}$ in Figure~2]{ZL09}}; Figures~\ref{F: JI JId}, \ref{F: JI LZi}, \ref{F: JId LZi}, or~\ref{F: JI Zp}] \label{V: JId} \vquad
\begin{enumerate}[\qquad(1)]
\item[(Gen)] $[111,112,113] = \JIdual$
\item[(Bas)] $ax^2 \approx ax$, $x^2y \approx y^2x$
\item[(Max)] $xy \approx yx$
\item[(Dec)] None
\item[(Sub)] 5
\end{enumerate}
\end{variety}

\begin{variety}[Subsection~\ref{subsec: Zp Nni}] \label{V: N2i} \vquad
\begin{enumerate}[\qquad(1)]
\item[(Gen)] $[111,112,123]=\Niltwoi$
\item[(Bas)] $x^3 \approx x^2$, $xy \approx yx$
\item[(Max)] $x^2y \approx xy^2$
\item[(Dec)] None
\item[(Sub)] Countably infinite
\end{enumerate}
\end{variety}

\begin{variety}[Gerhard and Petrich~{\cite[Variety $\mathtt{LNB}$ in Section~2]{GP85}}; Figures~\ref{F: LZi RZi}, \ref{F: JI LZi}, \ref{F: JId LZi}, or~\ref{F: O2}] \label{V: SL LZ} \vquad
\begin{enumerate}[\qquad(1)]
\item[(Gen)] $[111,121,333]$
\item[(Bas)] $x^2 \approx x$, $axy \approx ayx$
\item[(Max)] $xy \approx x$; $xy \approx yx$
\item[(Dec)] $\var\{\SL\} \vee \var\{\LZ\}$
\item[(Sub)] 4
\end{enumerate}
\end{variety}

\begin{variety}[Gerhard and Petrich~{\cite[Variety $\mathtt{RNB}$ in Section~2]{GP85}}; Figures~\ref{F: LZi RZi}, \ref{F: JI LZi}, \ref{F: JId LZi}, or~\ref{F: O2}] \label{V: SL RZ} \vquad
\begin{enumerate}[\qquad(1)]
\item[(Gen)] $[111,123,123]$
\item[(Bas)] $x^2 \approx x$, $xya \approx yxa$
\item[(Max)] $xy \approx y$; $xy \approx yx$
\item[(Dec)] $\var\{\SL\} \vee \var\{\RZ\}$
\item[(Sub)] 4
\end{enumerate}
\end{variety}

\begin{variety}[Subsection~\ref{subsec: JI Zp SL Zpp}] \label{V: SL Z2} \vquad
\begin{enumerate}[\qquad(1)]
\item[(Gen)] $[111,123,132]$
\item[(Bas)] $x^3 \approx x$, $xy \approx yx$
\item[(Max)] $x^2 \approx x$; $x^2y \approx y$
\item[(Dec)] $\var\{\SL\} \vee \var\{\Ztwo\}$
\item[(Sub)] 4
\end{enumerate}
\end{variety}

\begin{variety}[Gerhard and Petrich~{\cite[Variety $\mathtt{LRB}$ in Section~2]{GP85}}; Figures~\ref{F: LZi RZi}, \ref{F: JI LZi}, \ref{F: JId LZi}, or~\ref{F: O2}] \label{V: LZi} \vquad
\begin{enumerate}[\qquad(1)]
\item[(Gen)] $[111,123,333] = \LZi$
\item[(Bas)] $x^2 \approx x$, $xyx \approx xy$
\item[(Max)] $axy \approx ayx$
\item[(Dec)] None
\item[(Sub)] 5
\end{enumerate}
\end{variety}

\begin{variety}[Evans~{\cite[Figure~3]{Eva71}}; Figures~\ref{F: JI LZi}, \ref{F: JId LZi}, or~\ref{F: N3 P2}] \label{V: N2 RZ} \vquad
\begin{enumerate}[\qquad(1)]
\item[(Gen)] $[113,113,113]$
\item[(Bas)] $x^2 \approx yx$
\item[(Max)] $x^2 \approx x$; $xy \approx yx$
\item[(Dec)] $\var\{\Niltwo\} \vee \var\{\RZ\}$
\item[(Sub)] 4
\end{enumerate}
\end{variety}

\begin{variety}[Subsection~\ref{subsec: JI Zp SL Zpp}] \label{V: N2 Z2} \vquad
\begin{enumerate}[\qquad(1)]
\item[(Gen)] $[113,113,331]$
\item[(Bas)] $x^2ab \approx ab$, $xy \approx yx$
\item[(Max)] $x^3 \approx x$; $x^3 \approx x^2$
\item[(Dec)] $\var\{\Niltwo\} \vee \var\{\Ztwo\}$
\item[(Sub)] 4
\end{enumerate}
\end{variety}

\begin{variety}[Gerhard and Petrich~{\cite[Variety $\mathtt{RRB}$ in Section~2]{GP85}}; Figures~\ref{F: LZi RZi}, \ref{F: JI LZi}, \ref{F: JId LZi}, or~\ref{F: O2}] \label{V: RZi} \vquad
\begin{enumerate}[\qquad(1)]
\item[(Gen)] $[113,123,133] = \RZi$
\item[(Bas)] $x^2 \approx x$, $xyx \approx yx$
\item[(Max)] $xya \approx yxa$
\item[(Dec)] None
\item[(Sub)] 5
\end{enumerate}
\end{variety}

\begin{variety}[Lee {\etal}.~{\cite[Proposition~5.4]{LRS19}}] \label{V: Z3} \vquad
\begin{enumerate}[\qquad(1)]
\item[(Gen)] $[123,231,312] = \Zthree$
\item[(Bas)] $x^3a \approx a$, $xy \approx yx$
\item[(Max)] $x \approx y$
\item[(Dec)] None
\item[(Sub)] 2
\end{enumerate}
\end{variety}

\subsection{Varieties with primitive generator of order~$4$} \label{subsec: variety 4}

\begin{variety}[Tishchenko~{\cite[Variety~$\mathbf{N}_{3,2}$ on page~439]{Tis17}}; Figures~\ref{F: JI JId} or~\ref{F: N3 P2}] \label{V: F4} \vquad
\begin{enumerate}[\qquad(1)]
\item[(Gen)] $[1111,1111,1111,1121] = \Ffour$
\item[(Bas)] $x^2 \approx yzt$
\item[(Max)] $xy \approx yx$
\item[(Dec)] None
\item[(Sub)] 4
\end{enumerate}
\end{variety}

\begin{variety}[Tishchenko~{\cite[Variety~$\mathbf{N}_3$ on page~438]{Tis17}}; Figure~\ref{F: N3 P2}] \label{V: N3 F4} \vquad
\begin{enumerate}[\qquad(1)]
\item[(Gen)] $[1111,1111,1111,1122]$
\item[(Bas)] $x^3 \approx yzt$
\item[(Max)] $x^3 \approx x^2$; $xy \approx yx$
\item[(Dec)] $\var\{\Nilthree\} \vee \var\{\Ffour\}$
\item[(Sub)] 6
\end{enumerate}
\end{variety}

\begin{variety}[Tishchenko~{\cite[Variety~$\mathbf{CN}_{3,2}$ on page~439]{Tis17} \label{V: G4}}; Figures~\ref{F: JI JId}, \ref{F: N3 P2}, \ref{F: SL N3}, \ref{F: N3 Zn}, or~\ref{F: N4}] \vquad
\begin{enumerate}[\qquad(1)]
\item[(Gen)] $[1111,1111,1112,1121] = \Gfour$
\item[(Bas)] $x^2 \approx xyz$, $xy \approx yx$
\item[(Max)] $x^2 \approx xy$
\item[(Dec)] None
\item[(Sub)] 3
\end{enumerate}
\end{variety}

\begin{variety}[Lee {\etal}.~{\cite[Condition~A8]{LRS19}}; Figure~\ref{F: N4}] \label{V: N4} \vquad
\begin{enumerate}[\qquad(1)]
\item[(Gen)] $[1111,1111,1112,1123] = \Nilfour$
\item[(Bas)] $x^4 \approx xyzt$, $x^2y \approx xy^2$, $xy \approx yx$
\item[(Max)] $x^4 \approx x^3$
\item[(Dec)] None
\item[(Sub)] 8
\end{enumerate}
\end{variety}

\begin{variety}[Zhang and Luo~{\cite[Variety~$\mathbf{D \vee E}$ in Figure~2]{ZL09}}; Figure~\ref{F: JI JId}] \label{V: JI JId} \vquad
\begin{enumerate}[\qquad(1)]
\item[(Gen)] $[1111,1111,1113,1214]$
\item[(Bas)] $x^3 \approx x^2$, $xyx \approx x^2y^2$, $xyx \approx y^2x^2$, $ax^2b \approx axb$
\item[(Max)] $xyx \approx x^2y$; $xyx \approx yx^2$
\item[(Dec)] $\var\{\JI\} \vee \var\{\JIdual\}$
\item[(Sub)] 13
\end{enumerate}
\end{variety}

\begin{variety}[Subsection~\ref{subsec: JI Nni}] \label{V: JI N2i} \vquad
\begin{enumerate}[\qquad(1)]
\item[(Gen)] $[1111,1111,1113,1234]$
\item[(Bas)] $x^3 \approx x^2$, $x^2y^2 \approx y^2x^2$, $xya \approx yxa$
\item[(Max)] $x^2y \approx yx^2$; $xy^2 \approx yx^2$
\item[(Dec)] $\var\{\JI\} \vee \var\{\Niltwoi\}$
\item[(Sub)] Countably infinite
\end{enumerate}
\end{variety}

\begin{variety}[Lee {\etal}.~{\cite[Proposition~6.14]{LRS19}}; Figure~\ref{F: SL N3}] \label{V: SL N3} \vquad
\begin{enumerate}[\qquad(1)]
\item[(Gen)] $[1111,1111,1121,1114]$
\item[(Bas)] $x^2ab \approx xab$, $xy \approx yx$
\item[(Max)] $x^3 \approx x^2$; $x^3 \approx y^3$
\item[(Dec)] $\var\{\SL\} \vee \var\{\Nilthree\}$
\item[(Sub)] 8
\end{enumerate}
\end{variety}

\begin{variety}[Tishchenko~{\cite[Variety~$\mathbf{L}_{1,3}$ on page~438]{Tis17}}; Figure~\ref{F: N3 P2}] \label{V: LZ N3} \vquad
\begin{enumerate}[\qquad(1)]
\item[(Gen)] $[1111,1111,1121,4444]$
\item[(Bas)] $x^3 \approx xyz$
\item[(Max)] $x^3 \approx x^2$; $x^3 \approx y^3$
\item[(Dec)] $\var\{\LZ\} \vee \var\{\Nilthree\}$
\item[(Sub)] 10
\end{enumerate}
\end{variety}

\begin{variety}[Evans~{\cite[Figure~3]{Eva71}}; Figures~\ref{F: JI LZi} or~\ref{F: JId LZi}] \label{V: N2 SL LZ} \vquad
\begin{enumerate}[\qquad(1)]
\item[(Gen)] $[1111,1111,1131,4444]$
\item[(Bas)] $x^2a \approx xa$, $ax^2 \approx ax$, $axy \approx ayx$
\item[(Max)] $x^2 \approx x$; $x^2 \approx xy$; $xy \approx yx$
\item[(Dec)] $\var\{\Niltwo\} \vee \var\{\SL\} \vee \var\{\LZ\}$
\item[(Sub)] 8
\end{enumerate}
\end{variety}

\begin{variety}[Evans~{\cite[Figure~3]{Eva71}}; Figures~\ref{F: JI LZi} or~\ref{F: JId LZi}] \label{V: N2 SL RZ} \vquad
\begin{enumerate}[\qquad(1)]
\item[(Gen)] $[1111,1111,1134,1134]$
\item[(Bas)] $x^2a \approx xa$, $ax^2 \approx ax$, $xya \approx yxa$
\item[(Max)] $x^2 \approx x$; $x^2 \approx yx$; $xy \approx yx$
\item[(Dec)] $\var\{\Niltwo\} \vee \var\{\SL\} \vee \var\{\RZ\}$
\item[(Sub)] 8
\end{enumerate}
\end{variety}

\begin{variety}[Subsection~\ref{subsec: JI Zp SL Zpp}] \label{V: N2 SL Z2} \vquad
\begin{enumerate}[\qquad(1)]
\item[(Gen)] $[1111,1111,1134,1143]$
\item[(Bas)] $x^3a \approx xa$, $xy \approx yx$
\item[(Max)] $x^3 \approx x$; $x^3 \approx x^2$; $x^2 \approx y^2$
\item[(Dec)] $\var\{\Niltwo\} \vee \var\{\SL\} \vee \var\{\Ztwo\}$
\item[(Sub)] 8
\end{enumerate}
\end{variety}

\begin{variety}[Zhang and Luo~{\cite[Variety ${ \mathbf{L^1 \vee N} }$ in Figure~5]{ZL09}}; Figures~\ref{F: JI LZi} or~\ref{F: JId LZi}] \label{V: N2 LZi} \vquad
\begin{enumerate}[\qquad(1)]
\item[(Gen)] $[1111,1111,1134,4444]$
\item[(Bas)] $x^2a \approx xa$, $ax^2 \approx ax$, $xyx \approx xy$
\item[(Max)] $x^2 \approx x$; $axy \approx ayx$
\item[(Dec)] $\var\{\Niltwo\} \vee \var\{\LZi\}$
\item[(Sub)] 10
\end{enumerate}
\end{variety}

\begin{variety}[Zhang and Luo~{\cite[Variety ${ \mathbf{D \vee L} }$ in Figure~4]{ZL09}}; Figure~\ref{F: JI LZi}] \label{V: LZ JI} \vquad
\begin{enumerate}[\qquad(1)]
\item[(Gen)] $[1111,1111,1231,4444]$
\item[(Bas)] $x^2a \approx xa$, $axy^2 \approx ayx^2$
\item[(Max)] $ax^2 \approx ax$; $xy^2 \approx yx^2$
\item[(Dec)] $\var\{\LZ\} \vee \var\{\JI\}$
\item[(Sub)] 10
\end{enumerate}
\end{variety}

\begin{variety}[Dual of Variety~\ref{V: LZ JId}; Figure~\ref{F: JId LZi}] \label{V: RZ JI} \vquad
\begin{enumerate}[\qquad(1)]
\item[(Gen)] $[1111,1111,1234,1234]$
\item[(Bas)] $x^2a \approx xa$, $xya \approx yxa$
\item[(Max)] $ax^2 \approx ax$; $xy^2 \approx yx^2$
\item[(Dec)] $\var\{\RZ\} \vee \var\{\JI\}$
\item[(Sub)] 10
\end{enumerate}
\end{variety}

\begin{variety}[Subsection~\ref{subsec: JI Zp SL Zpp}] \label{V: Z2 JI} \vquad
\begin{enumerate}[\qquad(1)]
\item[(Gen)] $[1111,1111,1234,1243]$
\item[(Bas)] $x^3a \approx xa$, $x^2y^2 \approx y^2x^2$, $xya \approx yxa$
\item[(Max)] $x^3 \approx x^2$; $xy \approx yx$
\item[(Dec)] $\var\{\Ztwo\} \vee \var\{\JI\}$
\item[(Sub)] 10
\end{enumerate}
\end{variety}

\begin{variety}[Edmunds~{\cite[Semigroup $\mathtt{S(4,11)}$ on page~70]{Edm80}}; Figure~\ref{F: JI LZi}] \label{V: JI LZi} \vquad
\begin{enumerate}[\qquad(1)]
\item[(Gen)] $[1111,1111,1234,4444]$
\item[(Bas)] $x^2a \approx xa$, $xyx \approx xy^2$
\item[(Max)] $ax^2 \approx ax$; $axy^2 \approx ayx^2$
\item[(Dec)] $\var\{\JI\} \vee \var\{\LZi\}$
\item[(Sub)] 13
\end{enumerate}
\end{variety}

\begin{variety}[Subsection~\ref{subsec: JI Nni}] \label{V: JId N2i} \vquad
\begin{enumerate}[\qquad(1)]
\item[(Gen)] $[1111,1112,1113,1134]$
\item[(Bas)] $x^3 \approx x^2$, $x^2y^2 \approx y^2x^2$, $axy \approx ayx$
\item[(Max)] $x^2y \approx yx^2$; $x^2y \approx y^2x$
\item[(Dec)] $\var\{\JIdual\} \vee \var\{\Niltwoi\}$
\item[(Sub)] Countably infinite
\end{enumerate}
\end{variety}

\begin{variety}[Subsection~\ref{subsec: Zp Nni}] \label{V: N3i} \vquad
\begin{enumerate}[\qquad(1)]
\item[(Gen)] $[1111,1112,1123,1234] = \Nilthreei$
\item[(Bas)] $x^4 \approx x^3$, $xy \approx yx$
\item[(Max)] $x^3y^2 \approx x^2y^3$
\item[(Dec)] None
\item[(Sub)] Countably infinite
\end{enumerate}
\end{variety}

\begin{variety}[Subsection~\ref{subsec: B0 A0}] \label{V: B0} \vquad
\begin{enumerate}[\qquad(1)]
\item[(Gen)] $[1111,1112,1231,1114] = \Bz$
\item[(Bas)] $x^3 \approx x^2$, $x^2yx^2 \approx yxy$, $x^2y^2 \approx y^2x^2$
\item[(Max)] $a^2x^2b^2 \approx a^2xb^2$
\item[(Dec)] None
\item[(Sub)] Countably infinite
\end{enumerate}
\end{variety}

\begin{variety}[Subsection~\ref{subsec: B0 A0}] \label{V: A0} \vquad
\begin{enumerate}[\qquad(1)]
\item[(Gen)] $[1111,1112,1232,1114] = \Az$
\item[(Bas)] $x^3 \approx x^2$, $x^2yx^2 \approx yxy$
\item[(Max)] $x^2y^2 \approx y^2x^2$
\item[(Dec)] None
\item[(Sub)] Countably infinite
\end{enumerate}
\end{variety}

\begin{variety}[Subsection~\ref{subsec: JIi}] \label{V: JIi} \vquad
\begin{enumerate}[\qquad(1)]
\item[(Gen)] $[1111,1112,1233,1234] = \JIi$
\item[(Bas)] $x^3 \approx x^2$, $x^2y^2 \approx y^2x^2$, $xyx \approx yx^2$
\item[(Max)] $x^2ya^2 \approx yx^2a^2$
\item[(Dec)] None
\item[(Sub)] Countably infinite
\end{enumerate}
\end{variety}

\begin{variety}[Zhang and Luo~{\cite[Variety~${ \mathbf{E \vee L} }$ in Figure~4]{ZL09}}; Figure~\ref{F: JId LZi}] \label{V: LZ JId} \vquad
\begin{enumerate}[\qquad(1)]
\item[(Gen)] $[1111,1112,3333,1114]$
\item[(Bas)] $ax^2 \approx ax$, $axy \approx ayx$
\item[(Max)] $x^2a \approx xa$; $x^2y \approx y^2x$
\item[(Dec)] $\var\{\LZ\} \vee \var\{\JIdual\}$
\item[(Sub)] 10
\end{enumerate}
\end{variety}

\begin{variety}[Edmunds~{\cite[Semigroup $\mathtt{S(4,25)}$ on page~70]{Edm80}}; Figure~\ref{F: JId LZi}] \label{V: JId LZi} \vquad
\begin{enumerate}[\qquad(1)]
\item[(Gen)] $[1111,1112,3333,1134]$
\item[(Bas)] $ax^2 \approx ax$, $xyx \approx x^2y$
\item[(Max)] $x^2a \approx xa$; $a^2xy \approx a^2yx$
\item[(Dec)] $\var\{\JIdual\} \vee \var\{\LZi\}$
\item[(Sub)] 14
\end{enumerate}
\end{variety}

\begin{variety}[Subsection~\ref{subsec: LZ Nni}] \label{V: LZ N2i} \vquad
\begin{enumerate}[\qquad(1)]
\item[(Gen)] $[1111,1112,3333,1214]$
\item[(Bas)] $x^3 \approx x^2$, $axy \approx ayx$
\item[(Max)] $x^2y^2 \approx y^2x^2$; $a^2x^2 \approx a^2x$
\item[(Dec)] $\var\{\LZ\} \vee \var\{\Niltwoi\}$
\item[(Sub)] Countably infinite
\end{enumerate}
\end{variety}

\begin{variety}[Subsection~\ref{subsec: LZi Nni}] \label{V: N2i LZi} \vquad
\begin{enumerate}[\qquad(1)]
\item[(Gen)] $[1111,1112,3333,1234]$
\item[(Bas)] $x^3 \approx x^2$, $xyx \approx x^2y$
\item[(Max)] $a^2x^2 \approx a^2x$; $a^2x^2y^2 \approx a^2y^2x^2$
\item[(Dec)] $\var\{\Niltwoi\} \vee \var\{\LZi\}$
\item[(Sub)] Countably infinite
\end{enumerate}
\end{variety}

\begin{variety}[Tishchenko~{\cite[Variety~$\mathbf{L}_{2,2}$ on page~438]{Tis17}}; Figure~\ref{F: N3 P2}] \label{V: P2} \vquad
\begin{enumerate}[\qquad(1)]
\item[(Gen)] $[1111,1113,3333,4444] = \Ptwo$
\item[(Bas)] $abx \approx ab$
\item[(Max)] $x^2 \approx xy$
\item[(Dec)] None
\item[(Sub)] 5
\end{enumerate}
\end{variety}

\begin{variety}[Dual of Variety~\ref{V: JIi}] \label{V: JIid} \vquad
\begin{enumerate}[\qquad(1)]
\item[(Gen)] $[1111,1122,1133,1234] = \JIidual$
\item[(Bas)] $x^3 \approx x^2$, $x^2y^2 \approx y^2x^2$, $xyx \approx x^2y$
\item[(Max)] $a^2x^2y \approx a^2yx^2$
\item[(Dec)] None
\item[(Sub)] Countably infinite
\end{enumerate}
\end{variety}

\begin{variety}[Dual of Variety~\ref{V: LZ JI}; Figure~\ref{F: JI LZi}] \label{V: RZ JId} \vquad
\begin{enumerate}[\qquad(1)]
\item[(Gen)] $[1111,1122,1134,1134]$
\item[(Bas)] $ax^2 \approx ax$, $x^2ya \approx y^2xa$
\item[(Max)] $x^2a \approx xa$; $x^2y \approx y^2x$
\item[(Dec)] $\var\{\RZ\} \vee \var\{\JIdual\}$
\item[(Sub)] 10
\end{enumerate}
\end{variety}

\begin{variety}[Dual of Variety~\ref{V: Z2 JI}] \label{V: Z2 JId} \vquad
\begin{enumerate}[\qquad(1)]
\item[(Gen)] $[1111,1122,1134,1143]$
\item[(Bas)] $ax^3 \approx ax$, $x^2y^2 \approx y^2x^2$, $axy \approx ayx$
\item[(Max)] $x^3 \approx x^2$; $xy \approx yx$
\item[(Dec)] $\var\{\Ztwo\} \vee \var\{\JIdual\}$
\item[(Sub)] 10
\end{enumerate}
\end{variety}

\begin{variety}[Subsection~\ref{subsec: LZ Nni}] \label{V: RZ N2i} \vquad
\begin{enumerate}[\qquad(1)]
\item[(Gen)] $[1111,1122,1234,1234]$
\item[(Bas)] $x^3 \approx x^2$, $xya \approx yxa$
\item[(Max)] $x^2y^2 \approx y^2x^2$; $x^2a^2 \approx xa^2$
\item[(Dec)] $\var\{\RZ\} \vee \var\{\Niltwoi\}$
\item[(Sub)] Countably infinite
\end{enumerate}
\end{variety}

\begin{variety}[Subsection~\ref{subsec: Zp Nni}] \label{V: Z2 N2i} \vquad
\begin{enumerate}[\qquad(1)]
\item[(Gen)] $[1111,1122,1234,1243]$
\item[(Bas)] $x^4 \approx x^2$, $xy \approx yx$
\item[(Max)] $x^3 \approx x^2$; $x^3y \approx xy^3$
\item[(Dec)] $\var\{\Ztwo\} \vee \var\{\Niltwoi\}$
\item[(Sub)] Countably infinite
\end{enumerate}
\end{variety}

\begin{variety}[Gerhard and Petrich~{\cite[Variety $\mathtt{NB}$ in Section~2]{GP85}}; Figure~\ref{F: LZi RZi}] \label{V: SL LZ RZ} \vquad
\begin{enumerate}[\qquad(1)]
\item[(Gen)] $[1111,1214,3333,1214]$
\item[(Bas)] $x^2 \approx x$, $axya \approx ayxa$
\item[(Max)] $xyx \approx x$; $xyx \approx xy$; $xyx \approx yx$
\item[(Dec)] $\var\{\SL\} \vee \var\{\LZ\} \vee \var\{ \RZ\}$
\item[(Sub)] 8
\end{enumerate}
\end{variety}

\begin{variety}[Petrich~{\cite[Lemma~7.3(vii)]{Pet74}}; Figure~\ref{F: O2}] \label{V: SL LZ Z2} \vquad
\begin{enumerate}[\qquad(1)]
\item[(Gen)] $[1111,1214,3333,1412]$
\item[(Bas)] $x^3 \approx x$, $axy \approx ayx$
\item[(Max)] $x^2 \approx x$; $xy \approx yx$; $ax^2 \approx a$
\item[(Dec)] $\var\{\SL\} \vee \var\{\LZ\} \vee \var\{\Ztwo\}$
\item[(Sub)] 8
\end{enumerate}
\end{variety}

\begin{variety}[Gerhard and Petrich~{\cite[Variety $\mathtt{LQNB}$ in Section~2]{GP85}}; Figure~\ref{F: LZi RZi}] \label{V: RZ LZi} \vquad
\begin{enumerate}[\qquad(1)]
\item[(Gen)] $[1111,1234,1234,4444]$
\item[(Bas)] $x^2 \approx x$, $xyxa \approx xya$
\item[(Max)] $xyx \approx xy$; $axya \approx ayxa$
\item[(Dec)] $\var\{\RZ\} \vee \var\{\LZi\}$
\item[(Sub)] 10
\end{enumerate}
\end{variety}

\begin{variety}[Tishchenko~{\cite[Variety~$\mathbf{V}_2$ on page~111]{Tis07}}; Figure~\ref{F: O2}] \label{V: Z2 LZi} \vquad
\begin{enumerate}[\qquad(1)]
\item[(Gen)] $[1111,1234,1324,4444]$
\item[(Bas)] $x^3 \approx x$, $xyx \approx x^2y$
\item[(Max)] $x^2 \approx x$; $axy \approx ayx$
\item[(Dec)] $\var\{\Ztwo\} \vee \var\{\LZi\}$
\item[(Sub)] 10
\end{enumerate}
\end{variety}

\begin{variety}[Subsection~\ref{subsec: JI Zp SL Zpp}] \label{V: SL Z3} \vquad
\begin{enumerate}[\qquad(1)]
\item[(Gen)] $[1111,1234,1342,1423]$
\item[(Bas)] $x^4 \approx x$, $xy \approx yx$
\item[(Max)] $x^2 \approx x$; $x^3a \approx a$
\item[(Dec)] $\var\{\SL\} \vee \var\{\Zthree\}$
\item[(Sub)] 4
\end{enumerate}
\end{variety}

\begin{variety}[Tishchenko~{\cite[Proposition~2.25]{Tis07}}; Figure~\ref{F: O2}] \label{V: O2} \vquad
\begin{enumerate}[\qquad(1)]
\item[(Gen)] $[1111,1234,3333,3412] = \Otwo$
\item[(Bas)] $x^3 \approx x$, $xyxy \approx xy^2x$
\item[(Max)] $xyx \approx x^2y$
\item[(Dec)] None
\item[(Sub)] 11
\end{enumerate}
\end{variety}

\begin{variety}[Dual of Variety~\ref{V: LZ N3}; Figure~\ref{F: N3 P2}] \label{V: RZ N3} \vquad
\begin{enumerate}[\qquad(1)]
\item[(Gen)] $[1114,1114,1124,1114]$
\item[(Bas)] $x^3 \approx yzx$
\item[(Max)] $x^3 \approx x^2$; $x^3 \approx y^3$
\item[(Dec)] $\var\{\RZ\} \vee \var\{\Nilthree\}$
\item[(Sub)] 10
\end{enumerate}
\end{variety}

\begin{variety}[Subsection~\ref{subsec: N3 Zn}] \label{V: Z2 N3} \vquad
\begin{enumerate}[\qquad(1)]
\item[(Gen)] $[1114,1114,1124,4441]$
\item[(Bas)] $x^2abc \approx abc$, $xy \approx yx$
\item[(Max)] $x^4 \approx x^2$; $x^4 \approx x^3$
\item[(Dec)] $\var\{\Ztwo\} \vee \var\{\Nilthree\}$
\item[(Sub)] 8
\end{enumerate}
\end{variety}

\begin{variety}[Dual of Variety~\ref{V: N2 LZi}; Figures~\ref{F: JI LZi} or~\ref{F: JId LZi}] \label{V: N2 RZi} \vquad
\begin{enumerate}[\qquad(1)]
\item[(Gen)] $[1114,1114,1134,1144]$
\item[(Bas)] $x^2a \approx xa$, $ax^2 \approx ax$, $xyx \approx yx$
\item[(Max)] $x^2 \approx x$; $xya \approx yxa$
\item[(Dec)] $\var\{\Niltwo\} \vee \var\{\RZi\}$
\item[(Sub)] 10
\end{enumerate}
\end{variety}

\begin{variety}[Dual of Variety~\ref{V: JId LZi}; Figure~\ref{F: JId LZi}] \label{V: JI RZi} \vquad
\begin{enumerate}[\qquad(1)]
\item[(Gen)] $[1114,1114,1234,1144]$
\item[(Bas)] $x^2a \approx xa$, $xyx \approx yx^2$
\item[(Max)] $ax^2 \approx ax$; $xya^2 \approx yxa^2$
\item[(Dec)] $\var\{\JI\} \vee \var\{\RZi\}$
\item[(Sub)] 14
\end{enumerate}
\end{variety}

\begin{variety}[Dual of Variety~\ref{V: JI LZi}; Figure~\ref{F: JI LZi}] \label{V: JId RZi} \vquad
\begin{enumerate}[\qquad(1)]
\item[(Gen)] $[1114,1124,1134,1144]$
\item[(Bas)] $ax^2 \approx ax$, $xyx \approx y^2x$
\item[(Max)] $x^2a \approx xa$; $x^2ya \approx y^2xa$
\item[(Dec)] $\var\{\JIdual\} \vee \var\{\RZi\}$
\item[(Sub)] 13
\end{enumerate}
\end{variety}

\begin{variety}[Subsection~\ref{subsec: LZi Nni}] \label{V: N2i RZi} \vquad
\begin{enumerate}[\qquad(1)]
\item[(Gen)] $[1114,1124,1234,1144]$
\item[(Bas)] $x^3 \approx x^2$, $xyx \approx yx^2$
\item[(Max)] $x^2a^2 \approx xa^2$; $x^2y^2a^2 \approx y^2x^2a^2$
\item[(Dec)] $\var\{\Niltwoi\} \vee \var\{\RZi\}$
\item[(Sub)] Countably infinite
\end{enumerate}
\end{variety}

\begin{variety}[Gerhard and Petrich~{\cite[Variety $\mathtt{RQNB}$ in Section~2]{GP85}}; Figure~\ref{F: LZi RZi}] \label{V: LZ RZi} \vquad
\begin{enumerate}[\qquad(1)]
\item[(Gen)] $[1114,1224,1334,1444]$
\item[(Bas)] $x^2 \approx x$, $axyx \approx ayx$
\item[(Max)] $xyx \approx yx$; $axya \approx ayxa$
\item[(Dec)] $\var\{\LZ\} \vee \var\{\RZi\}$
\item[(Sub)] 10
\end{enumerate}
\end{variety}

\begin{variety}[Dual to Variety~\ref{V: SL LZ Z2}; Figure~\ref{F: O2}] \label{V: SL RZ Z2} \vquad
\begin{enumerate}[\qquad(1)]
\item[(Gen)] $[1114,1234,1234,4441]$
\item[(Bas)] $x^3 \approx x$, $xya \approx yxa$
\item[(Max)] $x^2 \approx x$; $x^2a \approx a$; $xy \approx yx$
\item[(Dec)] $\var\{\SL\} \vee \var\{\RZ\} \vee \var\{\Ztwo\}$
\item[(Sub)] 8
\end{enumerate}
\end{variety}

\begin{variety}[Dual of Variety~\ref{V: Z2 LZi}; Figure~\ref{F: O2}] \label{V: Z2 RZi} \vquad
\begin{enumerate}[\qquad(1)]
\item[(Gen)] $[1114,1234,1324,1444]$
\item[(Bas)] $x^3 \approx x$, $xyx \approx yx^2$
\item[(Max)] $x^2 \approx x$; $xya \approx yxa$
\item[(Dec)] $\var\{\Ztwo\} \vee \var\{\RZi\}$
\item[(Sub)] 10
\end{enumerate}
\end{variety}

\begin{variety}[Dual of Variety~\ref{V: O2}; Figure~\ref{F: O2}] \label{V: O2d} \vquad
\begin{enumerate}[\qquad(1)]
\item[(Gen)] $[1133,1234,1331,1432] = \Otwodual$
\item[(Bas)] $x^3 \approx x$, $xyxy \approx yx^2y$
\item[(Max)] $xyx \approx yx^2$
\item[(Dec)] None
\item[(Sub)] 11
\end{enumerate}
\end{variety}

\begin{variety}[Gerhard and Petrich~{\cite[Variety $\mathtt{Rec\,B}$ in Section~2]{GP85}}; Figure~\ref{F: LZi RZi}] \label{V: LZ RZ} \vquad
\begin{enumerate}[\qquad(1)]
\item[(Gen)] $[1133,2244,1133,2244]$
\item[(Bas)] $xyx \approx x$
\item[(Max)] $xy \approx x$; $xy \approx y$
\item[(Dec)] $\var\{\LZ\} \vee \var\{\RZ\}$
\item[(Sub)] 4
\end{enumerate}
\end{variety}

\begin{variety}[Tishchenko~{\cite[Variety ${ \mathbf{A}_2 \vee \mathbf{L}_1 }$ on page~108]{Tis07}}; Figure~\ref{F: O2}] \label{V: LZ Z2} \vquad
\begin{enumerate}[\qquad(1)]
\item[(Gen)] $[1133,2244,3311,4422]$
\item[(Bas)] $ax^2 \approx a$, $axy \approx ayx$
\item[(Max)] $x^2 \approx x$; $x^2 \approx y^2$
\item[(Dec)] $\var\{\LZ\} \vee \var\{\Ztwo\}$
\item[(Sub)] 4
\end{enumerate}
\end{variety}

\begin{variety}[Dual of Variety~\ref{V: P2}; Figure~\ref{F: N3 P2}] \label{V: P2d} \vquad
\begin{enumerate}[\qquad(1)]
\item[(Gen)] $[1134,1134,1134,1334] = \Ptwodual$
\item[(Bas)] $xab \approx ab$
\item[(Max)] $x^2 \approx yx$
\item[(Dec)] None
\item[(Sub)] 5
\end{enumerate}
\end{variety}

\begin{variety}[Subsection~\ref{subsec: JI Zp SL Zpp}] \label{V: N2 Z3} \vquad
\begin{enumerate}[\qquad(1)]
\item[(Gen)] $[1134,1134,3341,4413]$
\item[(Bas)] $x^3ab \approx ab$, $xy \approx yx$
\item[(Max)] $x^4 \approx x$; $x^3 \approx x^2$
\item[(Dec)] $\var\{\Niltwo\} \vee \var\{\Zthree\}$
\item[(Sub)] 4
\end{enumerate}
\end{variety}

\begin{variety}[Dual of Variety~\ref{V: LZ Z2}; Figure~\ref{F: O2}] \label{V: RZ Z2} \vquad
\begin{enumerate}[\qquad(1)]
\item[(Gen)] $[1234,1234,3412,3412]$
\item[(Bas)] $x^2a \approx a$, $xya \approx yxa$
\item[(Max)] $x^2 \approx x$; $x^2 \approx y^2$
\item[(Dec)] $\var\{\RZ\} \vee \var\{\Ztwo\}$
\item[(Sub)] 4
\end{enumerate}
\end{variety}

\begin{variety}[Lee {\etal}.~{\cite[Proposition~5.4]{LRS19}}; Figure~\ref{F: SL Zpp}] \label{V: Z4} \vquad
\begin{enumerate}[\qquad(1)]
\item[(Gen)] $[1234,2143,3421,4312] = \Zfour$
\item[(Bas)] $x^4a \approx a$, $xy \approx yx$
\item[(Max)] $x^3 \approx x$
\item[(Dec)] None
\item[(Sub)] 3
\end{enumerate}
\end{variety}

\subsection{Some varieties with primitive generator of order greater than~$4$} \label{subsec: variety >4}

\begin{variety}[Zhang and Luo~{\cite[Variety~$\mathbf{F \vee S}$ in Figure~2]{ZL09}}; Figure~\ref{F: JI JId}] \label{V: SL F4} \vquad
\begin{enumerate}[\qquad(1)]
\item[(Gen)] $[11111,11111,11111,11141,11211]$
\item[(Bas)] $x^3 \approx x^2$, $x^2ab \approx xab$, $xya \approx yxa$, $axy \approx ayx$
\item[(Max)] $x^2y \approx x^2$; $xy \approx yx$
\item[(Dec)] $\var\{\SL\} \vee \var\{\Ffour\}$
\item[(Sub)] 8
\end{enumerate}
\end{variety}

\begin{variety}[Tishchenko~{\cite[Variety~$\mathbf{V}_{1,3}$ on page~439]{Tis17}}; Figure~\ref{F: N3 P2}] \label{V: LZ G4} \vquad
\begin{enumerate}[\qquad(1)]
\item[(Gen)] $[11111,11111,11111,11211,55555]$
\item[(Bas)] $x^2 \approx xyz$
\item[(Max)] $x^2 \approx xy$; $x^2 \approx y^2$
\item[(Dec)] $\var\{\LZ\} \vee \var\{\Gfour\}$
\item[(Sub)] 7
\end{enumerate}
\end{variety}

\begin{variety}[Dual of Variety~\ref{V: Cd}; Figure~\ref{F: JId LZi}] \label{V: C} \vquad
\begin{enumerate}[\qquad(1)]
\item[(Gen)] $[11111,11111,11111,11345,13345]$
\item[(Bas)] $x^2a \approx xa$, $xyx \approx yx^2$, $xya^2 \approx yxa^2$
\item[(Max)] $xya \approx yxa$
\item[(Dec)] None
\item[(Sub)] 11
\end{enumerate}
\end{variety}

\begin{variety}[Zhang and Luo~{\cite[Variety~$\mathbf{G \vee S}$ in Figure~2]{ZL09}}; Figures~\ref{F: JI JId} or~\ref{F: SL N3}] \label{V: SL G4} \vquad
\begin{enumerate}[\qquad(1)]
\item[(Gen)] $[11111,11111,11112,11141,11211]$
\item[(Bas)] $x^3 \approx x^2$, $x^2ab \approx xab$, $xy \approx yx$
\item[(Max)] $x^2a \approx xa$; $x^2 \approx y^2$
\item[(Dec)] $\var\{\SL\} \vee \var\{\Gfour\}$
\item[(Sub)] 6
\end{enumerate}
\end{variety}

\begin{variety}[Tishchenko~{\cite[Variety~$\mathbf{L}_{2,3}$ in Proposition~3.1]{Tis17}}; Figure~\ref{F: N3 P2}] \label{V: N3 P2} \vquad
\begin{enumerate}[\qquad(1)]
\item[(Gen)] $[11111,11111,11214,44444,55555]$
\item[(Bas)] $xyx \approx xyz$
\item[(Max)] $x^3 \approx x^2$; $x^3 \approx xyx$
\item[(Dec)] $\var\{\Nilthree\} \vee \var\{\Ptwo\}$
\item[(Sub)] 13
\end{enumerate}
\end{variety}

\begin{variety}[Zhang and Luo~{\cite[Variety~$\mathbf{C}$ in Figure~4]{ZL09}}; Figure~\ref{F: JId LZi}] \label{V: Cd} \vquad
\begin{enumerate}[\qquad(1)]
\item[(Gen)] $[11111,11113,11133,11144,11155]$
\item[(Bas)] $ax^2 \approx ax$, $xyx \approx x^2y$, $a^2xy \approx a^2yx$
\item[(Max)] $axy \approx ayx$
\item[(Dec)] None
\item[(Sub)] 11
\end{enumerate}
\end{variety}

\begin{variety}[Gerhard and Petrich~{\cite[Variety $\mathtt{RB}$ in Section~2]{GP85}}; Figure~\ref{F: LZi RZi}] \label{V: LZi RZi} \vquad
\begin{enumerate}[\qquad(1)]
\item[(Gen)] $[11111,12125,33333,12345,12155]$
\item[(Bas)] $x^2 \approx x$, $xyxzx \approx xyzx$
\item[(Max)] $axyx \approx ayx$; $xyxa \approx xya$
\item[(Dec)] $\var\{\LZi\} \vee \var\{\RZi\}$
\item[(Sub)] 13
\end{enumerate}
\end{variety}

\begin{variety}[Dual of Variety~\ref{V: LZ G4}; Figure~\ref{F: N3 P2}] \label{V: RZ G4} \vquad
\begin{enumerate}[\qquad(1)]
\item[(Gen)] $[11115,11115,11115,11215,11115]$
\item[(Bas)] $x^2 \approx yzx$
\item[(Max)] $x^2 \approx yx$; $x^2 \approx y^2$
\item[(Dec)] $\var\{\RZ\} \vee \var\{\Gfour\}$
\item[(Sub)] 7
\end{enumerate}
\end{variety}

\begin{variety}[Dual of Variety~\ref{V: N3 P2}; Figure~\ref{F: N3 P2}] \label{V: N3 P2d} \vquad
\begin{enumerate}[\qquad(1)]
\item[(Gen)] $[11145,11145,11245,11145,11445]$
\item[(Bas)] $xyx \approx zyx$
\item[(Max)] $x^3 \approx x^2$; $x^3 \approx xyx$
\item[(Dec)] $\var\{\Nilthree\} \vee \var\{\Ptwodual\}$
\item[(Sub)] 13
\end{enumerate}
\end{variety}

\begin{variety}[Zhang and Luo~{\cite[Variety~$\mathbf{D \vee F}$ in Figure~2]{ZL09}}; Figure~\ref{F: JI JId}] \label{V: JI G4} \vquad
\begin{enumerate}[\qquad(1)]
\item[(Gen)] $[111111,111111,111111,111111,111211,113116]$
\item[(Bas)] $x^3 \approx x^2$, $xy^2 \approx yx^2$, $ax^2b \approx axb$
\item[(Max)] $x^2a \approx xa$; $x^2y \approx xy^2$
\item[(Dec)] $\var\{\JI\} \vee \var\{\Gfour\}$
\item[(Sub)] 10
\end{enumerate}
\end{variety}

\begin{variety}[Zhang and Luo~{\cite[Variety~$\mathbf{E \vee F}$ in Figure~2]{ZL09}}; Figure~\ref{F: JI JId}] \label{V: JId G4} \vquad
\begin{enumerate}[\qquad(1)]
\item[(Gen)] $[111111,111111,111111,111114,112111,111116]$
\item[(Bas)] $x^3 \approx x^2$, $x^2y \approx y^2x$, $ax^2b \approx axb$
\item[(Max)] $ax^2 \approx ax$; $x^2y \approx xy^2$
\item[(Dec)] $\var\{\JIdual\} \vee \var\{\Gfour\}$
\item[(Sub)] 10
\end{enumerate}
\end{variety}

\begin{variety}[Tishchenko~{\cite[Variety~$\mathbf{V}_{2,3}$ on page~439]{Tis17}}; Figure~\ref{F: N3 P2}] \label{V: G4 P2} \vquad
\begin{enumerate}[\qquad(1)]
\item[(Gen)] $[111111,111111,111111,112115,555555,666666]$
\item[(Bas)] $x^3 \approx x^2$, $xyx \approx xyz$
\item[(Max)] $xyx \approx x^2$; $xyx \approx xy$
\item[(Dec)] $\var\{\Gfour\} \vee \var\{\Ptwo\}$
\item[(Sub)] 9
\end{enumerate}
\end{variety}

\begin{variety}[Dual of Variety~\ref{V: G4 P2}; Figure~\ref{F: N3 P2}] \label{V: G4 P2d} \vquad
\begin{enumerate}[\qquad(1)]
\item[(Gen)] $[111156,111156,111156,112156,111156,111556]$
\item[(Bas)] $x^3 \approx x^2$, $xyx \approx zyx$
\item[(Max)] $xyx \approx x^2$; $xyx \approx yx$
\item[(Dec)] $\var\{\Gfour\} \vee \var\{\Ptwodual\}$
\item[(Sub)] 9
\end{enumerate}
\end{variety}

\begin{variety}[Mel'nik~{\cite[Variety $B_{24}$ in Figure~3]{Mel72}}; Figure~\ref{F: N4}] \label{V: CN 3 4} \vquad
\begin{enumerate}[\qquad(1)]
\item[(Gen)] $[1111111,1111111,1111112,1111121,1111122,1112235,1121254]$
\item[(Bas)] $x^3 \approx xyzt$, $x^2y \approx xy^2$, $xy \approx yx$
\item[(Max)] $x^2y \approx x^3$
\item[(Dec)] None
\item[(Sub)] 7
\end{enumerate}
\end{variety}

\begin{variety}[Mel'nik~{\cite[Variety $B_{26}$ in Figure~3]{Mel72}}; Figure~\ref{F: N4}] \label{V: CN 2 4} \vquad
\begin{enumerate}[\qquad(1)]
\item[(Gen)] $[1111\,1111, 1111\,1111, 1111\,1112, 1111\,1121, 1111\,1211,1111\,2134,$ \newline $\phantom{[}1112\,1315, 1121\,1451]$
\item[(Bas)] $x^2 \approx xyzt$, $xy \approx yx$
\item[(Max)] $x^2 \approx xyz$
\item[(Dec)] None
\item[(Sub)] 4
\end{enumerate}
\end{variety}

\begin{variety}[Mel'nik~{\cite[Variety $B_{25}$ in Figure~3]{Mel72}}; Figure~\ref{F: N4}] \label{V: CN 21 4} \vquad
\begin{enumerate}[\qquad(1)]
\item[(Gen)] $[1111\,1111,1111\,1111,1111\,1112,1111\,1121,1111\,1211,1111\,2134,$ \newline $\phantom{[}1112\,1315,1121\,1452]$
\item[(Bas)] $x^2y \approx xyzt$, $xy \approx yx$
\item[(Max)] $x^3 \approx xyz$; $x^3 \approx x^2$
\item[(Dec)] $\var\{\Nilthree\} \vee \bV_{\ref{V: CN 2 4}}$
\item[(Sub)] 6
\end{enumerate}
\end{variety}

\section{Problems}

In this section we propose a number of problems that are naturally prompted by the results in this paper.

\begin{prob}
Identify all varieties generated by a semigroup of order~$6$.
\end{prob}

Regarding groups we propose the following problems.

\begin{prob}
Given a finite group~$G$, find good bounds for the following:
\begin{enumerate}
 \item the number of critical groups in $\var\{G\}$;
 \item the order of the largest critical group in $\var\{G\}$;
 \item the number of subvarieties of $\var\{G\}$;
 \item the number of varieties covered by $\var\{G\}$.
\end{enumerate}
Solve the same problems for the class $\mathbf{C}(e,m,c)$ introduced in Subsection~\ref{subsec: C(e,m,c)}.
\end{prob}


\clearpage

\appendix

\section{Basic results on identities of some semigroups} \label{app: prelim}

The present section establishes some background equational results that are required in Sections~\ref{app: finite} and~\ref{app: infinite}.
For more information on universal algebra, refer to the monograph of Burris and Sankappanavar~\cite{BS81}.

Words are formed over some countably infinite set~$\sX$ of variables.
An \textit{identity} is an expression $\bu \approx \bv$ where $\bu, \bv \in \sX^+$.
An identity $\bu \approx \bv$ is \textit{nontrivial} if $\bu \neq \bv$.
A semigroup~$S$ \textit{satisfies} an identity $\bu \approx \bv$ if for any substitution $\varphi : \sX \to S$, the elements $\varphi(\bu)$ and $\varphi(\bv)$ of~$S$ are equal; otherwise, $S$ \textit{violates} $\bu \approx \bv$.
An identity $\bu \approx \bv$ is \textit{deducible} from some identity $\bu' \approx \bv'$ if there exist some substitution $\varphi: \sX \to \sX^+$ and some words $\mathbf{p},\mathbf{q} \in \sX^*$ such that $\bu = \mathbf{p}\big(\varphi(\bu')\big)\mathbf{q}$ and $\bv = \mathbf{p}\big(\varphi(\bv')\big)\mathbf{q}$.
An identity $\bu \approx \bv$ is \textit{deducible} from some set~$\Sigma$ of identities if there exists some sequence \[ \bu = \bw_0, \bw_1, \ldots, \bw_m = \bv \] of words where each identity $\bw_i \approx \bw_{i+1}$ is deducible from some identity in~$\Sigma$.

For any word~$\bw$,
\begin{itemize}
\item the \textit{head} of~$\bw$, denoted by $\head(\bw)$, is the first variable occurring in~$\bw$;
\item the \textit{tail} of~$\bw$, denoted by $\tail(\bw)$, is the last variable occurring in~$\bw$;
\item the \textit{initial part} of~$\bw$, denoted by $\ini(\bw)$, is the word obtained by retaining the first occurrence of each variable in~$\bw$;
\item the \textit{content} of~$\bw$, denoted by $\con(\bw)$, is the set of variables occurring in~$\bw$;
\item the number of occurrences of a variable~$x$ in~$\bw$ is denoted by $\occ(x,\bw)$.
\end{itemize}

\begin{lemma} \label{L: LZ LZi N3 Nni Zn word}
Let $\bu \approx \bv$ be any identity\up.
Then
\begin{enumerate}[\rm(i)]
\item $\LZ$ satisfies $\bu \approx \bv$ if and only if $\head(\bu) = \head(\bv)$\up;
\item $\LZi$ satisfies $\bu \approx \bv$ if and only if $\ini(\bu) = \ini(\bv)$\up;
\item $\Nilthree$ satisfies $\bu \approx \bv$ if and only if either \[ |\bu|, |\bv| \geq 3 \quad\text{or} \quad \occ(x,\bu) = \occ(x,\bv) \text{ for all $x \in \sX$}; \]
\item $\Nil_n^1$ satisfies $\bu \approx \bv$ if and only if for all $x \in \sX$\up, either \[ \occ(x,\bu) = \occ(x,\bv) < n \quad \text{or} \quad \occ(x,\bu), \occ(x,\bv) \geq n; \]
\item $\mathbb{Z}_n$ satisfies $\bu \approx \bv$ if and only if $\occ(x,\bu) \equiv \occ(x,\bv) \pmod n$ for all $x \in \sX$\up.
\end{enumerate}
\end{lemma}

\begin{proof}
These results are well-known and easily verified.
For instance, see Petrich and Reilly~\cite[Theorem~V.1.9]{PR99} for parts~(i) and~(ii) and Almeida~\cite[Lemma~6.1.4]{Alm94} for parts~(iv) and~(v).
\end{proof}

\begin{lemma} \label{L: excl Nni}
Let~$\bW$ be any variety that satisfies the identity
\begin{equation}
x^{n+k} \approx x^n \label{id: excl Nni xn+k=xn}
\end{equation}
for some $n \geq 2$ and $k \geq 1$\up.
Suppose that $\Nil_n^1 \notin \bW$\up.
Then~$\bW$ satisfies the identity
\begin{equation}
(x^ny)^{n-1+k} x^n \approx (x^ny)^{n-1}x^n. \label{id: excl Nni}
\end{equation}
\end{lemma}

\begin{proof}
By assumption, the variety~$\bW$ satisfies some identity $\alpha: \bu \approx \bv$ that is violated by the semigroup~$\Nil_n^1$.
In view of Lemma~\ref{L: LZ LZi N3 Nni Zn word}(iv), generality is not lost by assuming the existence of some variable $y \in \sX$ such that $\occ(y,\bu) = r < n$ and $\occ(y,\bv) = s > r$.
Then \[ \bu = \bu_0 y \bu_1 y \bu_2 \cdots y \bu_r \quad \text{and} \quad \bv = \bv_0 y \bv_1 y \bv_2 \cdots y \bv_s \] for some $\bu_i, \bv_j \in \sX^*$ such that $y \notin \con(\bu_i\bv_j)$.
Let~$\varphi$ denote the substitution that maps~$y$ to~$x^ny$ and every other variable to~$x^k$.
Then since \[ (x^ny)^r x^n \stackrel{\eqref{id: excl Nni xn+k=xn}}{\approx} \big(\varphi(\bu)\big) x^n \stackrel{\alpha}{\approx} \big(\varphi(\bv)\big) x^n \stackrel{\eqref{id: excl Nni xn+k=xn}}{\approx} (x^ny)^s x^n, \] the variety~$\bW$ satisfies the identity $(x^ny)^r x^n \approx (x^ny)^s x^n$.
It follows that~$\bW$ satisfies the identity $\beta: (x^ny)^{n-1} x^n \approx (x^ny)^{n-1+t} x^n$ for some $t \geq 1$.
Since
\begin{align*}
(x^ny)^{n-1} x^n & \stackrel{\beta}{\approx} (x^ny)^{n-1+t} x^n \stackrel{\beta}{\approx} (x^ny)^{n-1+2t} x^n \stackrel{\beta}{\approx} \cdots \\
& \stackrel{\beta}{\approx} (x^ny)^{n-1+kt} x^n \stackrel{\eqref{id: excl Nni xn+k=xn}}{\approx} (x^ny)^{n-1+k} x^n,
\end{align*}
the variety~$\bW$ also satisfies the identity~\eqref{id: excl Nni}.
\end{proof}

\begin{lemma}[{\cite[Lemma~7]{GS82}}] \label{L: JI word}
The semigroup~$\JI$ satisfies an identity $\bu \approx \bv$ if and only if $\con(\bu) = \con(\bv)$ and either of the following conditions holds\up:
\begin{enumerate}[\rm(i)]
\item $\occ\big(\tail(\bu),\bu\big) = \occ\big(\tail(\bv),\bv\big) = 1$ with $\tail(\bu) = \tail(\bv)$\up;
\item $\occ\big(\tail(\bu),\bu\big), \occ\big(\tail(\bv),\bv\big) \geq 2$\up.
\end{enumerate}
\end{lemma}

\begin{lemma} \label{L: excl JI}
Let~$\bW$ be any variety that satisfies the identity
\begin{equation}
x^{2n} \approx x^n \label{id: excl JI x2n=xn}
\end{equation}
for some $n \geq 2$\up.
Suppose that $\JI \notin \bW$\up.
Then~$\bW$ satisfies one of the identities
\begin{align}
(x^ny)^{n+1} & \approx x^ny, \label{id: excl JI xnyn+1=xny} \\
x^nyx^n & \approx x^ny. \label{id: excl JI xnyxn=xny}
\end{align}
\end{lemma}

\begin{proof}
By assumption, the variety~$\bW$ satisfies an identity $\alpha: \bu \approx \bv$ that is violated by the semigroup~$\JI$.
It is well known and easily shown that if $\con(\bu) \neq \con(\bv)$, then the identity $(x^ny)^nx^n \approx x^n$ is deducible from the identities $\{ \eqref{id: excl JI x2n=xn}, \bu \approx \bv \}$ and so is satisfied by the variety~$\bW$, whence~$\bW$ also satisfies the identity~\eqref{id: excl JI xnyn+1=xny}.
Therefore assume that $\con(\bu) = \con(\bv)$.
By Lemma~\ref{L: JI word}, there are two cases.

\noindent\textsc{Case~1}: $\tail(\bu) = \tail(\bv) = y$ with $\occ(y,\bu) = 1$ and $\occ(y,\bv) = m \geq 2$.
Then \[ \bu = \bw_0 y \quad \text{and} \quad \bv = \bw_1 y \bw_2 y \cdots \bw_m y \] for some $\bw_i \in \sX^*$ such that $y \notin \con(\bw_i)$.
Let~$\varphi$ denote the substitution that maps~$y$ to $x^ny$ and every other variable to~$x^n$.
Then \[ x^n y \stackrel{\eqref{id: excl JI x2n=xn}}{\approx} x^n\big( \varphi(\bu) \big) \stackrel{\alpha}{\approx} x^n \big( \varphi(\bv) \big) \stackrel{\eqref{id: excl JI x2n=xn}}{\approx} (x^ny)^m, \] so that~$\bW$ satisfies the identity $\beta: x^n y \approx (x^ny)^{\ell+1}$ with $\ell = m-1$.
Since \[ x^n y \stackrel{\beta}{\approx} (x^n y)^{\ell+1} \stackrel{\beta}{\approx} (x^n y)^{2\ell+1} \stackrel{\beta}{\approx} \cdots \stackrel{\beta}{\approx} (x^n y)^{n\ell+1} \stackrel{\eqref{id: excl JI x2n=xn}}{\approx} (x^n y)^{n+1}, \] the variety~$\bW$ also satisfies the identity~\eqref{id: excl JI xnyn+1=xny}.

\noindent\textsc{Case~2}: $\tail(\bu) = y \neq z = \tail(\bv)$ with $\occ(y,\bu) = 1$ and $\occ(z,\bv) \geq 1$.
The assumption $\con(\bu) = \con(\bv)$ implies that $\occ(y,\bv) = m \geq 1$.
Then \[ \bu = \bw_0 y \quad \text{and} \quad \bv = \bw_1 y \bw_2 y \cdots \bw_m y \bw_{m+1} z \] for some $\bw_i \in \sX$ such that $y \notin \con(\bw_i)$.
Let~$\varphi$ denote the substitution in Case~1.
Then \[ x^n y \stackrel{\eqref{id: excl JI x2n=xn}}{\approx} x^n\big( \varphi(\bu) \big) \stackrel{\alpha}{\approx} x^n \big( \varphi(\bv) \big) \stackrel{\eqref{id: excl JI x2n=xn}}{\approx} (x^ny)^m x^n, \] so that~$\bW$ satisfies the identity $\gamma: x^n y \approx (x^ny)^mx^n$.
Since \[ x^n (yx^n) \stackrel{\gamma}{\approx} \big(x^n(yx^n)\big)^mx^n \stackrel{\eqref{id: excl JI x2n=xn}}{\approx} (x^ny)^mx^n \stackrel{\gamma}{\approx} x^n y, \] the variety~$\bW$ also satisfies the identity~\eqref{id: excl JI xnyxn=xny}.
\end{proof}

\begin{lemma} \label{L: HFB implies countable}
A variety that contains only finitely based subvarieties\up, contains at most countably many subvarieties\up.
\end{lemma}

\begin{proof}
Up to renaming of variables, there can only be countably many finite sets of identities.
\end{proof}

\section{Some finite lattices of varieties} \label{app: finite}

\subsection{Subvarieties of ${ \bV_{\ref{V: JI JId}} = \var\{\JI,\protect\JIdual\} }$}

\begin{proposition}[Zhang and Luo~{\cite[Figure~2]{ZL09}}] \quad \label{P: JI JId}
\begin{enumerate}[\rm(i)]
\item The proper nontrivial subvarieties of $\bV_{\ref{V: JI JId}} = \var\{\JI,\JIdual\}$ are
\begin{align*}
\bV_{\ref{V: N2}} & = \var\{\Niltwo\}, & \bV_{\ref{V: SL}} & = \var\{\SL\}, & \bV_{\ref{V: N2 SL}} & = \var\{\Niltwo,\SL\}, \\
\bV_{\ref{V: JI}} & = \var\{\JI\}, & \bV_{\ref{V: JId}} & = \var\{\JIdual\}, & \bV_{\ref{V: F4}} & = \var\{\Ffour\}, \\
\bV_{\ref{V: G4}} & = \var\{\Gfour\}, & \bV_{\ref{V: SL F4}} & = \var\{\SL,\Ffour\}, & \bV_{\ref{V: SL G4}} & = \var\{\SL,\Gfour\}, \\
\bV_{\ref{V: JI G4}} & = \var\{\JI,\Gfour\}, & \bV_{\ref{V: JId G4}} & = \var\{\JIdual,\Gfour\}. &&
\end{align*}

\item The lattice $\sL(\bV_{\ref{V: JI JId}})$ is given in Figure~\ref{F: JI JId}.
\end{enumerate}
\end{proposition}

\begin{figure}[ht]
\begin{center}
\begin{picture}(100,170)(00,20) \setlength{\unitlength}{0.7mm}
\put(25,85){\circle*{\circlesize}}
\put(10,70){\circle*{\circlesize}} \put(40,70){\circle*{\circlesize}}
\put(10,50){\circle*{\circlesize}} \put(25,55){\circle*{\circlesize}} \put(40,50){\circle*{\circlesize}}
\put(25,45){\circle*{\circlesize}}
\put(10,40){\circle*{\circlesize}} \put(25,35){\circle*{\circlesize}}
\put(10,30){\circle*{\circlesize}} \put(25,25){\circle*{\circlesize}}
\put(10,20){\circle*{\circlesize}}
\put(10,10){\circle*{\circlesize}}
\put(10,70){\line(1,1){15}} \put(10,40){\line(1,1){30}} \put(10,30){\line(1,1){15}} \put(10,20){\line(1,1){30}} \put(10,10){\line(1,1){15}}
\put(40,70){\line(-1,1){15}} \put(25,55){\line(-1,1){15}} \put(25,35){\line(-1,1){15}}
\put(10,10){\line(0,1){30}} \put(10,50){\line(0,1){20}} \put(25,25){\line(0,1){30}} \put(40,50){\line(0,1){20}}
\put(26,87){\makebox(0,0)[b]{$\bV_{\ref{V: JI JId}}$}}
\put(8,70){\makebox(0,0)[r]{$\bV_{\ref{V: JI G4}}$}} \put(42,70){\makebox(0,0)[l]{$\bV_{\ref{V: JId G4}}$}}
\put(8,50){\makebox(0,0)[r]{$\bV_{\ref{V: JI}}$}} \put(24,60){\makebox(0,0)[b]{$\bV_{\ref{V: SL F4}}$}} \put(42,50){\makebox(0,0)[l]{$\bV_{\ref{V: JId}}$}}
\put(26,48){\makebox(0,0)[l]{$\bV_{\ref{V: SL G4}}$}}
\put(8,40){\makebox(0,0)[r]{$\bV_{\ref{V: F4}}$}} \put(27,35){\makebox(0,0)[tl]{$\bV_{\ref{V: N2 SL}}$}}
\put(8,30){\makebox(0,0)[r]{$\bV_{\ref{V: G4}}$}} \put(27,25){\makebox(0,0)[tl]{$\bV_{\ref{V: SL}}$}}
\put(8,20){\makebox(0,0)[r]{$\bV_{\ref{V: N2}}$}}
\put(8,10){\makebox(0,0)[r]{$\bT$}}
\end{picture}
\end{center}
\caption{The lattice $\sL(\bV_{\ref{V: JI JId}})$}
\label{F: JI JId}
\end{figure}

\subsection{Subvarieties of ${ \bV_{\ref{V: LZi RZi}} = \var\{ \LZi, \RZi \} }$}

\begin{proposition}[Gerhard and Petrich~{\cite[Section~2]{GP85}}] \quad
\begin{enumerate}[\rm(i)]
\item The proper nontrivial subvarieties of $\bV_{\ref{V: LZi RZi}} = \var\{ \LZi, \RZi \}$ are
\begin{align*}
\bV_{\ref{V: SL}} & = \var\{\SL\}, & \bV_{\ref{V: LZ}} & =\var\{\LZ\}, & \bV_{\ref{V: RZ}} & = \var\{\RZ\}, \\
\bV_{\ref{V: SL LZ}} & = \var\{\SL,\LZ\}, & \bV_{\ref{V: SL RZ}} & = \var\{\SL,\RZ\}, & \bV_{\ref{V: LZi}} & = \var\{\LZi\}, \\
\bV_{\ref{V: RZi}} & = \var\{\RZi\}, & \bV_{\ref{V: SL LZ RZ}} & = \var\{\SL,\LZ,\RZ\}, & \bV_{\ref{V: RZ LZi}} & = \var\{\RZ,\LZi\}, \\
\bV_{\ref{V: LZ RZi}} & = \var\{\LZ,\RZi\}, & \bV_{\ref{V: LZ RZ}} & = \var\{\LZ,\RZ\}. & &
\end{align*}
\item The lattice $\sL(\bV_{\ref{V: LZi RZi}})$ is given in Figure~\ref{F: LZi RZi}.
\end{enumerate}
\end{proposition}

\begin{figure}[ht]
\begin{center}
\begin{picture}(120,130)(00,10) \setlength{\unitlength}{0.7mm}
\put(30,60){\circle*{\circlesize}}
\put(10,50){\circle*{\circlesize}} \put(50,50){\circle*{\circlesize}}
\put(10,40){\circle*{\circlesize}} \put(30,40){\circle*{\circlesize}} \put(50,40){\circle*{\circlesize}}
\put(10,30){\circle*{\circlesize}} \put(30,30){\circle*{\circlesize}} \put(50,30){\circle*{\circlesize}}
\put(10,20){\circle*{\circlesize}} \put(30,20){\circle*{\circlesize}} \put(50,20){\circle*{\circlesize}}
\put(30,10){\circle*{\circlesize}}
\put(10,50){\line(2,1){20}} \put(10,30){\line(2,1){40}} \put(30,20){\line(2,1){20}} \put(10,20){\line(2,1){20}} \put(30,10){\line(2,1){20}}
\put(30,10){\line(-2,1){20}} \put(30,20){\line(-2,1){20}} \put(50,20){\line(-2,1){20}} \put(50,30){\line(-2,1){40}} \put(50,50){\line(-2,1){20}}
\put(10,20){\line(0,1){30}} \put(30,10){\line(0,1){10}} \put(30,30){\line(0,1){10}} \put(50,20){\line(0,1){30}}
\put(31,62){\makebox(0,0)[b]{$\bV_{\ref{V: LZi RZi}}$}}
\put(08,50){\makebox(0,0)[br]{$\bV_{\ref{V: RZ LZi}}$}} \put(52,50){\makebox(0,0)[bl]{$\bV_{\ref{V: LZ RZi}}$}}
\put(08,40){\makebox(0,0)[r]{$\bV_{\ref{V: LZi}}$}} \put(30,43){\makebox(0,0)[b]{$\bV_{\ref{V: SL LZ RZ}}$}} \put(52,40){\makebox(0,0)[l]{$\bV_{\ref{V: RZi}}$}}
\put(08,30){\makebox(0,0)[r]{$\bV_{\ref{V: SL LZ}}$}} \put(32,31){\makebox(0,0)[l]{$\bV_{\ref{V: LZ RZ}}$}} \put(52,30){\makebox(0,0)[l]{$\bV_{\ref{V: SL RZ}}$}}
\put(08,20){\makebox(0,0)[tr]{$\bV_{\ref{V: LZ}}$}} \put(32,20){\makebox(0,0)[tl]{$\bV_{\ref{V: SL}}$}} \put(52,20){\makebox(0,0)[tl]{$\bV_{\ref{V: RZ}}$}}
\put(30,07.5){\makebox(0,0)[t]{$\bT$}}
\end{picture}
\end{center}
\caption{The lattice $\sL(\bV_{\ref{V: LZi RZi}})$}
\label{F: LZi RZi}
\end{figure}

\subsection{Subvarieties of ${ \bV_{\ref{V: JI LZi}} = \var\{\JI,\LZi\} }$ and ${ \bV_{\ref{V: JId RZi}} = \var\{\protect\JIdual,\RZi\} }$}

\begin{proposition}[Zhang and Luo~{\cite[Subvarieties of~$\mathbf{A}$ in Figure~5]{ZL09}}] \quad \label{P: Ji LZi}
\begin{enumerate}[\rm(i)]
\item The proper nontrivial subvarieties of $\bV_{\ref{V: JI LZi}} = \var\{\JI,\LZi\}$ are
\begin{align*}
\bV_{\ref{V: N2}} & = \var\{\Niltwo\}, & \bV_{\ref{V: SL}} & = \var\{\SL\}, & \bV_{\ref{V: LZ}} & = \var\{\LZ\}, \\
\bV_{\ref{V: N2 SL}} & = \var\{\Niltwo,\SL\}, & \bV_{\ref{V: JI}} & = \var\{\JI\}, & \bV_{\ref{V: N2 LZ}} & = \var\{\Niltwo,\LZ\}, \\
\bV_{\ref{V: SL LZ}} & = \var\{\SL,\LZ\}, & \bV_{\ref{V: LZi}} & = \var\{\LZi\}, & \bV_{\ref{V: N2 SL LZ}} & = \var\{\Niltwo,\SL,\LZ\}, \\
\bV_{\ref{V: N2 LZi}} & = \var\{\Niltwo,\LZi\}, & \bV_{\ref{V: LZ JI}} & = \var\{\LZ,\JI\}. &&
\end{align*}
\item The proper nontrivial subvarieties of $\bV_{\ref{V: JId RZi}} = \var\{\JIdual,\RZi\}$ are
\begin{align*}
\bV_{\ref{V: N2}} & = \var\{\Niltwo\}, & \bV_{\ref{V: SL}} & = \var\{\SL\}, & \bV_{\ref{V: RZ}} & = \var\{\RZ\}, \\
\bV_{\ref{V: N2 SL}} & = \var\{\Niltwo,\SL\}, & \bV_{\ref{V: JId}} & = \var\{\JIdual\}, & \bV_{\ref{V: SL RZ}} & = \var\{\SL,\RZ\}, \\
\bV_{\ref{V: N2 RZ}} & = \var\{\Niltwo,\RZ\}, & \bV_{\ref{V: RZi}} & = \var\{\RZi\}, & \bV_{\ref{V: N2 SL RZ}} & = \var\{\Niltwo,\SL,\RZ\}, \\
\bV_{\ref{V: RZ JId}} & = \var\{\RZ,\JIdual\}, & \bV_{\ref{V: N2 RZi}} & = \var\{\Niltwo,\RZi\}. &&
\end{align*}
\item The lattices $\sL(\bV_{\ref{V: JI LZi}})$ and $\sL(\bV_{\ref{V: JId RZi}})$ are given in Figure~\ref{F: JI LZi}.
\end{enumerate}
\end{proposition}

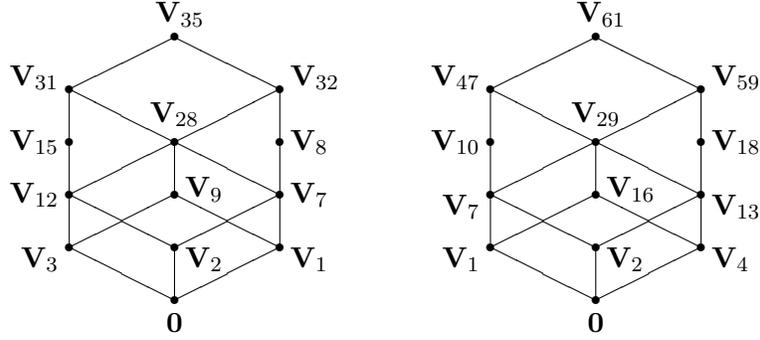
\begin{figure}[ht]
\begin{center}
\begin{picture}(280,130)(00,10) \setlength{\unitlength}{0.7mm}
\put(30,60){\circle*{\circlesize}}
\put(10,50){\circle*{\circlesize}} \put(50,50){\circle*{\circlesize}}
\put(10,40){\circle*{\circlesize}} \put(30,40){\circle*{\circlesize}} \put(50,40){\circle*{\circlesize}}
\put(10,30){\circle*{\circlesize}} \put(30,30){\circle*{\circlesize}} \put(50,30){\circle*{\circlesize}}
\put(10,20){\circle*{\circlesize}} \put(30,20){\circle*{\circlesize}} \put(50,20){\circle*{\circlesize}}
\put(30,10){\circle*{\circlesize}}
\put(10,50){\line(2,1){20}} \put(10,30){\line(2,1){40}} \put(30,20){\line(2,1){20}} \put(10,20){\line(2,1){20}} \put(30,10){\line(2,1){20}}
\put(30,10){\line(-2,1){20}} \put(30,20){\line(-2,1){20}} \put(50,20){\line(-2,1){20}} \put(50,30){\line(-2,1){40}} \put(50,50){\line(-2,1){20}}
\put(10,20){\line(0,1){30}} \put(30,10){\line(0,1){10}} \put(30,30){\line(0,1){10}} \put(50,20){\line(0,1){30}}
\put(31,62){\makebox(0,0)[b]{$\bV_{\ref{V: JI LZi}}$}}
\put(08,50){\makebox(0,0)[br]{$\bV_{\ref{V: N2 LZi}}$}} \put(52,50){\makebox(0,0)[bl]{$\bV_{\ref{V: LZ JI}}$}}
\put(08,40){\makebox(0,0)[r]{$\bV_{\ref{V: LZi}}$}} \put(30,43){\makebox(0,0)[b]{$\bV_{\ref{V: N2 SL LZ}}$}} \put(52,40){\makebox(0,0)[l]{$\bV_{\ref{V: JI}}$}}
\put(08,30){\makebox(0,0)[r]{$\bV_{\ref{V: SL LZ}}$}} \put(32,31){\makebox(0,0)[l]{$\bV_{\ref{V: N2 LZ}}$}} \put(52,30){\makebox(0,0)[l]{$\bV_{\ref{V: N2 SL}}$}}
\put(08,20){\makebox(0,0)[tr]{$\bV_{\ref{V: LZ}}$}} \put(32,20){\makebox(0,0)[tl]{$\bV_{\ref{V: SL}}$}} \put(52,20){\makebox(0,0)[tl]{$\bV_{\ref{V: N2}}$}}
\put(30,07.5){\makebox(0,0)[t]{$\bT$}}
\put(110,60){\circle*{\circlesize}}
\put(90,50){\circle*{\circlesize}} \put(130,50){\circle*{\circlesize}}
\put(90,40){\circle*{\circlesize}} \put(110,40){\circle*{\circlesize}} \put(130,40){\circle*{\circlesize}}
\put(90,30){\circle*{\circlesize}} \put(110,30){\circle*{\circlesize}} \put(130,30){\circle*{\circlesize}}
\put(90,20){\circle*{\circlesize}} \put(110,20){\circle*{\circlesize}} \put(130,20){\circle*{\circlesize}}
\put(110,10){\circle*{\circlesize}}
\put(90,50){\line(2,1){20}} \put(90,30){\line(2,1){40}} \put(110,20){\line(2,1){20}} \put(90,20){\line(2,1){20}} \put(110,10){\line(2,1){20}}
\put(110,10){\line(-2,1){20}} \put(110,20){\line(-2,1){20}} \put(130,20){\line(-2,1){20}} \put(130,30){\line(-2,1){40}} \put(130,50){\line(-2,1){20}}
\put(90,20){\line(0,1){30}} \put(110,10){\line(0,1){10}} \put(110,30){\line(0,1){10}} \put(130,20){\line(0,1){30}}
\put(111,62){\makebox(0,0)[b]{$\bV_{\ref{V: JId RZi}}$}}
\put(88,50){\makebox(0,0)[br]{$\bV_{\ref{V: RZ JId}}$}} \put(132,50){\makebox(0,0)[bl]{$\bV_{\ref{V: N2 RZi}}$}}
\put(88,40){\makebox(0,0)[r]{$\bV_{\ref{V: JId}}$}} \put(110,43){\makebox(0,0)[b]{$\bV_{\ref{V: N2 SL RZ}}$}} \put(132,40){\makebox(0,0)[l]{$\bV_{\ref{V: RZi}}$}}
\put(88,30){\makebox(0,0)[tr]{$\bV_{\ref{V: N2 SL}}$}} \put(112,31){\makebox(0,0)[l]{$\bV_{\ref{V: N2 RZ}}$}} \put(132,30){\makebox(0,0)[tl]{$\bV_{\ref{V: SL RZ}}$}}
\put(88,20){\makebox(0,0)[tr]{$\bV_{\ref{V: N2}}$}} \put(112,20){\makebox(0,0)[tl]{$\bV_{\ref{V: SL}}$}} \put(132,20){\makebox(0,0)[tl]{$\bV_{\ref{V: RZ}}$}}
\put(110,07.5){\makebox(0,0)[t]{$\bT$}}
\end{picture}
\end{center}
\caption{The lattices $\sL(\bV_{\ref{V: JI LZi}})$ and $\sL(\bV_{\ref{V: JId RZi}})$}
\label{F: JI LZi}
\end{figure}

\subsection{Subvarieties of ${ \bV_{\ref{V: JId LZi}} = \var\{\protect\JIdual,\LZi\} }$ and ${ \bV_{\ref{V: JI RZi}} = \var\{\JI,\RZi\} }$}

\begin{proposition}[Zhang and Luo~{\cite[Subvarieties of~$\mathbf{B}$ in Figure~5]{ZL09}}] \quad \label{P: JId LZi}
\begin{enumerate}[\rm(i)]
\item The proper nontrivial subvarieties of $\bV_{\ref{V: JId LZi}} = \var\{\JIdual,\LZi\}$ are
\begin{align*}
\bV_{\ref{V: N2}} & = \var\{\Niltwo\}, & \bV_{\ref{V: SL}} & = \var\{\SL\}, & \bV_{\ref{V: LZ}} & = \var\{\LZ\}, \\
\bV_{\ref{V: N2 SL}} & = \var\{\Niltwo,\SL\}, & \bV_{\ref{V: N2 LZ}} & = \var\{\Niltwo,\LZ\}, & \bV_{\ref{V: JId}} & = \var\{\JIdual\}, \\
\bV_{\ref{V: SL LZ}} & = \var\{\SL,\LZ\}, & \bV_{\ref{V: LZi}} & = \var\{\LZi\}, & \bV_{\ref{V: N2 SL LZ}} & = \var\{\Niltwo,\SL,\LZ\}, \\
\bV_{\ref{V: N2 LZi}} & = \var\{\Niltwo,\LZi\}, & \bV_{\ref{V: LZ JId}} & = \var\{\LZ,\JIdual\}, && \\
\bV_{\ref{V: Cd}} & = \makebox[1cm][l]{$\var\{[11111,11113,11133,11144,11155]\}$.} &&&&
\end{align*}

\item The proper nontrivial subvarieties of $\bV_{\ref{V: JI RZi}} = \var\{\JI,\RZi\}$ are
\begin{align*}
\bV_{\ref{V: N2}} & = \var\{\Niltwo\}, & \bV_{\ref{V: SL}} & = \var\{\SL\}, & \bV_{\ref{V: RZ}} & = \var\{\RZ\}, \\
\bV_{\ref{V: N2 SL}} & = \var\{\Niltwo,\SL\}, & \bV_{\ref{V: JI}} & = \var\{\JI\}, & \bV_{\ref{V: SL RZ}} & = \var\{\SL,\RZ\}, \\
\bV_{\ref{V: N2 RZ}} & = \var\{\Niltwo,\RZ\}, & \bV_{\ref{V: RZi}} & = \var\{\RZi\}, & \bV_{\ref{V: N2 SL RZ}} & = \var\{\Niltwo,\SL,\RZ\}, \\
\bV_{\ref{V: RZ JI}} & = \var\{\RZ,\JI\}, & \bV_{\ref{V: N2 RZi}} & = \var\{\Niltwo,\RZi\}, && \\
\bV_{\ref{V: C}} & = \makebox[1cm][l]{$\var\{[11111,11111,11111,11345,13345]\}$.} &&&&
\end{align*}

\item The lattices $\sL(\bV_{\ref{V: JId LZi}})$ and $\sL(\bV_{\ref{V: JI RZi}})$ are given in Figure~\ref{F: JId LZi}.
\end{enumerate}
\end{proposition}

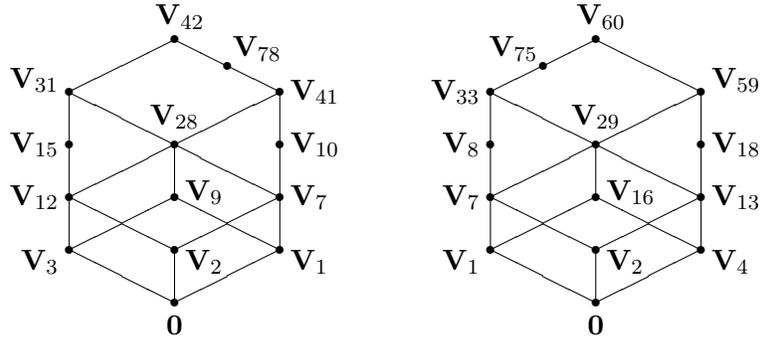
\begin{figure}[ht]
\begin{center}
\begin{picture}(280,130)(00,10) \setlength{\unitlength}{0.7mm}
\put(30,60){\circle*{\circlesize}}
\put(40,55){\circle*{\circlesize}}
\put(10,50){\circle*{\circlesize}} \put(50,50){\circle*{\circlesize}}
\put(10,40){\circle*{\circlesize}} \put(30,40){\circle*{\circlesize}} \put(50,40){\circle*{\circlesize}}
\put(10,30){\circle*{\circlesize}} \put(30,30){\circle*{\circlesize}} \put(50,30){\circle*{\circlesize}}
\put(10,20){\circle*{\circlesize}} \put(30,20){\circle*{\circlesize}} \put(50,20){\circle*{\circlesize}}
\put(30,10){\circle*{\circlesize}}
\put(10,50){\line(2,1){20}} \put(10,30){\line(2,1){40}} \put(30,20){\line(2,1){20}} \put(10,20){\line(2,1){20}} \put(30,10){\line(2,1){20}}
\put(30,10){\line(-2,1){20}} \put(30,20){\line(-2,1){20}} \put(50,20){\line(-2,1){20}} \put(50,30){\line(-2,1){40}} \put(50,50){\line(-2,1){20}}
\put(10,20){\line(0,1){30}} \put(30,10){\line(0,1){10}} \put(30,30){\line(0,1){10}} \put(50,20){\line(0,1){30}}
\put(31,62){\makebox(0,0)[b]{$\bV_{\ref{V: JId LZi}}$}}
\put(41,56){\makebox(0,0)[bl]{$\bV_{\ref{V: Cd}}$}}
\put(08,50){\makebox(0,0)[br]{$\bV_{\ref{V: N2 LZi}}$}} \put(52,50){\makebox(0,0)[l]{$\bV_{\ref{V: LZ JId}}$}}
\put(08,40){\makebox(0,0)[r]{$\bV_{\ref{V: LZi}}$}} \put(30,43){\makebox(0,0)[b]{$\bV_{\ref{V: N2 SL LZ}}$}} \put(52,40){\makebox(0,0)[l]{$\bV_{\ref{V: JId}}$}}
\put(08,30){\makebox(0,0)[r]{$\bV_{\ref{V: SL LZ}}$}} \put(32,31){\makebox(0,0)[l]{$\bV_{\ref{V: N2 LZ}}$}} \put(52,30){\makebox(0,0)[l]{$\bV_{\ref{V: N2 SL}}$}}
\put(08,20){\makebox(0,0)[tr]{$\bV_{\ref{V: LZ}}$}} \put(32,20){\makebox(0,0)[tl]{$\bV_{\ref{V: SL}}$}} \put(52,20){\makebox(0,0)[tl]{$\bV_{\ref{V: N2}}$}}
\put(30,07.5){\makebox(0,0)[t]{$\bT$}}
\put(110,60){\circle*{\circlesize}}
\put(100,55){\circle*{\circlesize}}
\put(90,50){\circle*{\circlesize}} \put(130,50){\circle*{\circlesize}}
\put(90,40){\circle*{\circlesize}} \put(110,40){\circle*{\circlesize}} \put(130,40){\circle*{\circlesize}}
\put(90,30){\circle*{\circlesize}} \put(110,30){\circle*{\circlesize}} \put(130,30){\circle*{\circlesize}}
\put(90,20){\circle*{\circlesize}} \put(110,20){\circle*{\circlesize}} \put(130,20){\circle*{\circlesize}}
\put(110,10){\circle*{\circlesize}}
\put(90,50){\line(2,1){20}} \put(90,30){\line(2,1){40}} \put(110,20){\line(2,1){20}} \put(90,20){\line(2,1){20}} \put(110,10){\line(2,1){20}}
\put(110,10){\line(-2,1){20}} \put(110,20){\line(-2,1){20}} \put(130,20){\line(-2,1){20}} \put(130,30){\line(-2,1){40}} \put(130,50){\line(-2,1){20}}
\put(90,20){\line(0,1){30}} \put(110,10){\line(0,1){10}} \put(110,30){\line(0,1){10}} \put(130,20){\line(0,1){30}}
\put(111,62){\makebox(0,0)[b]{$\bV_{\ref{V: JI RZi}}$}}
\put(99,56){\makebox(0,0)[br]{$\bV_{\ref{V: C}}$}}
\put(88,50){\makebox(0,0)[r]{$\bV_{\ref{V: RZ JI}}$}} \put(132,50){\makebox(0,0)[bl]{$\bV_{\ref{V: N2 RZi}}$}}
\put(88,40){\makebox(0,0)[r]{$\bV_{\ref{V: JI}}$}} \put(110,43){\makebox(0,0)[b]{$\bV_{\ref{V: N2 SL RZ}}$}} \put(132,40){\makebox(0,0)[l]{$\bV_{\ref{V: RZi}}$}}
\put(88,30){\makebox(0,0)[r]{$\bV_{\ref{V: N2 SL}}$}} \put(112,31){\makebox(0,0)[l]{$\bV_{\ref{V: N2 RZ}}$}} \put(132,30){\makebox(0,0)[l]{$\bV_{\ref{V: SL RZ}}$}}
\put(88,20){\makebox(0,0)[tr]{$\bV_{\ref{V: N2}}$}} \put(112,20){\makebox(0,0)[tl]{$\bV_{\ref{V: SL}}$}} \put(132,20){\makebox(0,0)[tl]{$\bV_{\ref{V: RZ}}$}}
\put(110,07.5){\makebox(0,0)[t]{$\bT$}}
\end{picture}
\end{center}
\caption{The lattices $\sL(\bV_{\ref{V: JId LZi}})$ and $\sL(\bV_{\ref{V: JI RZi}})$}
\label{F: JId LZi}
\end{figure}

\subsection{Subvarieties of ${ \bV_{\ref{V: O2}} = \var\{\Otwo\} }$ and ${ \bV_{\ref{V: O2d}} = \var\{\protect\Otwodual\} }$}

\begin{proposition}[{\cite[Figure~7]{Tis07}}] \quad \label{P: O2}
\begin{enumerate}[\rm(i)]
\item The proper nontrivial subvarieties of $\bV_{\ref{V: O2}} = \var\{\Otwo\}$ are
\begin{align*}
\bV_{\ref{V: SL}} & = \var\{\SL\}, & \bV_{\ref{V: LZ}} & = \var\{\LZ\}, & \bV_{\ref{V: Z2}} & = \var\{\Ztwo\}, \\
\bV_{\ref{V: SL LZ}} & = \var\{\SL,\LZ\}, & \bV_{\ref{V: SL Z2}} & = \var\{\SL,\Ztwo\}, & \bV_{\ref{V: LZi}} & = \var\{\LZi\}, \\
\bV_{\ref{V: SL LZ Z2}} & = \var\{\SL,\LZ,\Ztwo\}, & \bV_{\ref{V: Z2 LZi}} & = \var\{\Ztwo,\LZi\}, & \bV_{\ref{V: LZ Z2}} & = \var\{ \LZ,\Ztwo \}.
\end{align*}

\item The proper nontrivial subvarieties of $\bV_{\ref{V: O2d}} = \var\{\Otwodual\}$ are
\begin{align*}
\bV_{\ref{V: SL}} & = \var\{\SL\}, & \bV_{\ref{V: RZ}} & = \var\{\RZ\}, & \bV_{\ref{V: Z2}} & = \var\{\Ztwo\}, \\
\bV_{\ref{V: SL RZ}} & = \var\{\SL,\RZ\}, & \bV_{\ref{V: SL Z2}} & = \var\{\SL,\Ztwo\}, & \bV_{\ref{V: RZi}} & = \var\{\RZi\}, \\
\bV_{\ref{V: SL RZ Z2}} & = \var\{\SL,\RZ,\Ztwo\}, & \bV_{\ref{V: Z2 RZi}} & = \var\{\Ztwo,\RZi\}, & \bV_{\ref{V: RZ Z2}} & = \var\{ \RZ,\Ztwo \}.
\end{align*}

\item The lattices $\sL(\bV_{\ref{V: O2}})$ and $\sL(\bV_{\ref{V: O2d}})$ are given in Figure~\ref{F: O2}.
\end{enumerate}
\end{proposition}

\begin{figure}[ht]
\begin{center}
\begin{picture}(280,130)(00,10) \setlength{\unitlength}{0.7mm}
\put(30,60){\circle*{\circlesize}}
\put(30,50){\circle*{\circlesize}}
\put(10,40){\circle*{\circlesize}} \put(30,40){\circle*{\circlesize}}
\put(10,30){\circle*{\circlesize}} \put(30,30){\circle*{\circlesize}} \put(50,30){\circle*{\circlesize}}
\put(10,20){\circle*{\circlesize}} \put(30,20){\circle*{\circlesize}} \put(50,20){\circle*{\circlesize}}
\put(30,10){\circle*{\circlesize}}
\put(10,40){\line(2,1){20}} \put(10,30){\line(2,1){20}} \put(30,20){\line(2,1){20}} \put(10,20){\line(2,1){20}} \put(30,10){\line(2,1){20}}
\put(30,10){\line(-2,1){20}} \put(30,20){\line(-2,1){20}} \put(50,20){\line(-2,1){20}} \put(50,30){\line(-2,1){20}}
\put(10,20){\line(0,1){20}} \put(30,10){\line(0,1){10}} \put(30,30){\line(0,1){30}} \put(50,20){\line(0,1){10}}
\put(31,62){\makebox(0,0)[b]{$\bV_{\ref{V: O2}}$}}
\put(32,51){\makebox(0,0)[l]{$\bV_{\ref{V: Z2 LZi}}$}}
\put(08,40){\makebox(0,0)[r]{$\bV_{\ref{V: LZi}}$}} \put(32,41){\makebox(0,0)[l]{$\bV_{\ref{V: SL LZ Z2}}$}}
\put(08,30){\makebox(0,0)[r]{$\bV_{\ref{V: SL LZ}}$}} \put(32,31){\makebox(0,0)[l]{$\bV_{\ref{V: LZ Z2}}$}} \put(52,30){\makebox(0,0)[l]{$\bV_{\ref{V: SL Z2}}$}}
\put(08,20){\makebox(0,0)[r]{$\bV_{\ref{V: LZ}}$}} \put(32,18){\makebox(0,0)[l]{$\bV_{\ref{V: SL}}$}} \put(52,20){\makebox(0,0)[l]{$\bV_{\ref{V: Z2}}$}}
\put(30,07.5){\makebox(0,0)[t]{$\bT$}}
\put(110,60){\circle*{\circlesize}}
\put(110,50){\circle*{\circlesize}}
\put(110,40){\circle*{\circlesize}} \put(130,40){\circle*{\circlesize}}
\put(90,30){\circle*{\circlesize}} \put(110,30){\circle*{\circlesize}} \put(130,30){\circle*{\circlesize}}
\put(90,20){\circle*{\circlesize}} \put(110,20){\circle*{\circlesize}} \put(130,20){\circle*{\circlesize}}
\put(110,10){\circle*{\circlesize}}
\put(90,30){\line(2,1){20}} \put(110,20){\line(2,1){20}} \put(90,20){\line(2,1){20}} \put(110,10){\line(2,1){20}}
\put(110,10){\line(-2,1){20}} \put(110,20){\line(-2,1){20}} \put(130,20){\line(-2,1){20}} \put(130,30){\line(-2,1){20}} \put(130,40){\line(-2,1){20}}
\put(90,20){\line(0,1){10}} \put(110,10){\line(0,1){10}} \put(110,30){\line(0,1){30}} \put(130,20){\line(0,1){20}}
\put(111,62){\makebox(0,0)[b]{$\bV_{\ref{V: O2d}}$}}
\put(108,51){\makebox(0,0)[r]{$\bV_{\ref{V: Z2 RZi}}$}}
\put(108,42){\makebox(0,0)[r]{$\bV_{\ref{V: SL RZ Z2}}$}} \put(132,40){\makebox(0,0)[l]{$\bV_{\ref{V: RZi}}$}}
\put(88,30){\makebox(0,0)[r]{$\bV_{\ref{V: SL Z2}}$}} \put(112,31){\makebox(0,0)[l]{$\bV_{\ref{V: RZ Z2}}$}} \put(132,30){\makebox(0,0)[l]{$\bV_{\ref{V: SL RZ}}$}}
\put(88,20){\makebox(0,0)[r]{$\bV_{\ref{V: Z2}}$}} \put(112,18){\makebox(0,0)[l]{$\bV_{\ref{V: SL}}$}} \put(132,20){\makebox(0,0)[l]{$\bV_{\ref{V: RZ}}$}}
\put(110,07.5){\makebox(0,0)[t]{$\bT$}}
\end{picture}
\end{center}
\caption{The lattices $\sL(\bV_{\ref{V: O2}})$ and $\sL(\bV_{\ref{V: O2d}})$}
\label{F: O2}
\end{figure}
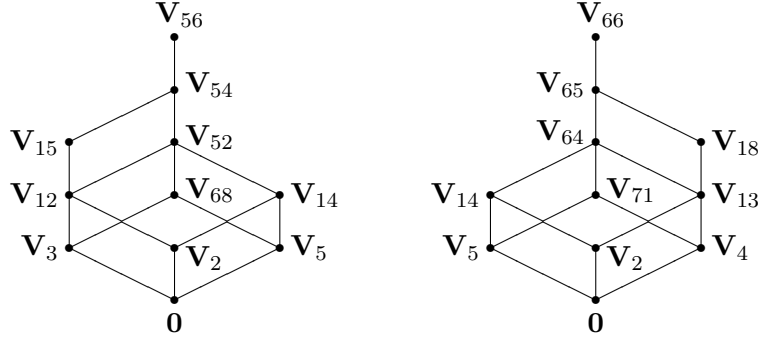

\subsection{Subvarieties of ${ \bV_{\ref{V: N3 P2}} = \var\{\Nilthree,\Ptwo\} }$ and ${ \bV_{\ref{V: N3 P2d}} = \var\{\Nilthree,\protect\Ptwodual\} }$}

\begin{proposition}[Tishchenko~{\cite[Figure~1]{Tis17}}]
\quad \label{P: N3 P2}
\begin{enumerate}[\rm(i)]
\item The proper nontrivial subvarieties of $\bV_{\ref{V: N3 P2}} = \var\{\Nilthree,\Ptwo\}$ are
\begin{align*}
\bV_{\ref{V: N2}} & = \var\{\Niltwo\}, & \bV_{\ref{V: LZ}} & = \var\{\LZ\}, & \bV_{\ref{V: N3}} & = \var\{\Nilthree\}, \\
\bV_{\ref{V: N2 LZ}} & = \var\{\Niltwo,\LZ\}, & \bV_{\ref{V: F4}} & = \var\{\Ffour\}, & \bV_{\ref{V: N3 F4}} & = \var\{\Nilthree,\Ffour\}, \\
\bV_{\ref{V: G4}} & = \var\{\Gfour\}, & \bV_{\ref{V: LZ N3}} & = \var\{\LZ,\Nilthree\}, & \bV_{\ref{V: P2}} & = \var\{\Ptwo\}, \\
\bV_{\ref{V: LZ G4}} & = \var\{\LZ,\Gfour\}, & \bV_{\ref{V: G4 P2}} & = \var\{\Gfour,\Ptwo\}. & &
\end{align*}

\item The proper nontrivial subvarieties of $\bV_{\ref{V: N3 P2d}} = \var\{\Nilthree,\Ptwodual\}$ are
\begin{align*}
\bV_{\ref{V: N2}} & = \var\{\Niltwo\}, & \bV_{\ref{V: RZ}} & = \var\{\RZ\}, & \bV_{\ref{V: N3}} & = \var\{\Nilthree\}, \\
\bV_{\ref{V: N2 RZ}} & = \var\{\Niltwo,\RZ\}, & \bV_{\ref{V: F4}} & = \var\{\Ffour\}, & \bV_{\ref{V: N3 F4}} & = \var\{\Nilthree,\Ffour\}, \\
\bV_{\ref{V: G4}} & = \var\{\Gfour\}, & \bV_{\ref{V: RZ N3}} & = \var\{\RZ,\Nilthree\}, & \bV_{\ref{V: P2d}} & = \var\{\Ptwodual\}, \\
\bV_{\ref{V: RZ G4}} & = \var\{\RZ,\Gfour\}, & \bV_{\ref{V: G4 P2d}} & = \var\{\Gfour,\Ptwodual\}. & &
\end{align*}

\item The lattices $\sL(\bV_{\ref{V: N3 P2}})$ and $\sL(\bV_{\ref{V: N3 P2d}})$ are given in Figure~\ref{F: N3 P2}.
\end{enumerate}
\end{proposition}

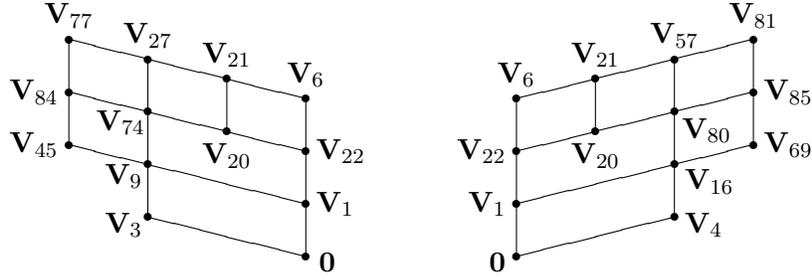
\begin{figure}[ht]
\begin{center}
\begin{picture}(280,110)(00,10) \setlength{\unitlength}{0.7mm}
\put(05,51.25){\circle*{\circlesize}} \put(20,47.5){\circle*{\circlesize}} \put(35,43.75){\circle*{\circlesize}} \put(50,40){\circle*{\circlesize}}
\put(05,41.25){\circle*{\circlesize}} \put(20,37.5){\circle*{\circlesize}} \put(35,33.75){\circle*{\circlesize}} \put(50,30){\circle*{\circlesize}}
\put(05,31.25){\circle*{\circlesize}} \put(20,27.5){\circle*{\circlesize}} \put(50,20){\circle*{\circlesize}}
\put(20,17.5){\circle*{\circlesize}} \put(50,10){\circle*{\circlesize}}
\put(50,40){\line(-4,1){45}} \put(50,30){\line(-4,1){45}} \put(50,20){\line(-4,1){45}} \put(50,10){\line(-4,1){30}}
\put(05,31.25){\line(0,1){20}} \put(20,17.5){\line(0,1){30}} \put(35,33.75){\line(0,1){10}} \put(50,10){\line(0,1){30}}
\put(05,53.5){\makebox(0,0)[b]{$\bV_{\ref{V: N3 P2}}$}} \put(20,49.75){\makebox(0,0)[b]{$\bV_{\ref{V: LZ N3}}$}} \put(35,46){\makebox(0,0)[b]{$\bV_{\ref{V: N3 F4}}$}} \put(50,42.25){\makebox(0,0)[b]{$\bV_{\ref{V: N3}}$}}
\put(03,41.25){\makebox(0,0)[r]{$\bV_{\ref{V: G4 P2}}$}} \put(19.5,35.5){\makebox(0,0)[r]{$\bV_{\ref{V: LZ G4}}$}} \put(35,31.5){\makebox(0,0)[t]{$\bV_{\ref{V: F4}}$}} \put(52,30){\makebox(0,0)[l]{$\bV_{\ref{V: G4}}$}}
\put(03,31.25){\makebox(0,0)[r]{$\bV_{\ref{V: P2}}$}} \put(19,25.25){\makebox(0,0)[r]{$\bV_{\ref{V: N2 LZ}}$}} \put(52,20){\makebox(0,0)[l]{$\bV_{\ref{V: N2}}$}}
\put(19,16){\makebox(0,0)[r]{$\bV_{\ref{V: LZ}}$}} \put(52.5,09){\makebox(0,0)[l]{$\bT$}}
\put(90,40){\circle*{\circlesize}} \put(105,43.75){\circle*{\circlesize}} \put(120,47.5){\circle*{\circlesize}} \put(135,51.25){\circle*{\circlesize}}
\put(90,30){\circle*{\circlesize}} \put(105,33.75){\circle*{\circlesize}} \put(120,37.5){\circle*{\circlesize}} \put(135,41.25){\circle*{\circlesize}}
\put(90,20){\circle*{\circlesize}} \put(120,27.5){\circle*{\circlesize}} \put(135,31.25){\circle*{\circlesize}}
\put(90,10){\circle*{\circlesize}} \put(120,17.5){\circle*{\circlesize}}
\put(90,40){\line(4,1){45}} \put(90,30){\line(4,1){45}} \put(90,20){\line(4,1){45}} \put(90,10){\line(4,1){30}} 
\put(90,10){\line(0,1){30}} \put(105,33.75){\line(0,1){10}} \put(120,17.5){\line(0,1){30}} \put(135,31.25){\line(0,1){20}}
\put(90,42.25){\makebox(0,0)[b]{$\bV_{\ref{V: N3}}$}} \put(105,46){\makebox(0,0)[b]{$\bV_{\ref{V: N3 F4}}$}} \put(120,49.75){\makebox(0,0)[b]{$\bV_{\ref{V: RZ N3}}$}} \put(135,53.5){\makebox(0,0)[b]{$\bV_{\ref{V: N3 P2d}}$}}
\put(88,30){\makebox(0,0)[r]{$\bV_{\ref{V: G4}}$}} \put(105,31.5){\makebox(0,0)[t]{$\bV_{\ref{V: F4}}$}} \put(121.5,34){\makebox(0,0)[l]{$\bV_{\ref{V: RZ G4}}$}} \put(137,41.25){\makebox(0,0)[l]{$\bV_{\ref{V: G4 P2d}}$}}
\put(88,20){\makebox(0,0)[r]{$\bV_{\ref{V: N2}}$}} \put(122,24.25){\makebox(0,0)[l]{$\bV_{\ref{V: N2 RZ}}$}} \put(137,31.25){\makebox(0,0)[l]{$\bV_{\ref{V: P2d}}$}}
\put(88,09){\makebox(0,0)[r]{$\bT$}} \put(122,16){\makebox(0,0)[l]{$\bV_{\ref{V: RZ}}$}}
\end{picture}
\end{center}
\caption{The lattices $\sL(\bV_{\ref{V: N3 P2}})$ and $\sL(\bV_{\ref{V: N3 P2d}})$}
\label{F: N3 P2}
\end{figure}

\subsection{Subvarieties of ${ \bV_{\ref{V: SL N3}} = \var\{\SL,\Nilthree\} }$}

\begin{lemma}[{\cite[Lemma 1.3]{Ver07}}] \label{L: SL V}
Let~$\bV$ be any variety such that $\SL \notin \bV$\up.
\begin{enumerate}[\rm(i)]
\item The lattice $\sL(\bV)$ is isomorphic to the interval \[ \mathscr{I} = [\var\{\SL\}, \var\{\SL\} \vee \bV]. \]
\item The lattice $\sL(\var\{\SL\} \vee \bV)$ is isomorphic to the direct product \[ \sL(\var\{\SL\}) \times \sL(\bV). \]
Consequently\up, $\sL(\var\{\SL\} \vee \bV)$ is the disjoint union of~$\sL(\bV)$ and~$\mathscr{I}$\up.
\end{enumerate}
\end{lemma}

\begin{proposition} \label{P: SL N3}
The lattice $\bV_{\ref{V: SL N3}} = \var\{\SL,\Nilthree\}$ is given in Figure~\ref{F: SL N3}\up.
\end{proposition}

\begin{proof}
By Proposition~\ref{P: N3 P2}, the subvarieties of $\bV_{\ref{V: N3}} = \var\{\Nilthree\}$ constitute the chain $\bT \subset \bV_{\ref{V: N2}} \subset \bV_{\ref{V: G4}} \subset \bV_{\ref{V: N3}}$.
Since $\bV_{\ref{V: SL N3}} = \var\{\SL\} \vee \var\{\Nilthree\}$, the result follows from Lemma~\ref{L: SL V}.
\end{proof}

\begin{figure}[ht]
\begin{center}
\begin{picture}(100,90)(00,20) \setlength{\unitlength}{0.7mm}
\put(40,50){\circle*{\circlesize}}
\put(10,40){\circle*{\circlesize}} \put(40,40){\circle*{\circlesize}}
\put(10,30){\circle*{\circlesize}} \put(40,30){\circle*{\circlesize}}
\put(10,20){\circle*{\circlesize}} \put(40,20){\circle*{\circlesize}}
\put(10,10){\circle*{\circlesize}}
\put(10,40){\line(3,1){30}} \put(10,30){\line(3,1){30}} \put(10,20){\line(3,1){30}} \put(10,10){\line(3,1){30}}
\put(10,10){\line(0,1){30}} \put(40,20){\line(0,1){30}}
\put(42.5,50){\makebox(0,0)[l]{$\bV_{\ref{V: SL N3}} = \var\{\SL,\Nilthree\}$}}
\put(8,40){\makebox(0,0)[r]{$\bV_{\ref{V: N3}} = \var\{\Nilthree\}$}} \put(42.5,40){\makebox(0,0)[l]{$\bV_{\ref{V: SL G4}} = \var\{\SL,\Gfour\}$}}
\put(8,30){\makebox(0,0)[r]{$\bV_{\ref{V: G4}} = \var\{\Gfour\}$}} \put(42.5,30){\makebox(0,0)[l]{$\bV_{\ref{V: N2 SL}} = \var\{\Niltwo,\SL\}$}}
\put(8,20){\makebox(0,0)[r]{$\bV_{\ref{V: N2}} = \var\{\Niltwo\}$}} \put(42.5,20){\makebox(0,0)[l]{$\bV_{\ref{V: SL}} = \var\{\SL\}$}}
\put(8,10){\makebox(0,0)[r]{$\bT$}}
\end{picture}
\end{center}
\caption{The lattice $\sL(\bV_{\ref{V: SL N3}})$}
\label{F: SL N3}
\end{figure}

\subsection{Subvarieties of ${ \var\{\Nilthree,\mathbb{Z}_n\} }$} \label{subsec: N3 Zn}

\begin{lemma} \label{L: N3 Zn}
Let $n \geq 1$ be any integer\up.
\begin{enumerate}[\rm(i)]
\item The lattice $\sL(\var\{\Nilthree,\mathbb{Z}_n\})$ is isomorphic to the direct product \[ \sL(\var\{\Nilthree\}) \times \sL(\var\{\mathbb{Z}_n\}). \]
Consequently\up, $\sL(\var\{\Nilthree,\mathbb{Z}_n\})$ is the disjoint union of the intervals \[ \mathscr{I}_d = [\var\{\mathbb{Z}_d\},\var\{\Nilthree,\mathbb{Z}_d\}], \] where~$d$ ranges over all divisors of~$n$\up.
\item The interval $\mathscr{I}_d$  coincides with the chain \[ \var\{\mathbb{Z}_d\} \subset \var\{\Niltwo,\mathbb{Z}_d\} \subset \var\{\Gfour,\mathbb{Z}_d\} \subset \var\{\Nilthree,\mathbb{Z}_d\}. \]
\end{enumerate}
\end{lemma}

\begin{proof}
(i) This follows from Vernikov~\cite[Proposition~2]{Ver88}.

(ii) This follows from part~(i) since by Figure~\ref{F: SL N3}, the lattice $\sL(\var\{\Nilthree\})$ coincides with the chain $\bT \subset \var\{\Niltwo\} \subset \var\{\Gfour\} \subset \var\{\Nilthree\}$.
\end{proof}

\begin{proposition} \label{P: N3 Zp}
For any prime $p \geq 2$\up, the lattice $\sL(\var\{\Nilthree,\mathbb{Z}_p\})$ is given in Figure~\ref{F: N3 Zn}\up.
\end{proposition}

\begin{proof}
This follows from Lemma~\ref{L: N3 Zn}.
\end{proof}

\begin{figure}[ht]
\begin{center}
\begin{picture}(90,90)(00,20) \setlength{\unitlength}{0.7mm}
\put(40,50){\circle*{\circlesize}}
\put(10,40){\circle*{\circlesize}} \put(40,40){\circle*{\circlesize}}
\put(10,30){\circle*{\circlesize}} \put(40,30){\circle*{\circlesize}}
\put(10,20){\circle*{\circlesize}} \put(40,20){\circle*{\circlesize}}
\put(10,10){\circle*{\circlesize}}
\put(10,40){\line(3,1){30}} \put(10,30){\line(3,1){30}} \put(10,20){\line(3,1){30}} \put(10,10){\line(3,1){30}}
\put(10,10){\line(0,1){30}} \put(40,20){\line(0,1){30}}
\put(42.5,50){\makebox(0,0)[l]{$\var\{\Nilthree,\mathbb{Z}_p\}$}}
\put(8,40){\makebox(0,0)[r]{$\bV_{\ref{V: N3}} = \var\{\Nilthree\}$}} \put(42.5,40){\makebox(0,0)[l]{$\var\{\Gfour,\mathbb{Z}_p\}$}}
\put(8,30){\makebox(0,0)[r]{$\bV_{\ref{V: G4}} = \var\{\Gfour\}$}} \put(42.5,30){\makebox(0,0)[l]{$\var\{\Niltwo,\mathbb{Z}_p\}$}}
\put(8,20){\makebox(0,0)[r]{$\bV_{\ref{V: N2}} = \var\{\Niltwo\}$}} \put(42.5,20){\makebox(0,0)[l]{$\var\{\mathbb{Z}_p\}$}}
\put(8,10){\makebox(0,0)[r]{$\bT$}}
\end{picture}
\end{center}
\caption{The lattice $\sL(\var\{\Nilthree,\mathbb{Z}_p\})$ with prime ${ p \geq 2 }$}
\label{F: N3 Zn}
\end{figure}

\begin{proposition} \label{P: bases N3 Zn} Let $n \geq 2$ be any integer\up.
Then the identities
\begin{subequations} \label{id: N3 Zn}
\begin{align}
x^n abc & \approx abc, \label{id: N3 Zn xnabc=abc} \\
xy & \approx yx \label{id: N3 Zn xy=yx}
\end{align}
\end{subequations}
constitute an {\ib} for the variety $\var\{\Nilthree,\mathbb{Z}_n\}$\up.
\end{proposition}

\begin{proof}
It is routinely checked that the identities~\eqref{id: N3 Zn} are satisfied by the variety $\var\{\Nilthree,\mathbb{Z}_n\}$.
Therefore it remains to show that any nontrivial identity $\bu \approx \bv$ satisfied by $\var\{\Nilthree,\mathbb{Z}_n\}$ is deducible from~\eqref{id: N3 Zn}.
By Lemma~\ref{L: LZ LZi N3 Nni Zn word} parts~(iii) and~(v), the following properties hold:
\begin{enumerate}[\ \ (a)]
\item either $|\bu|,|\bv| \geq 3$ or $\occ(x,\bu) = \occ(x,\bv)$ for all $x \in \sX$;
\item $\occ(x,\bu) \equiv \occ(x,\bv) \pmod n$ for all variables~$x$.
\end{enumerate}
If $\occ(x,\bu) = \occ(x,\bv)$ for all $x \in \sX$, then it is clear that the identity $\bu \approx \bv$ is deducible from~\eqref{id: N3 Zn xy=yx}.
Therefore suppose that $|\bu|,|\bv| \geq 3$.
Generality is not lost by assuming that $\con(\bu) = \{ x_1,x_2, \ldots,x_k \}$ and $\con(\bv) \backslash \con(\bu) = \{ y_1,y_2,\ldots,y_m \}$ for some $k \geq 1$ and $m \geq 0$.
Let $e_i = \occ(x_i,\bu)$, so that $\sum_{i=1}^k e_i = |\bu| \geq 3$.
By~(b), there exist $r_i, s_j \geq 1$ such that $\occ(x_i,\bv) = e_i +r_in \geq 0$ and $\occ(y_j,\bv) = s_jn \geq 0$.
Let $r_i' \geq 1$ be any integer such that $r_i+r_i' \geq 1$.
Then
\begin{align*}
\bv & \stackrel{\eqref{id: N3 Zn xnabc=abc}}{\approx} \makebox[1.75in][l]{$\displaystyle\bigg(\prod_{i=1}^k x_i^{r_i'n} \bigg) \bv$} \text{since $|\bv| \geq 3$} \\
& \stackrel{\eqref{id: N3 Zn xy=yx}}{\approx} \bigg(\prod_{i=1}^k x_i^{r_i+r_i'}\prod_{i=1}^m y_i^{s_i} \bigg)^{\!n} \, \prod_{i=1}^k x_i^{e_i} \\
& \stackrel{\eqref{id: N3 Zn xnabc=abc}}{\approx} \makebox[1.75in][l]{$\displaystyle\prod_{i=1}^k x_i^{e_i}$} \text{since $\sum_{i=1}^k e_i \geq 3$} \\
& \stackrel{\eqref{id: N3 Zn xy=yx}}{\approx} \bu. \qedhere
\end{align*}
\end{proof}

\begin{proposition}
Let $n \geq 2$ be any integer\up.
Then the identities
\begin{equation}
x^n abc \approx abc, \quad xy \approx yx, \quad x^{n+2} \approx x^2 \label{id: G4 Zn}
\end{equation}
constitute an {\ib} for the variety $\var\{\Gfour,\mathbb{Z}_n\}$\up.
\end{proposition}

\begin{proof}
Let~$\bW$ denote the variety defined by the identities~\eqref{id: G4 Zn}.
Then it is routinely checked that the inclusions $\var\{\Gfour,\mathbb{Z}_n\} \subseteq \bW \subseteq \var\{\Nilthree,\mathbb{Z}_n\}$ hold.
But the semigroup~$\Nilthree$ violates the last identity in~\eqref{id: G4 Zn}, so that $\bW \neq \var\{\Nilthree,\mathbb{Z}_n\}$.
Therefore $\bW = \var\{\Gfour,\mathbb{Z}_n\}$ by Lemma~\ref{L: N3 Zn}(ii).
\end{proof}

\subsection{Subvarieties of ${ \var\{\JI,\mathbb{Z}_p\} }$ and ${ \var\{\SL,\mathbb{Z}_{p^2}\} }$} \label{subsec: JI Zp SL Zpp}

\begin{lemma}[{\cite[Part~(b) of the main theorem]{Sap91}}] \label{L: Sapir}
Let~$\mathbf{G}$ be any periodic variety generated by a group\up.
Then each subvariety of $\var\{\JI \} \vee \mathbf{G}$ is the join of some subvariety of~$\mathbf{G}$ with some of the following varieties\up: \[ \bT, \quad \bV_{\ref{V: N2}} = \var\{\Niltwo\}, \quad \bV_{\ref{V: SL}} = \var\{\SL\}, \quad \bV_{\ref{V: JI}} = \var\{\JI\}. \]
\end{lemma}

\begin{proposition} \label{P: JI Zp SL Zpp}
Let $p \geq 2$ be any prime\up.
\begin{enumerate}[\rm(i)]
\item The lattice $\sL(\var\{\JI,\mathbb{Z}_p\})$ is given in Figure~\ref{F: JI Zp}\up.
\item The lattice $\sL(\var\{\SL,\mathbb{Z}_{p^2}\})$ is given in Figure~\ref{F: SL Zpp}\up.
\end{enumerate}
\end{proposition}

\begin{proof}
This follows from Lemma~\ref{L: Sapir}.
\end{proof}

\begin{figure}[ht]
\begin{center}
\begin{picture}(270,180)(00,10) \setlength{\unitlength}{0.7mm}
\put(40,70){\circle*{\circlesize}}
\put(40,50){\circle*{\circlesize}}
\put(20,30){\circle*{\circlesize}} \put(60,30){\circle*{\circlesize}}
\put(40,10){\circle*{\circlesize}}
\put(100,90){\circle*{\circlesize}}
\put(100,70){\circle*{\circlesize}}
\put(80,50){\circle*{\circlesize}} \put(120,50){\circle*{\circlesize}}
\put(100,30){\circle*{\circlesize}}
\put(40,70){\line(3,1){60}} \put(40,50){\line(3,1){60}} \put(20,30){\line(3,1){60}} \put(60,30){\line(3,1){60}} \put(40,10){\line(3,1){60}}
\put(40,50){\line(0,1){20}} \put(100,70){\line(0,1){20}}
\put(20,30){\line(1,1){20}} \put(40,10){\line(1,1){20}} \put(80,50){\line(1,1){20}} \put(100,30){\line(1,1){20}}
\put(20,30){\line(1,-1){20}} \put(40,50){\line(1,-1){20}} \put(80,50){\line(1,-1){20}} \put(100,70){\line(1,-1){20}}
\put(38,70){\makebox(0,0)[r]{$\bV_{\ref{V: JI}} = \var\{\JI\}$}}
\put(38,50){\makebox(0,0)[r]{$\bV_{\ref{V: N2 SL}} = \var\{\Niltwo,\SL\}$}}
\put(18,30){\makebox(0,0)[r]{$\bV_{\ref{V: N2}} = \var\{\Niltwo\}$}} \put(60,25){\makebox(0,0)[l]{$\bV_{\ref{V: SL}} = \var\{\SL\}$}}
\put(40,07.5){\makebox(0,0)[t]{$\bT$}}
\put(103,90){\makebox(0,0)[l]{$\var\{\JI,\mathbb{Z}_p\}$}}
\put(103,70){\makebox(0,0)[l]{$\var\{\Niltwo,\SL,\mathbb{Z}_p\}$}}
\put(83,50){\makebox(0,0)[l]{$\var\{\Niltwo,\mathbb{Z}_p\}$}} \put(122,50){\makebox(0,0)[l]{$\var\{\SL,\mathbb{Z}_p\}$}}
\put(103,30){\makebox(0,0)[l]{$\var\{\mathbb{Z}_p\}$}}
\end{picture}
\end{center}
\caption{The lattice $\sL(\var\{\JI,\mathbb{Z}_p\})$ with prime ${ p \geq 2 }$}
\label{F: JI Zp}
\end{figure}

\begin{figure}[ht]
\begin{center}
\begin{picture}(100,80)(00,20) \setlength{\unitlength}{0.7mm}
\put(40,40){\circle*{\circlesize}}
\put(10,30){\circle*{\circlesize}} \put(40,30){\circle*{\circlesize}}
\put(10,20){\circle*{\circlesize}} \put(40,20){\circle*{\circlesize}}
\put(10,10){\circle*{\circlesize}}
\put(10,30){\line(3,1){30}} \put(10,20){\line(3,1){30}} \put(10,10){\line(3,1){30}}
\put(10,10){\line(0,1){20}} \put(40,20){\line(0,1){20}}
\put(42.5,40){\makebox(0,0)[l]{$\var\{\SL,\mathbb{Z}_{p^2}\}$}}
\put(8,30){\makebox(0,0)[r]{$\var\{\mathbb{Z}_{p^2}\}$}} \put(42.5,30){\makebox(0,0)[l]{$\var\{\SL,\mathbb{Z}_p\}$}}
\put(8,20){\makebox(0,0)[r]{$\var\{\mathbb{Z}_p\}$}} \put(42.5,20){\makebox(0,0)[l]{$\var\{\SL\}$}}
\put(8,10){\makebox(0,0)[r]{$\bT$}}
\end{picture}
\end{center}
\caption{The lattice $\sL(\var\{\SL,\mathbb{Z}_{p^2}\})$ with prime ${ p \geq 2 }$}
\label{F: SL Zpp}
\end{figure}
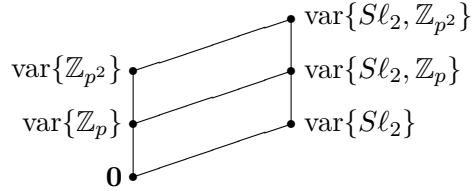

\begin{proposition} \label{P: bases JI Zn} Let $n \geq 2$ be any integer\up.
\begin{enumerate}[\rm(i)]
\item The identities
\begin{subequations} \label{id: JI Zn}
\begin{align}
x^{n+1}a & \approx xa, \label{id: JI Zn xn+1a=xa} \\
x^{m_1}y^{m_2} & \approx y^{m_2}x^{m_1}, \quad m_1,m_2 \geq 2, \label{id: JI Zn xxyy=yyxx} \\
xya & \approx yxa. \label{id: JI Zn xya=yxa}
\end{align}
\end{subequations}
constitute an {\ib} for the variety $\var\{\JI,\mathbb{Z}_n\}$\up.
\item The identities
\[
x^{n+1}a \approx xa, \quad x^2y^2 \approx y^2x^2, \quad xya \approx yxa
\]
also constitute an {\ib} for the variety $\var\{\JI,\mathbb{Z}_n\}$\up.
\end{enumerate}
\end{proposition}

\begin{proof}
(i) It is routinely checked that the identities~\eqref{id: JI Zn} are satisfied by the variety $\var\{\JI,\mathbb{Z}_n\}$.
Therefore it remains to show that any identity $\bu \approx \bv$ satisfied by $\var\{\JI,\mathbb{Z}_n\}$ is deducible from~\eqref{id: JI Zn}.
By Lemma~\ref{L: JI word}, generality is not lost by assuming that $\con(\bu) = \con(\bv) = \{x_1,x_2,\ldots,x_m \}$, so that $e_i = \occ(x_i,\bu) \geq 1$ and $f_i = \occ(x_i,\bv) \geq 1$.
Then $e_i \equiv f_i \pmod n$ by Lemma~\ref{L: LZ LZi N3 Nni Zn word}(v).
By Lemma~\ref{L: JI word}, there are two cases.

\noindent\textsc{Case~1}: $\tail(\bu) = \tail(\bv) = x_k$ with either $e_k = f_k = 1$ or $e_k,f_k \geq 2$.
Then
\begin{align*}
\bu & \stackrel{\eqref{id: JI Zn xya=yxa}}{\approx} \bigg(\prod_{i \neq k} x_i^{e_i}\bigg) x_k^{e_k} \\
& \stackrel{\eqref{id: JI Zn xn+1a=xa}}{\approx} \bigg(\prod_{i \neq k} x_i^{f_i}\bigg) x_k^{f_k} \quad \text{since ${ e_i \equiv f_i \ (\bmod n) }$} \\
& \stackrel{\eqref{id: JI Zn xya=yxa}}{\approx} \bv.
\end{align*}

\noindent\textsc{Case~2}: $\tail(\bu) = x_k$ and $\tail(\bv) = x_\ell$ with $k < \ell$ and $e_k, f_\ell \geq 2$.
Choose any integer $g_i > \max\{e_i,f_i\}$ such that $g_i \equiv e_i \equiv f_i \pmod n$.
Then
\begin{align*}
\bu & \stackrel{\eqref{id: JI Zn xya=yxa}}{\approx} \bigg(\prod_{i \neq k,\ell} x_i^{e_i}\bigg) x_\ell^{e_\ell} x_k^{e_k} \\
& \stackrel{\eqref{id: JI Zn xn+1a=xa}}{\approx} \bigg(\prod_{i \neq k,\ell} x_i^{g_i}\bigg) x_\ell^{g_\ell} x_k^{g_k} \quad \text{since ${ g_i \equiv e_i \;(\bmod\; n) }$ and ${ e_k \geq 2 }$} \\
& \stackrel{\eqref{id: JI Zn xxyy=yyxx}}{\approx} \bigg(\prod_{i \neq k,\ell} x_i^{g_i}\bigg) x_k^{g_k} x_\ell^{g_\ell} \\
& \stackrel{\eqref{id: JI Zn xn+1a=xa}}{\approx} \bigg(\prod_{i \neq k,\ell} x_i^{f_i}\bigg) x_k^{f_k} x_\ell^{f_\ell} \quad \text{since ${ g_i \equiv f_i \;(\bmod\; n) }$ and ${ f_\ell \geq 2 }$} \\
& \stackrel{\eqref{id: JI Zn xya=yxa}}{\approx} \bv.
\end{align*}

(ii) It suffices to show that the identities~\eqref{id: JI Zn xxyy=yyxx} are deducible from the identities $\alpha: x^2y^2 \approx y^2x^2$ and $\beta: xya \approx yxa$.
Write $m_i = 2p_i+r_i$ where $p_i \geq 1$ and $r_i \in \{ 0,1 \}$.
Then
\[
x^{m_1}y^{m_2} \stackrel{\beta}{\approx} y^{r_2} x^{r_1} x^{2p_1}  y^{2p_2} \stackrel{\alpha}{\approx} y^{r_2} x^{r_1} y^{2p_2} x^{2p_1} \stackrel{\beta}{\approx} y^{m_2}x^{m_1}. \qedhere
\]
\end{proof}

\begin{proposition}[{\cite[Lemma~7.3 and Diagram~8]{Pet74}}] \label{P: bases N2 SL Zn N2 Zn SL Zn}
Let $n \geq 2$ be any integer\up.
\begin{enumerate}[\rm(i)]
\item The identities \[ x^{n+1} a \approx xa, \quad xy \approx yx \] constitute an {\ib} for the variety $\var\{\Niltwo,\SL,\mathbb{Z}_n\}$\up.
\item The identities \[ x^n ab \approx ab, \quad xy \approx yx \] constitute an {\ib} for the variety $\var\{\Niltwo,\mathbb{Z}_n\}$\up.
\item The identities \[ x^{n+1} \approx x, \quad xy \approx yx \] constitute an {\ib} for the variety $\var\{\SL,\mathbb{Z}_n\}$\up.
\end{enumerate}
\end{proposition}

\subsection{Subvarieties of ${ \bV_{\ref{V: N4}} = \var\{\Nilfour\} }$} \label{subsec: N4}

\begin{proposition}[Mel'nik~{\cite[Subvarieties of $B_{23}$ in Figure~3]{Mel72}}] \label{P: N4} \quad
\begin{enumerate}[\rm(i)]
\item The proper nontrivial subvarieties of $\bV_{\ref{V: N4}} = \var\{\Nilfour\}$ are
\begin{align*}
\bV_{\ref{V: N2}} & = \var\{\Niltwo\}, \qquad \bV_{\ref{V: N3}} = \var\{\Nilthree\}, \qquad \bV_{\ref{V: G4}} = \var\{\Gfour\}, \\
\bV_{\ref{V: CN 3 4}} & = \var\!\left\{ \! \begin{array}[c]{l} [1111111,1111111,1111112,1111121,1111122,1112235, \\ \phantom{[} 1121254] \end{array} \! \right\}\!, \\
\bV_{\ref{V: CN 2 4}} & = \var\!\left\{ \! \begin{array}[c]{l} [1111\,1111,1111\,1111,1111\,1112,1111\,1121,1111\,1211, \\ \phantom{[} 1111\,2134,1112\,1315,1121\,1451] \end{array} \! \right\}\!, \\
\bV_{\ref{V: CN 21 4}} & = \var\!\left\{ \! \begin{array}[c]{l} [1111\,1111,1111\,1111,1111\,1112,1111\,1121,1111\,1211, \\ \phantom{[} 1111\,2134,1112\,1315,1121\,1452] \end{array} \! \right\}\!.
\end{align*}

\item The lattice $\sL(\bV_{\ref{V: N4}})$ is given in Figure~\ref{F: N4}\up.
\end{enumerate}
\end{proposition}

\begin{figure}[htbp]
\begin{center}
\begin{picture}(90,140)(00,10) \setlength{\unitlength}{0.7mm}
\put(20,70){\circle*{\circlesize}}
\put(20,60){\circle*{\circlesize}}
\put(20,50){\circle*{\circlesize}}
\put(10,40){\circle*{\circlesize}} \put(30,40){\circle*{\circlesize}}
\put(20,30){\circle*{\circlesize}}
\put(20,20){\circle*{\circlesize}}
\put(20,10){\circle*{\circlesize}}
\put(10,40){\line(1,1){10}} \put(20,30){\line(1,1){10}} 
\put(10,40){\line(1,-1){10}} \put(20,50){\line(1,-1){10}}
\put(20,10){\line(0,1){20}} \put(20,50){\line(0,1){20}}
\put(22.5,70){\makebox(0,0)[l]{$\bV_{\ref{V: N4}}$}}
\put(22.5,60){\makebox(0,0)[l]{$\bV_{\ref{V: CN 3 4}}$}}
\put(22.5,50){\makebox(0,0)[l]{$\bV_{\ref{V: CN 21 4}}$}}
\put(8,40){\makebox(0,0)[r]{$\bV_{\ref{V: N3}}$}} \put(32.5,40){\makebox(0,0)[l]{$\bV_{\ref{V: CN 2 4}}$}}
\put(22.5,28){\makebox(0,0)[l]{$\bV_{\ref{V: G4}}$}}
\put(22.5,19){\makebox(0,0)[l]{$\bV_{\ref{V: N2}}$}}
\put(22.5,10){\makebox(0,0)[l]{$\bT$}}
\end{picture}
\end{center}
\caption{The lattice $\sL(\bV_{\ref{V: N4}})$}
\label{F: N4}
\end{figure}

\section{Some varieties with infinitely many subvarieties} \label{app: infinite}

\subsection{The variety ${ \var\{\mathbb{Z}_p,\Nil_n^1\} }$} \label{subsec: Zp Nni}

\begin{proposition} \label{P: Zp Nni}
Let $p \geq 1$ and $n \geq 2$ be any integers\up.
\begin{enumerate}[\rm(i)]
\item The identities
\begin{subequations} \label{id: Zp Nni}
\begin{align}
x^{n+p} & \approx x^n, \label{id: Zp Nni xn+p=xn} \\
xy & \approx yx \label{id: Zp Nni xy=yx}
\end{align}
\end{subequations}
constitute an {\ib} for the variety $\var\{\mathbb{Z}_p,\Nil_n^1\}$\up.
\item The variety $\var\{\mathbb{Z}_p,\Nil_n^1\}$ contains countably infinitely many subvarieties\up.
\end{enumerate}
\end{proposition}

\begin{proof}
(i) It is routinely checked that the identities~\eqref{id: Zp Nni} are satisfied by the variety $\var\{\mathbb{Z}_p,\Nil_n^1\}$.
Hence it remains to show that any identity $\bu \approx \bv$ satisfied by $\var\{\mathbb{Z}_p,\Nil_n^1\}$ is deducible from~\eqref{id: Zp Nni}.
Generality is not lost by assuming that $\bu, \bv \in \{ x_1,x_2,\ldots,x_m \}^*$ with $e_i = \occ(x_i,\bu)$ and $f_i = \occ(x_i,\bv)$.
Then it follows from Lemma~\ref{L: LZ LZi N3 Nni Zn word} parts~(iv) and~(v) that for each~$i$,
\begin{enumerate}[\ \ (a)]
\item either $e_i = f_i < n$ or $e_i,f_i \geq n$;
\item $e_i \equiv f_i \pmod p$.
\end{enumerate}
If $e_i \neq f_i$ for some~$i$, then $e_i,f_i \in \{ n+rp \,|\, r \geq 0 \}$ by~(a) and~(b), whence the identity $x^{e_i} \approx x^{f_i}$ is deducible from~\eqref{id: Zp Nni xn+p=xn}.
It follows that
\[
\bu \stackrel{\eqref{id: Zp Nni xy=yx}}{\approx} \prod_{i=1}^m x_i^{e_i} \stackrel{\eqref{id: Zp Nni xn+p=xn}}{\approx} \prod_{i=1}^m x_i^{f_i} \stackrel{\eqref{id: Zp Nni xy=yx}}{\approx} \bv
\]

(ii) Any variety of commutative semigroups is finitely based~\cite{Per69}.
Hence by Lemma~\ref{L: HFB implies countable}, the variety $\var\{\mathbb{Z}_p,\Nil_n^1\}$ contains countably many subvarieties.
The result then holds since the subvariety $\var\{\Niltwo^1\}$ of $\var\{\mathbb{Z}_p,\Nil_n^1\}$ contains infinitely many subvarieties \cite[Figure~5(b)]{Eva71}.
\end{proof}

\begin{corollary} \label{C: Nni}
Let $n \geq 2$ be any integer\up.
\begin{enumerate}[\rm(i)]
\item The identities
\begin{subequations} \label{id: Nni}
\begin{align}
x^{n+1} & \approx x^n, \label{id: Nni xn+1=xn} \\
xy & \approx yx \label{id: Nni xy=yx}
\end{align}
\end{subequations}
constitute an {\ib} for the variety $\var\{\Nil_n^1\}$\up.
\item The variety $\var\{\Nil_n^1\}$ contains countably infinitely many subvarieties\up.
\end{enumerate}
\end{corollary}

\begin{lemma}[Lee {\etal}.~{\cite[Proposition~5.10]{LRS19}}] \label{L: Nni max}
Each proper subvariety of $\var\{\Nil_n^1\}$ satisfies the identity
\begin{equation}
x^ny^{n-1} \approx x^{n-1} y^n. \label{id: Nni max}
\end{equation}
\end{lemma}

\begin{proposition}
Let $n \geq 2$ be any integer\up.
\begin{enumerate}[\rm(i)]
\item The variety $\var\{ \eqref{id: Nni},\eqref{id: Nni max} \}$ is the only maximal subvariety of $\var\{\Nil_n^1\}$\up.
\item The variety $\var\{ \eqref{id: Nni},\eqref{id: Nni max} \}$ is not finitely generated\up.
\end{enumerate}
\end{proposition}

\begin{proof}
(i) This follows from Corollary~\ref{C: Nni}(i) and Lemma~\ref{L: Nni max}.

(ii) It is easily seen that the variety $\var\{ \eqref{id: Nni},\eqref{id: Nni max} \}$ violates the identity
\begin{equation}
x_1 x_2 \cdots x_m y^n \approx x_1 x_2 \cdots x_m y^{n-1} \label{id: Nni finite}
\end{equation}
for any $m \geq 1$.
Hence it suffices to show that each finite semigroup~$S$ in the variety $\var\{ \eqref{id: Nni},\eqref{id: Nni max} \}$ satisfies the identity~\eqref{id: Nni finite} for all $m \geq n|S|$.
Choose any elements $a_1,a_2,\ldots,a_m,b \in S$.
Then the list $a_1,a_2,\ldots,a_m$ contains some element $a \in S$ at least~$n$ times, due to the magnitude of~$m$.
Therefore $a_1 a_2 \cdots a_m \stackrel{\eqref{id: Nni xy=yx}}{=} sa^n$ for some $s \in S$, whence
\begin{align*}
a_1a_2 \cdots a_m b^n & \stackrel{\eqref{id: Nni xy=yx}}{=} sa^nb^n \stackrel{\eqref{id: Nni max}}{=} sa^{n+1}b^{n-1} \\
& \stackrel{\eqref{id: Nni xn+1=xn}}{=} sa^nb^{n-1} \stackrel{\eqref{id: Nni xy=yx}}{=} a_1a_2 \cdots a_m b^{n-1}. \qedhere
\end{align*}
\end{proof}

\begin{lemma} \label{L: Zp Nni max}
Let $p \geq 2$ be any prime and $n \geq 2$ be any integer\up.
Then each proper subvariety of $\var\{\mathbb{Z}_p,\Nil_n^1\}$ satisfies one of the following identities\up:
\begin{align}
x^{n-1+p} y^{n-1} & \approx x^{n-1} y^{n-1+p}, \label{id: Zk Nni max excl Nni} \\
x^{n+1} & \approx x^n. \label{id: Zk Nni max excl Zp}
\end{align}
\end{lemma}

\begin{proof}
Let~$\bW$ be any proper subvariety of $\var\{\mathbb{Z}_p,\Nil_n^1\}$.
Then either $\mathbb{Z}_p \notin \bW$ or $\Nil_n^1 \notin \bW$.
First suppose that $\Nil_n^1 \notin \bW$.
Then it follows from Lemma~\ref{L: excl Nni} that the variety~$\bW$ satisfies the identity $\alpha: (x^ny)^{n-1+p} x^n \approx (x^ny)^{n-1} x^n$.
Let $r \geq 1$ be such that $n^2 + r \equiv n \pmod p$.
Then since
\begin{align*}
x^{n-1+p} y^n & \stackrel{\eqref{id: Zp Nni xn+p=xn}}{\approx} x^{n-1+p} y^{n^2+r} = x^{n-1+p} (y^{n})^n y^r \\
& \stackrel{\eqref{id: Zp Nni xn+p=xn}}{\approx} x^{n-1+p} (y^{n})^{n+p} y^r \stackrel{\eqref{id: Zp Nni xy=yx}}{\approx} (y^nx)^{n-1+p} y^ny^r \\
& \makebox[0.38in]{$\stackrel{\alpha}{\approx}$} (y^nx)^{n-1} y^ny^r \stackrel{\eqref{id: Zp Nni xy=yx}}{\approx} x^{n-1} y^{n^2+r} \stackrel{\eqref{id: Zp Nni xn+p=xn}}{\approx} x^{n-1} y^n,
\end{align*}
it follows that~$\bW$ satisfies the identity $\beta: x^{n-1+p} y^{n-1+p} \approx x^{n-1} y^{n-1+p}$.
But since
\begin{align*}
x^{n-1} y^{n-1+p} & \stackrel{\beta}{\approx} x^{n-1+p} y^{n-1+p} \stackrel{\eqref{id: Zp Nni xy=yx}}{\approx} y^{n-1+p} x^{n-1+p} \\
& \stackrel{\beta}{\approx} y^{n-1} x^{n-1+p} \stackrel{\eqref{id: Zp Nni xy=yx}}{\approx} x^{n-1+p} y^{n-1},
\end{align*}
the variety~$\bW$ also satisfies the identity~\eqref{id: Zk Nni max excl Nni}.

It remains to assume that $\mathbb{Z}_p \notin \bW$.
Then by Lemma~\ref{L: LZ LZi N3 Nni Zn word}(v), the variety~$\bW$ satisfies an identity $\gamma: \bu \approx \bv$ with $\occ(x,\bu) \not\equiv \occ(x,\bv) \pmod p$ for some variable $x \in \sX$.
Generality is not lost with the assumption that $e \equiv \occ(x,\bu) \pmod p$ and $f \equiv \occ(x,\bv) \pmod p$ with $0 \leq e < f \leq p-1$.
Let~$\varphi$ denote the substitution that fixes~$x$ and maps every other variable to~$x^p$.
Then
\[
x^{n+e} \stackrel{\eqref{id: Zp Nni xy=yx}}{\approx} \big(\varphi(\bu)\big)x^n \stackrel{\gamma}{\approx} \big(\varphi(\bv)\big)x^n \stackrel{\eqref{id: Zp Nni xy=yx}}{\approx} x^{n+f},
\]
so that the variety~$\bW$ satisfies the identity $\delta: x^{n+e} \approx x^{n+f}$.
Since
\[
x^n \stackrel{\eqref{id: Zp Nni xn+p=xn}}{\approx} x^{n+e} x^{p-e} \stackrel{\delta}{\approx} x^{n+f} x^{p-e} \stackrel{\eqref{id: Zp Nni xn+p=xn}}{\approx} x^{n+f-e},
\]
the variety~$\bW$ satisfies the identity $\varepsilon: x^n \approx x^{n+\ell}$ for some $\ell \geq 1$.
Since~$p$ is prime, there exists some $m \geq 1$ such that $m \ell \equiv 1 \pmod p$.
Therefore
\[
x^n \stackrel{\varepsilon}{\approx} x^{n+\ell} \stackrel{\varepsilon}{\approx} x^{n+2\ell} \stackrel{\varepsilon}{\approx} \cdots \stackrel{\varepsilon}{\approx} x^{n+m\ell} \stackrel{\eqref{id: Zp Nni xn+p=xn}}{\approx} x^{n+1},
\]
so that the variety~$\bW$ satisfies the identity~\eqref{id: Zk Nni max excl Zp}.
\end{proof}

\begin{proposition}
For any prime $p \geq 2$ and integer $n \geq 2$\up, let \[ \bU = \var\{\eqref{id: Zp Nni},\eqref{id: Zk Nni max excl Nni}\} \quad \text{and} \quad \bV = \var\{\eqref{id: Zp Nni},\eqref{id: Zk Nni max excl Zp}\}. \]
Then
\begin{enumerate}[\rm(i)]
\item $\bU$ and~$\bV$ are precisely all maximal subvarieties of  $\var\{\mathbb{Z}_p,\Nil_n^1\}$\up;
\item $\bU$ is not finitely generated\up;
\item $\bV = \var\{\Nil_n^1\}$\up.
\end{enumerate}
\end{proposition}

\begin{proof}
(i) Since~$\mathbb{Z}_p$ satisfies $\{ \eqref{id: Zp Nni},\eqref{id: Zk Nni max excl Nni} \}$ and violates~\eqref{id: Zk Nni max excl Zp}, while $\Nil_n^1$ satisfies $\{\eqref{id: Zp Nni},\eqref{id: Zk Nni max excl Zp}\}$ and violates~\eqref{id: Zk Nni max excl Nni}, the varieties~$\bU$ and~$\bV$ are incomparable.
The result then follows from Lemma~\ref{L: Zp Nni max}.

(ii) It is easily seen that the variety~$\bU$ violates the identity
\begin{equation}
x^{n-1+p}y_1 y_2 \cdots y_m \approx x^{n-1}y_1 y_2 \cdots y_m \label{id: Zp Nni finite}
\end{equation}
for any $m \geq 1$.
Hence it suffices to show that each finite semigroup~$S$ in~$\bU$ satisfies the identity~\eqref{id: Zp Nni finite} for all $m \geq (n+p)|S|$.
Choose any elements $a,b_1,b_2,\ldots,b_m \in S$.
Then the list $b_1,b_2,\ldots,b_m$ contains some element $b \in S$ at least $n+p$ times, due to the magnitude of~$m$.
Therefore $b_1 b_2 \cdots b_m \stackrel{\eqref{id: Zp Nni xy=yx}}{=} b^{n+p}s$ for some $s \in S$, whence
\begin{align*}
a^{n-1}b_1b_2 \cdots b_m & \stackrel{\eqref{id: Zp Nni xy=yx}}{=} a^{n-1} b^{n+p}s \stackrel{\eqref{id: Zk Nni max excl Nni}}{=} a^{n-1+p} b^ns \\
& \stackrel{\eqref{id: Zp Nni xn+p=xn}}{=} a^{n-1+p} b^{n+p} s \stackrel{\eqref{id: Zp Nni xy=yx}}{=} a^{n-1+p} b_1b_2 \cdots b_m.
\end{align*}

(iii) This follows from Corollary~\ref{C: Nni}(i).
\end{proof}

\subsection{The varieties ${ \var\{\JI,\Nil_n^1\} }$ and ${ \var\{\protect\JIdual,\Nil_n^1\} }$} \label{subsec: JI Nni}

\begin{proposition} \label{P: bases JI Nni} Let $n \geq 2$ be any integer\up.
\begin{enumerate}[\rm(i)]
\item The identities
\begin{subequations} \label{id: JI Nni}
\begin{align}
x^{n+1} & \approx x^n, \label{id: JI Nni xn+1=xn} \\
x^{m_1}y^{m_2} & \approx y^{m_2}x^{m_2}, \quad m_1,m_2 \in \{ 2,3,4,\ldots\}, \label{id: JI Nni xxyy=yyxx} \\
xya & \approx yxa \label{id: JI Nni xya=yxa}
\end{align}
\end{subequations}
constitute an {\ib} for the variety $\var\{\JI,\Nil_n^1\}$\up.

\item The identities
\[
x^{n+1} \approx x^n, \quad x^2y^2 \approx y^2x^2, \quad xya \approx yxa
\]
also constitute an {\ib} for the variety $\var\{\JI,\Nil_n^1\}$\up.
\item The variety $\var\{\JI,\Nil_n^1\}$ contains countably infinitely many subvarieties\up.
\end{enumerate}
\end{proposition}

\begin{proof}
(i) It is routinely checked that the identities~\eqref{id: JI Nni} are satisfied by the variety $\var\{\JI,\Nil_n^1\}$.
Hence it remains to show that any identity $\bu \approx \bv$ satisfied by $\var\{\JI,\Nil_n^1\}$ is deducible from~\eqref{id: JI Nni}.
By Lemma~\ref{L: JI word}, generality is not lost by assuming that $\con(\bu) = \con(\bv) = \{x_1,x_2,\ldots,x_m \}$, so that $e_i = \occ(x_i,\bu) \geq 1$ and $f_i = \occ(x_i,\bv) \geq 1$.
Further, it follows from Lemma~\ref{L: LZ LZi N3 Nni Zn word}(iv) that
\begin{enumerate}[\ \ (a)]
\item for each~$i$, either $e_i = f_i < n$ or $e_i,f_i \geq n$.
\end{enumerate}
There are two cases.

\noindent\textsc{Case~1}: $\tail(\bu) = \tail(\bv) = x_k$.
Then
\begin{align*}
\bu & \stackrel{\eqref{id: JI Nni xya=yxa}}{\approx} \bigg(\prod_{i \neq k} x_i^{e_i}\bigg) x_k^{e_k} \\
& \stackrel{\eqref{id: JI Nni xn+1=xn}}{\approx} \bigg(\prod_{i \neq k} x_i^{f_i}\bigg) x_k^{f_k} \quad \text{by~(a)} \\
& \stackrel{\eqref{id: JI Nni xya=yxa}}{\approx} \bv.
\end{align*}

\noindent\textsc{Case~2}: $\tail(\bu) = x_k$ and $\tail(\bv) = x_\ell$ with $k < \ell$.
Then by~(a) and Lemma~\ref{L: JI word},
\begin{enumerate}[\ \ (a)]
\item[(b)] $e_k, f_k, e_\ell, f_\ell \geq 2$.
\end{enumerate}
Hence
\begin{align*}
\bu & \stackrel{\eqref{id: JI Nni xya=yxa}}{\approx} \bigg(\prod_{i \neq k,\ell} x_i^{e_i}\bigg) x^{e_\ell} x^{e_k} \\
& \stackrel{\eqref{id: JI Nni xxyy=yyxx}}{\approx} \bigg(\prod_{i \neq k,\ell} x_i^{e_i}\bigg) x^{e_k} x^{e_\ell} \quad \text{by~(b)} \\
& \stackrel{\eqref{id: JI Nni xn+1=xn}}{\approx} \bigg(\prod_{i \neq k,\ell} x_i^{f_i}\bigg) x^{f_k} x^{f_\ell} \quad \text{by~(a)} \\
& \stackrel{\eqref{id: JI Nni xya=yxa}}{\approx} \bv.
\end{align*}

(ii) As shown in the proof of Proposition~\ref{P: bases JI Zn}(ii), the identities~\eqref{id: JI Nni xxyy=yyxx} are deducible from $x^2y^2 \approx y^2x^2$ and $xya \approx yxa$.
The result thus follows from part~(i).

(iii) Any finitely generated variety that satisfies the identity~\eqref{id: JI Nni xya=yxa} is finitely based~\cite{Per69}.
Hence by Lemma~\ref{L: HFB implies countable}, the variety $\var\{\JI,\Nil_n^1\}$ contains countably many subvarieties.
The result then holds since the subvariety $\var\{\Niltwo^1\}$ of $\var\{\JI,\Nil_n^1\}$ contains infinitely many subvarieties \cite[Figure~5(b)]{Eva71}.
\end{proof}

\begin{lemma} \label{L: JI Nni max}
Let $n \geq 2$ be any integer\up.
Then each proper subvariety of $\var\{\JI,\Nil_n^1\}$ satisfies one of the following identities\up:
\begin{align}
x^{n-1}y^n & \approx y^{n-1}x^n, \label{id: JI Nni max xn-1yn=yn-1xn} \\
x^ny & \approx yx^n. \label{id: JI Nni max xny=yxn}
\end{align}
\end{lemma}

\begin{proof}
Let~$\bW$ be any proper subvariety of $\var\{\JI,\Nil_n^1\}$.
Then either $\JI \notin \bW$ or $\Nil_n^1 \notin \bW$.
First suppose that $\Nil_n^1 \notin \bW$.
Then by Lemma~\ref{L: excl Nni}, the variety~$\bW$ satisfies the identity~\eqref{id: excl Nni} with $k = 1$.
Since
\[
x^{n-1} y^n \stackrel{\eqref{id: JI Nni}}{\approx} (y^nx)^{n-1} y^n \stackrel{\eqref{id: excl Nni}}{\approx} (y^nx)^n y^n \stackrel{\eqref{id: JI Nni}}{\approx} x^n y^n,
\]
the variety~$\bW$ satisfies the identity $\alpha : x^{n-1} y^n \approx x^n y^n$; since
\[
x^{n-1} y^n \stackrel{\alpha}{\approx} x^n y^n \stackrel{\eqref{id: JI Nni xxyy=yyxx}}{\approx} y^n x^n \stackrel{\alpha}{\approx} y^{n-1} x^n,
\]
it also satisfies the identity~\eqref{id: JI Nni max xn-1yn=yn-1xn}.

It remains to assume that $\JI \notin \bW$, so that
by Lemma~\ref{L: excl JI}, the variety~$\bW$ satisfies one of the identities~\eqref{id: excl JI xnyn+1=xny} and~\eqref{id: excl JI xnyxn=xny}.
Since
\begin{align*}
x^ny & \stackrel{\eqref{id: excl JI xnyn+1=xny}}{\approx} (x^ny)^{n+1} \stackrel{\eqref{id: JI Nni}}{\approx} (x^ny)^{n+1} x^n \stackrel{\eqref{id: excl JI xnyn+1=xny}}{\approx} x^nyx^n \stackrel{\eqref{id: JI Nni}}{\approx} yx^n \\
\text{and} \quad x^ny & \stackrel{\eqref{id: excl JI xnyxn=xny}}{\approx} x^nyx^n \stackrel{\eqref{id: JI Nni}}{\approx} yx^n,
\end{align*}
the variety~$\bW$ also satisfies the identity~\eqref{id: JI Nni max xny=yxn}.
\end{proof}

\begin{proposition}
For any integer $n \geq 2$\up, let \[ \bU = \var\{\eqref{id: JI Nni},\eqref{id: JI Nni max xn-1yn=yn-1xn}\} \quad \text{and} \quad \bV = \var\{\eqref{id: JI Nni},\eqref{id: JI Nni max xny=yxn}\}. \]
Then
\begin{enumerate}[\rm(i)]
\item $\bU$ and~$\bV$ are precisely all maximal subvarieties of  $\var\{\JI,\Nil_n^1\}$\up;
\item $\bU$ is not finitely generated\up;
\item $\bV$ is not finitely generated\up.
\end{enumerate}
\end{proposition}

\begin{proof}
(i) Since the semigroup~$\JI$ satisfies $\{ \eqref{id: JI Nni},\eqref{id: JI Nni max xn-1yn=yn-1xn} \}$ and violates~\eqref{id: JI Nni max xny=yxn}, while the semigroup~$\Nil_n^1$ satisfies $\{ \eqref{id: JI Nni},\eqref{id: JI Nni max xny=yxn} \}$ and violates~\eqref{id: JI Nni max xn-1yn=yn-1xn}, the varieties~$\bU$ and~$\bV$ are incomparable.
The result then follows from Lemma~\ref{L: JI Nni max}.

(ii) It is easily seen that the variety $\bU$ violates the identity
\begin{equation}
x^ny_1 y_2 \cdots y_m \approx x^{n-1}y_1 y_2 \cdots y_m \label{id: JI Nni finite1}
\end{equation}
for any $m \geq 1$.
Hence it suffices to show that each finite semigroup~$S$ in~$\bU$ satisfies the identity~\eqref{id: JI Nni finite1} for all $m \geq (n+1)|S|$.
Choose any $a,b_1,b_2,\ldots,b_m \in S$.
Then the list $b_1,b_2,\ldots,b_m$ contains some element $b \in S$ at least $n+1$ times, due to the magnitude of~$m$.
Therefore $b_1 b_2 \cdots b_m \stackrel{\eqref{id: JI Nni xya=yxa}}{=} b^nsbt$ for some $s,t \in S^1$, whence
\begin{align*}
a^{n-1}b_1b_2 \cdots b_m & \stackrel{\eqref{id: JI Nni xya=yxa}}{=} a^{n-1} b^nsbt \stackrel{\eqref{id: JI Nni xn+1=xn}}{=} a^{n-1} b^nbsbt \stackrel{\eqref{id: JI Nni max xn-1yn=yn-1xn}}{=} b^{n-1} a^nbsbt \\
& \stackrel{\eqref{id: JI Nni xya=yxa}}{=} a^nb^nsbt \stackrel{\eqref{id: JI Nni xya=yxa}}{=} a^nb_1b_2 \cdots b_m.
\end{align*}

(iii) It is easily seen that the variety~$\bV$ violates the identity
\begin{equation}
x_1 x_2 \cdots x_myz \approx x_1 x_2 \cdots x_mzy \label{id: JI Nni finite2}
\end{equation}
for any $m \geq 1$.
Hence it suffices to show that each finite semigroup~$S$ in~$\bV$ satisfies the identity~\eqref{id: JI Nni finite2} for all $m \geq n|S|$.
Choose any elements $a_1,a_2,\ldots,a_m,b,c \in S$.
Then the list $a_1,a_2,\ldots,a_m$ contains some element $a \in S$ at least~$n$ times, due to the magnitude of~$m$.
Thus $a_1 a_2 \cdots a_mb \stackrel{\eqref{id: JI Nni xya=yxa}}{=} sa^nb$ and $a_1 a_2 \cdots a_mc \stackrel{\eqref{id: JI Nni xya=yxa}}{=} sa^nc$ for some $s \in S$, whence
\begin{align*}
a_1a_2 \cdots a_mbc & \stackrel{\eqref{id: JI Nni xya=yxa}}{=} sa^nbc \stackrel{\eqref{id: JI Nni max xny=yxn}}{=} sbca^n \stackrel{\eqref{id: JI Nni xya=yxa}}{=} scba^n \\
& \stackrel{\eqref{id: JI Nni max xny=yxn}}{=} sa^ncb \stackrel{\eqref{id: JI Nni xya=yxa}}{=} a_1a_2 \cdots a_mcb. \qedhere
\end{align*}
\end{proof}

\begin{corollary} \label{C: bases JId Nni} Let $n \geq 2$ be any integer\up.
Then
\begin{enumerate}[\rm(i)]
\item the identities
\[
x^{n+1} \approx x^n, \quad x^2y^2 \approx y^2x^2, \quad axy \approx ayx
\]
constitute an {\ib} for the variety $\var\{\JIdual,\Nil_n^1\}$\up;
\item $\var\{\JIdual,\Nil_n^1\}$ contains countably infinitely many subvarieties\up;
\item $\var\{\JIdual,\Nil_n^1\}$ contains precisely two maximal subvarieties\up.
\end{enumerate}
\end{corollary}

\subsection{The varieties ${ \var\{\LZ,\Nil_n^1\} }$ and ${ \var\{\RZ,\Nil_n^1\} }$} \label{subsec: LZ Nni}

\begin{proposition} \label{P: bases LZ Nni} Let $n \geq 2$ be any integer\up.
\begin{enumerate}[\rm(i)]
\item The identities
\begin{subequations} \label{id: LZ Nni}
\begin{align}
x^{n+1} & \approx x^n, \label{id: LZ Nni xn+1=xn} \\
axy & \approx ayx. \label{id: LZ Nni axy=ayx}
\end{align}
\end{subequations}
constitute an {\ib} for the variety $\var\{\LZ,\Nil_n^1\}$\up.
\item The variety $\var\{\LZ,\Nil_n^1\}$ contains countably infinitely many subvarieties\up.
\end{enumerate}
\end{proposition}

\begin{proof}
(i) It is routinely checked that the identities~\eqref{id: LZ Nni} are satisfied by the variety $\var\{\LZ,\Nil_n^1\}$.
Therefore it remains to show that any identity $\bu \approx \bv$ satisfied by $\var\{\LZ,\Nil_n^1\}$ is deducible from the identities~\eqref{id: LZ Nni}.
Generality is not lost by assuming that $\bu, \bv \in \{ x_1,x_2,\ldots,x_m\}^*$ with $e_i = \occ(x_i,\bu)$ and $f_i = \occ(x_i,\bv)$.
By Lemma~\ref{L: LZ LZi N3 Nni Zn word} parts~(i) and~(iv),
\begin{enumerate}[\ \ (a)]
\item $\head(\bu) = \head(\bv) = x_k$ for some~$k$;
\item for each~$i$, either $e_i = f_i < n$ or $e_i,f_i \geq n$.
\end{enumerate}
Hence
\begin{align*}
\bu & \stackrel{\eqref{id: LZ Nni axy=ayx}}{\approx} x_k^{e_k} \prod_{i \neq k} x_i^{e_i} \\
& \stackrel{\eqref{id: LZ Nni xn+1=xn}}{\approx} x_k^{f_k} \prod_{i \neq k} x_i^{f_i} \quad \text{by~(b)} \\
& \stackrel{\eqref{id: LZ Nni axy=ayx}}{\approx} \bv.
\end{align*}

(ii) See the proof of Proposition~\ref{P: bases JI Nni}(iii).
\end{proof}

\begin{lemma} \label{L: LZ Nni max}
Let $n \geq 2$ be any integer\up.
Then each proper subvariety of the variety ${ \var\{\LZ,\Nil_n^1\} }$ satisfies one of the following identities\up:
\begin{align}
x^ny^n & \approx y^nx^n, \label{id: LZ Nni max xnyn=ynxn} \\
a^nx^n & \approx a^nx^{n-1}. \label{id: LZ Nni max anxn=anxn-1}
\end{align}
\end{lemma}

\begin{proof}
Let~$\bW$ be any proper subvariety of $\var\{\LZ,\Nil_n^1\}$.
Then either $\LZ \notin \bW$ or $\Nil_n^1 \notin \bW$.
First suppose that $\LZ \notin \bW$.
Then the variety~$\bW$ satisfies the identity $\alpha : x^n(yx^n)^n \approx (yx^n)^n$ \cite[Theorem~5.15]{LRS19}.
Since
\[
x^ny^n \stackrel{\eqref{id: LZ Nni}}{\approx} x^n(yx^n)^n \stackrel{\alpha}{\approx} (yx^n)^n \stackrel{\eqref{id: LZ Nni}}{\approx} y^nx^n,
\]
the variety~$\bW$ satisfies the identity~\eqref{id: LZ Nni max xnyn=ynxn}.

It remains to assume that $\Nil_n^1 \notin \bW$.
Then by Lemma~\ref{L: excl Nni}, the variety~$\bW$ satisfies the identity~\eqref{id: excl Nni} with $k = 1$.
Since
\[
a^nx^n \stackrel{\eqref{id: LZ Nni}}{\approx} (a^nx)^n a^n \stackrel{\eqref{id: excl Nni}}{\approx} (a^nx)^{n-1}a^n \stackrel{\eqref{id: LZ Nni}}{\approx} a^nx^{n-1},
\]
the variety~$\bW$ satisfies the identity~\eqref{id: LZ Nni max anxn=anxn-1}.
\end{proof}

\begin{proposition}
For any integer $n \geq 2$\up, let \[ \bU = \var\{\eqref{id: LZ Nni},\eqref{id: LZ Nni max xnyn=ynxn}\} \quad \text{and} \quad \bV = \var\{\eqref{id: LZ Nni},\eqref{id: LZ Nni max anxn=anxn-1}\}. \]
Then
\begin{enumerate}[\rm(i)]
\item $\bU$ and~$\bV$ are the only maximal subvarieties of $\var\{\LZ,\Nil_n^1\}$\up;
\item $\bU = \var\{ \JIdual,\Niltwo^1 \}$ if $n = 2$\up;
\item $\bV$ is not finitely generated\up.
\end{enumerate}
\end{proposition}

\begin{proof}
(i) Since the semigroup~$\LZ$ satisfies $\{ \eqref{id: LZ Nni},\eqref{id: LZ Nni max anxn=anxn-1} \}$ and violates~\eqref{id: LZ Nni max xnyn=ynxn}, while the semigroup~$\Nil_n^1$ satisfies $\{ \eqref{id: LZ Nni},\eqref{id: LZ Nni max xnyn=ynxn} \}$ and violates~\eqref{id: LZ Nni max anxn=anxn-1}, the varieties~$\bU$ and~$\bV$ are incomparable.
The result then follows from Lemma~\ref{L: LZ Nni max}.

(ii) This follows from the dual of Proposition~\ref{P: bases JI Nni}(ii).

(iii) It is easily seen that the variety $\bV$ violates the identity
\begin{equation}
x_1 x_2 \cdots x_m y^n \approx x_1 x_2 \cdots x_m y^{n-1} \label{id: LZ Nni finite}
\end{equation}
for any $m \geq 1$.
Hence it suffices to show that each finite semigroup~$S$ in~$\bV$ satisfies the identity~\eqref{id: LZ Nni finite} for all $m \geq (n+1)|S|$.
Choose any $a_1,a_2,\ldots,a_m,b \in S$.
Then the list $a_1,a_2,\ldots,a_m$ contains some element $a \in S$ at least $n+1$ times, due to the magnitude of~$m$.
Therefore $a_1 a_2 \cdots a_m \stackrel{\eqref{id: LZ Nni axy=ayx}}{=} sata^n$ for some $s,t \in S^1$, whence
\[
a_1 a_2 \cdots a_mb^n \stackrel{\eqref{id: LZ Nni axy=ayx}}{=} sata^nb^n \stackrel{\eqref{id: LZ Nni max anxn=anxn-1}}{=} sata^nb^{n-1}  \stackrel{\eqref{id: LZ Nni axy=ayx}}{=} a_1 a_2 \cdots a_mb^{n-1}. \qedhere
\]
\end{proof}

\begin{corollary} \label{C: bases RZ Nni} Let $n \geq 2$ be any integer\up.
Then
\begin{enumerate}[\rm(i)]
\item the identities
\[
x^{n+1} \approx x^n, \quad xya \approx yxa
\]
constitute an {\ib} for the variety $\var\{\RZ,\Nil_n^1\}$\up;
\item $\var\{\RZ,\Nil_n^1\}$ contains countably infinitely many subvarieties\up;
\item $\var\{\RZ,\Nil_n^1\}$ contains precisely two maximal subvarieties\up.
\end{enumerate}
\end{corollary}

\subsection{The varieties ${ \var\{\LZi,\Nil_n^1\} }$ and ${ \var\{\RZi,\Nil_n^1\} }$} \label{subsec: LZi Nni}

\begin{proposition} \label{P: bases LZi Nni} Let $n \geq 2$ be any integer\up.
\begin{enumerate}[\rm(i)]
\item The identities
\begin{subequations} \label{id: LZi Nni}
\begin{align}
x^{n+1} & \approx x^n, \label{id: LZi Nni xn+1=xn} \\
xyx & \approx x^2y. \label{id: LZi Nni xyx=xxy}
\end{align}
\end{subequations}
constitute an {\ib} for the variety $\var\{\LZi,\Nil_n^1\}$\up.
\item The variety $\var\{\LZi,\Nil_n^1\}$ contains countably infinitely many subvarieties\up.
\end{enumerate}
\end{proposition}

\begin{proof}
(i) It is routinely checked that the identities~\eqref{id: LZi Nni} are satisfied by the variety $\var\{\LZi,\Nil_n^1\}$.
Therefore it remains to show that any identity $\bu \approx \bv$ satisfied by $\var\{\LZi,\Nil_n^1\}$ is deducible from the identities~\eqref{id: LZi Nni}.
In view of Lemma~\ref{L: LZ LZi N3 Nni Zn word}(ii), generality is not lost by assuming that
\begin{enumerate}[\ \ (a)]
\item $\ini(\bu) = \ini(\bv) = \prod_{i=1}^m x_i$,
\end{enumerate}
so that $e_i = \occ(x_i,\bu) \geq 1$ and $f_i = \occ(x_i,\bv) \geq 1$.
By Lemma~\ref{L: LZ LZi N3 Nni Zn word}(iv),
\begin{enumerate}[\ \ (a)]
\item[(b)] for each~$i$, either $e_i = f_i < n$ or $e_i,f_i \geq n$.
\end{enumerate}
Hence
\begin{align*}
\bu & \stackrel{\eqref{id: LZi Nni xyx=xxy}}{\approx} \makebox[0.6in][l]{$\displaystyle\prod_{i=1}^m x_i^{e_i}$} \text{by~(a)} \\
& \stackrel{\eqref{id: LZi Nni xn+1=xn}}{\approx} \makebox[0.6in][l]{$\displaystyle\prod_{i=1}^m x_i^{f_i}$} \text{by~(b)} \\
& \stackrel{\eqref{id: LZi Nni xyx=xxy}}{\approx} \makebox[0.6in][l]{$\bv$} \text{by~(a)}.
\end{align*}

(ii) Any variety that satisfies the identity~\eqref{id: LZi Nni xyx=xxy} is finitely based~\cite{Pol82}.
Hence by Lemma~\ref{L: HFB implies countable}, the variety $\var\{\LZi,\Nil_n^1\}$ contains countably many subvarieties.
The result then holds since the subvariety $\var\{\Niltwo^1\}$ of $\var\{\LZi,\Nil_n^1\}$ contains infinitely many subvarieties \cite[Figure~5(b)]{Eva71}.
\end{proof}

\begin{lemma} \label{L: LZi Nni max}
Let $n \geq 2$ be any integer\up.
Then each proper subvariety of the variety $\var\{\LZi,\Nil_n^1\}$ satisfies one of the following identities\up:
\begin{align}
a^nx^ny^n & \approx a^ny^nx^n, \label{id: LZi Nni max anxnyn=anynxn} \\
a^nx^n & \approx a^nx^{n-1}. \label{id: LZi Nni max anxn=anxn-1}
\end{align}
\end{lemma}

\begin{proof}
Let~$\bW$ be any proper subvariety of $\var\{\LZi,\Nil_n^1\}$.
Then either $\LZi \notin \bW$ or $\Nil_n^1 \notin \bW$.
First suppose that $\LZi \notin \bW$.
Then the variety~$\bW$ satisfies the identity $\alpha : a^n(xa^n)^n \big(ya^n(xa^n)^n\big)^n \approx a^n\big(ya^n(xa^n)^n\big)^n$ \cite[Theorem~5.17]{LRS19}.
Since
\[
a^nx^ny^n \stackrel{\eqref{id: LZi Nni}}{\approx} a^n(xa^n)^n \big(ya^n(xa^n)^n\big)^n \stackrel{\alpha}{\approx} a^n\big(ya^n(xa^n)^n\big)^n \stackrel{\eqref{id: LZi Nni}}{\approx} a^ny^nx^n,
\]
the variety~$\bW$ satisfies the identity~\eqref{id: LZi Nni max anxnyn=anynxn}.

It remains to assume that $\Nil_n^1 \notin \bW$.
Then by Lemma~\ref{L: excl Nni}, the variety~$\bW$ satisfies the identity~\eqref{id: excl Nni} with $k = 1$.
Since
\[
a^nx^n \stackrel{\eqref{id: LZi Nni}}{\approx} (a^nx)^n a^n \stackrel{\eqref{id: excl Nni}}{\approx} (a^nx)^{n-1}a^n \stackrel{\eqref{id: LZi Nni}}{\approx} a^nx^{n-1},
\]
the variety~$\bW$ satisfies the identity~\eqref{id: LZi Nni max anxn=anxn-1}.
\end{proof}

\begin{proposition}
For any integer $n \geq 2$\up, let \[ \bU = \var\{\eqref{id: LZi Nni},\eqref{id: LZi Nni max anxnyn=anynxn}\} \quad \text{and} \quad \bV = \var\{\eqref{id: LZi Nni},\eqref{id: LZi Nni max anxn=anxn-1}\}. \]
Then
\begin{enumerate}[\rm(i)]
\item $\bU$ and~$\bV$ are the only maximal subvarieties of $\var\{\LZi,\Nil_n^1\}$\up;
\item $\bU$ is not finitely generated\up;
\item $\bV$ is not finitely generated\up.
\end{enumerate}
\end{proposition}

\begin{proof}
(i) Since the semigroup~$\LZi$ satisfies $\{ \eqref{id: LZi Nni},\eqref{id: LZi Nni max anxn=anxn-1} \}$ and violates~\eqref{id: LZi Nni max anxnyn=anynxn}, while the semigroup~$\Nil_n^1$ satisfies $\{ \eqref{id: LZi Nni},\eqref{id: LZi Nni max anxnyn=anynxn} \}$ and violates~\eqref{id: LZi Nni max anxn=anxn-1}, the varieties~$\bU$ and~$\bV$ are incomparable.
The result then follows from Lemma~\ref{L: LZi Nni max}.

(ii) It is easily seen that the variety~$\bU$ violates the identity
\begin{equation}
x_1 x_2 \cdots x_m y^nz^n \approx x_1 x_2 \cdots x_m z^ny^n \label{id: LZi Nni finite1}
\end{equation}
for any $m \geq 1$.
Hence it suffices to show that each finite semigroup~$S$ in the variety~$\bU$ satisfies the identity~\eqref{id: LZi Nni finite1} for all $m \geq n|S|$.
Choose any elements $a_1,a_2,\ldots,a_m,b,c \in S$.
Then the list $a_1,a_2,\ldots,a_m$ contains some element $a \in S$ at least~$n$ times, due to the magnitude of~$m$.
Therefore $a_1 a_2 \cdots a_m \stackrel{\eqref{id: LZi Nni xyx=xxy}}{=} sa^nt$ for some $s,t \in S^1$, whence
\begin{align*}
a_1 a_2 \cdots a_mb^nc^n & \stackrel{\eqref{id: LZi Nni xyx=xxy}}{=} sa^ntb^nc^n \stackrel{\eqref{id: LZi Nni}}{=} sa^nta^nb^nc^n \stackrel{\eqref{id: LZi Nni max anxnyn=anynxn}}{=} sa^nta^nc^nb^n \\
& \stackrel{\eqref{id: LZi Nni}}{=} sa^ntc^nb^n \stackrel{\eqref{id: LZi Nni xyx=xxy}}{=} a_1 a_2 \cdots a_mc^nb^n.
\end{align*}

(iii) It is easily seen that the variety $\bV$ violates the identity
\begin{equation}
x_1 x_2 \cdots x_m y^n \approx x_1 x_2 \cdots x_m y^{n-1} \label{id: LZi Nni finite2}
\end{equation}
for any $m \geq 1$.
Hence it suffices to show that each finite semigroup~$S$ in the variety~$\bV$ satisfies the identity~\eqref{id: LZi Nni finite2} for all $m \geq n|S|$.
Choose any elements $a_1,a_2,\ldots,a_m,b \in S$.
Then the list $a_1,a_2,\ldots,a_m$ contains some element $a \in S$ at least~$n$ times, due to the magnitude of~$m$.
Therefore $a_1 a_2 \cdots a_m \stackrel{\eqref{id: LZi Nni xyx=xxy}}{=} sa^nt$ for some $s,t \in S^1$, whence
\begin{align*}
a_1 a_2 \cdots a_mb^n & \stackrel{\eqref{id: LZi Nni xyx=xxy}}{=} sa^ntb^n \stackrel{\eqref{id: LZi Nni}}{=} sa^nta^nb^n \stackrel{\eqref{id: LZi Nni max anxn=anxn-1}}{=} sa^nta^nb^{n-1} \\
& \stackrel{\eqref{id: LZi Nni}}{=} sa^ntb^{n-1} \stackrel{\eqref{id: LZi Nni xyx=xxy}}{=} a_1 a_2 \cdots a_mb^{n-1}. \qedhere
\end{align*}
\end{proof}

\begin{corollary} \label{C: bases RZi Nni} Let $n \geq 2$ be any integer\up.
Then
\begin{enumerate}[\rm(i)]
\item the identities
\[
x^{n+1} \approx x^n, \quad xyx \approx yx^2
\]
constitute an {\ib} for the variety $\var\{\RZi,\Nil_n^1\}$\up;
\item $\var\{\RZi,\Nil_n^1\}$ contains countably infinitely many subvarieties\up;
\item $\var\{\RZi,\Nil_n^1\}$ contains precisely two maximal subvarieties\up.
\end{enumerate}
\end{corollary}

\subsection{The varieties ${ \bV_{\ref{V: B0}} = \var\{\Bz\} }$ and ${ \bV_{\ref{V: A0}} = \var\{\Az\} }$} \label{subsec: B0 A0}

\begin{proposition}[Edmunds~{\cite[Semigroups $\mathtt{S(4,21)}$ and $\mathtt{S(4,22)}$ on page~70]{Edm80}}; Lee~\cite{Lee04}] \quad
\begin{enumerate}[\rm(i)]
\item The identities \[ x^3 \approx x^2, \quad x^2yx^2 \approx yxy, \quad x^2y^2 \approx y^2x^2 \] constitute an {\ib} for the variety $\bV_{\ref{V: B0}} = \var\{\Bz\}$\up.
\item The identities \[ x^3 \approx x^2, \quad x^2yx^2 \approx yxy \] constitute an {\ib} for the variety $\bV_{\ref{V: A0}} = \var\{\Az\}$\up.
\item The varieties $\var\{\Bz\}$ and $\var\{\Az\}$ each contains countably infinitely many subvarieties\up.
\end{enumerate}
\end{proposition}

\begin{proposition}[Lee~\cite{Lee04,Lee06}] \quad
\begin{enumerate}[\rm(i)]
\item The variety $\var\{\Bz\}$ is the unique maximal subvariety of $\var\{\Az\}$\up.
\item The identities
\begin{equation}
x^3 \approx x^2, \quad x^2yx^2 \approx yxy, \quad x^2y^2 \approx y^2x^2, \quad a^2x^2b^2 \approx a^2xb^2. \label{id: B0 max}
\end{equation}
constitute an {\ib} for the unique maximal subvariety of $\var\{\Bz\}$\up.
\item The unique maximal subvariety of $\var\{\Bz\}$ is not finitely generated\up.
\end{enumerate}
\end{proposition}

\subsection{The varieties ${ \bV_{\ref{V: JIi}} = \var\{\JIi\} }$ and ${ \bV_{\ref{V: JIid}} = \var\{\protect\JIidual\} }$} \label{subsec: JIi}

\begin{proposition} \label{P: JIi} \quad
\begin{enumerate}[\rm(i)]
\item The identities
\begin{subequations} \label{id: JIi}
\begin{align}
x^3 & \approx x^2, \label{id: JIi xxx=xx} \\
x^2y^2 & \approx y^2x^2, \label{id: JIi xxyy=yyxx} \\
xyx & \approx yx^2. \label{id: JIi xyx=yxx}
\end{align}
\end{subequations}
constitute an {\ib} for the variety $\var\{\JIi\}$\up.
\item The variety $\var\{\JIi\}$ contains countably infinitely many subvarieties\up.
\end{enumerate}
\end{proposition}

\begin{proof}
(i) See Edmunds~\cite[Semigroup $\mathtt{S(4,23)}$ on page~70]{Edm80}.

(ii) Any variety that satisfies the identity~\eqref{id: JIi xyx=yxx} is finitely based~\cite{Pol82}.
Hence by Lemma~\ref{L: HFB implies countable}, the variety $\var\{\JIi\}$ contains countably many subvarieties.
The result then holds since the subvariety $\var\{\Niltwo^1\}$ of $\var\{\JIi\}$ contains infinitely many subvarieties \cite[Figure~5(b)]{Eva71}.
\end{proof}

\begin{lemma} \label{L: JIi max}
Each proper subvariety of $\var\{\JIi\}$ satisfies the identity
\begin{equation}
x^2ya^2 \approx yx^2 a^2. \label{id: JIi max}
\end{equation}
\end{lemma}

\begin{proof}
Let~$\bW$ be any proper subvariety of $\var\{\JIi\}$, so that $\JIi \notin \bW$.
Then it follows from Almeida~\cite[Proposition~11.7.9]{Alm94} that~$\bW$ satisfies either~\eqref{id: JIi max} or $\alpha: x^2y^2 \approx xy^2$.
Since
\[
x^2ya^2 \stackrel{\alpha}{\approx} x^2y^2a^2 \stackrel{\eqref{id: JIi xxyy=yyxx}}{\approx} y^2x^2a^2 \stackrel{\alpha}{\approx} yx^2a^2,
\]
the variety~$\bW$ always satisfies the identity~\eqref{id: JIi max}.
\end{proof}

\begin{proposition}
Let $\bU = \var\{ \eqref{id: JIi},\eqref{id: JIi max} \}$\up.
Then
\begin{enumerate}[\rm(i)]
\item $\bU$ is the unique maximal subvariety of $\bV_{\ref{V: JIi}} = \var\{\JIi\}$\up;
\item $\bU$ is not finitely generated\up.
\end{enumerate}
\end{proposition}

\begin{proof}
(i) This follows from Proposition~\ref{P: JIi}(i) and Lemma~\ref{L: JIi max}.

(ii) It is easily seen that the variety~$\bU$ violates the identity
\begin{equation}
x^2yz_1 z_2 \cdots z_m \approx yx^2z_1 z_2 \cdots z_m \label{id: JIi finite}
\end{equation}
for any $m \geq 1$.
Hence it suffices to show that each finite semigroup~$S$ in~$\bU$ satisfies the identity~\eqref{id: JIi finite} for all $m > |S|$.
Choose any elements $a,b,c_1,c_2,\ldots,c_m \in S$.
The list $c_1,c_2,\ldots,c_m$ contains some element $c \in S$ twice, due to the magnitude of~$m$.
Therefore $c_1 c_2 \cdots c_m = s_1cs_2cs_3$ for some $s_i \in S^1$, whence
\begin{align*}
a^2bc_1 c_2 \cdots c_m = a^2bs_1cs_2cs_3 & \stackrel{\eqref{id: JIi}}{=} a^2bc^2s_1cs_2cs_3 \stackrel{\eqref{id: JIi max}}{=} ba^2c^2s_1cs_2cs_3 \\
& \stackrel{\eqref{id: JIi}}{=} ba^2s_1cs_2cs_3 = ba^2c_1 c_2 \cdots c_m. \qedhere
\end{align*}
\end{proof}

\begin{corollary} \quad
\begin{enumerate}[\rm(i)]
\item The identities
\[
x^3 \approx x^2, \quad x^2y^2 \approx y^2x^2, \quad xyx \approx x^2y.
\]
constitute an {\ib} for the variety $\bV_{\ref{V: JIid}} = \var\{\JIidual\}$\up.
\item The variety $\var\{\JIidual\}$ contains countably infinitely many subvarieties\up.
\item The variety $\var\{\JIidual\}$ contains a unique maximal subvariety\up.
\end{enumerate}
\end{corollary}



\end{document}